\RequirePackage{fix-cm}
\documentclass{svjour3}
\usepackage{graphics,graphicx,epsf,subfigure,epstopdf}
\usepackage{amsfonts,amsmath,amsxtra,amscd,amssymb,bm, amsthm}
\usepackage{mathtools}
\usepackage[title]{appendix}
\usepackage[numbers]{natbib}
\usepackage{lmodern}
\usepackage{algorithm}
\usepackage{algpseudocode}
\usepackage[top=2.5cm,bottom=2.5cm,right=2.5cm,left=2.5cm]{geometry}
\usepackage[mathscr]{eucal}
\usepackage{color}
\numberwithin{equation}{section}

\usepackage{hyperref}

\DeclareMathOperator*{\argmin}{arg\,min}

\def\eqref#1{(\ref{#1})}

\newcommand{\tsum}{\textstyle\sum}
\newcommand{\tint}{\textstyle\int}
\newcommand{\bbe}{\mathbb{E}}

\newcommand{\beq}{\begin{equation}}
\newcommand{\eeq}{\end{equation}}
\newcommand{\beqa}{\begin{eqnarray}}
\newcommand{\eeqa}{\end{eqnarray}}
\newcommand{\beqas}{\begin{eqnarray*}}
\newcommand{\eeqas}{\end{eqnarray*}}

\usepackage{enumitem}

\newtheorem{assumption}{Assumption}
\newcommand{\gap}{\mathrm{gap}}
\newcommand{\res}{\mathrm{res}}

\newcommand{\bbr}{\mathbb{R}}

\def\cS{{\cal S}}
\def\cA{{\cal A}}

\newcommand{\vbar}{\ensuremath{\bar v}}

\def\brho{{\boldsymbol{\nu}}}
\def\bPhi{{\bar{\Phi}}}
\def\bphi{{\bar{\phi}}}

\def\myfinal{}

\def\vgap{\vspace*{.1in}}

\newcommand{\norm}[1]{\left\lVert#1\right\rVert}

\usepackage[dvipsnames]{xcolor}

\title{Computation of Strong Solutions to Stochastic Variational Inequalities}

\author{Yao Ji \and Guanghui Lan   \and Jason  Zhu  \thanks{GL, YJ and JZ were partially supported by Air Force Office of Scientific Research grant FA9550-22-1-
0447 and American Heart Association grant 23CSA1052735.} \thanks{The results of this work were presented at the Modeling and Optimization: Theory and Applications (MOPTA) Conference, Lehigh University, Bethlehem, PA, August 18--19, 2026.}}

\institute{
    Y. Ji \and G. Lan \and J. Zhu \at
    H. Milton Stewart School of Industrial and Systems
  Engineering,
    Georgia Institute of Technology,
    225 North Ave, Atlanta, GA 30332, USA \\
    \email{yaoji@gatech.edu, george.lan@isye.gatech.edu, jzhu657@gatech.edu}
  }

\begin{document}

\maketitle

\begin{abstract}
This paper studies the computation of strong solutions of monotone variational inequalities (VIs) with Lipschitz continuous operators. Building on the idea of accumulative regularization (AR), we develop a general framework for VIs, with particular emphasis on stochastic settings. Under unbiased stochastic oracles with uniformly bounded variance $ \sigma^2$, AR computes an approximate solution with expected operator residual bounded by $\varepsilon$ using at most
$$
\widetilde{\mathcal{O}}\left(\tfrac{LD_0}{\varepsilon}+\tfrac{ \sigma^2}{\varepsilon^2}(\log\tfrac{LD_0}{\varepsilon})^3\right)
$$
stochastic oracle calls, where $L$ is the Lipschitz constant and $D_0$ bounds the initial distance to the solution. 
This substantially improves the existing $\mathcal{O}( \sigma^2/\varepsilon^4)$ complexity for residual reduction and matches the lower bound up to logarithmic factors. For strongly monotone VIs, measured by the distance to the solution, AR achieves the optimal oracle complexity when the strong monotonicity modulus is known. By treating the problem as merely monotone, AR still achieves nearly optimal complexity without knowledge of this modulus. We further introduce a state-dependent noise model applicable to general monotone VIs with potentially nonunique solutions, extending state-dependent noise analysis beyond the strongly monotone setting. Under this model, AR, when equipped with an enhanced stochastic operator extrapolation (SOE) method, achieves nearly optimal complexity with the stochastic term depending on the variance at a solution.
We apply these results to policy evaluation in reinforcement learning, where we show that estimators of the projected Bellman operator satisfy our state-dependent noise condition. This yields improved rates for fast temporal-difference (FTD) learning, including optimal Bellman-residual rates in the long-horizon regime without dependence on the deteriorating strong-monotonicity parameter.
To the best of our knowledge, the AR framework, the improved complexity under uniformly bounded noise, the state-dependent noise model and its guarantees, the adaptivity to an unknown strong monotonicity modulus, and the improved guarantees for FTD learning are all new in the VI literature.

\end{abstract}

\section{Introduction} \label{sec_intro}

We consider the general variational inequality (VI) problem of finding a point
\begin{equation}
\label{VIP}
x^* \in X \quad \text{such that} \quad \langle F(x^*),\, x - x^* \rangle \ge 0, \qquad \forall\, x \in X,
\end{equation}
where $X \subseteq \mathbb{R}^n$ is a nonempty closed convex set and $F: X \to \mathbb{R}^n$ is
an $L$-Lipschitz continuous map, i.e., there exists $L>0$ such that
\begin{equation}
\label{eq:Lipschitz}
\|F(x)-F(z)\| \le L\|x-z\|, \qquad \forall\, x,z\in X.
\end{equation}
VIs satisfying \eqref{eq:Lipschitz} are often referred to as \emph{Lipschitz continuous} and points $x^*$ satisfying $\eqref{VIP}$ are referred to as \emph{strong solutions}.

Throughout this paper, we assume that the solution set $X^*$ satisfying \eqref{VIP} is nonempty. Moreover, we assume that $F$ satisfies the (strong) monotonicity condition
\begin{equation}
\label{strongly_monotone}
\langle F(x)-F(z),\, x-z\rangle \ge \mu\|x-z\|^2, \qquad \forall\, x,z\in X,
\end{equation}
for some $\mu\ge 0$. When $\mu>0$, $F$ is said to be \emph{strongly monotone}; when $\mu=0$, $F$ is \emph{monotone}.

In order to analyze the complexity of VI algorithms, we first need to discuss a few different termination criteria
for the VI problem in \eqref{VIP} to measure the accuracy of a given candidate solution $\bar x \in X$.
If $F$ satisfies the strong monotonicity condition in \eqref{strongly_monotone} for some $\mu> 0$,  then the solution $x^*$ is unique, and
the distance to the optimal solution $\|\bar x - x^*\|$ is a natural measure of solution accuracy. For monotone VIs, the solution need not be unique, and there exist at least three other widely used termination criteria. One is called the weak gap, defined as
\beq \label{def_weak_gap}
\gap_W(\bar{x}) \coloneqq \max_{x \in X} \langle F(x),  \bar{x} -x \rangle
\eeq
for a given $\bar x \in X$.
A stronger notion of gap function is given by
\beq \label{def_strong_gap}
\gap_S(\bar x) \coloneqq \max_{x \in X} \langle F(\bar x),  \bar{x} -x \rangle.
\eeq
Both of these gap functions work only in the case when $X$ is bounded, but their extensions to the unbounded case  have been investigated by Monteiro and Svaiter in \cite{MonSva10-1}. Note that under monotonicity condition, $\gap_S(\bar x) \le \varepsilon$ implies $\gap_W(\bar x) \le \varepsilon$, but not vice versa.
Also, under the strong monotonicity assumption, the condition $\gap_S(\bar x) \le \varepsilon$ implies $\mu\|\bar x -      x^*\|^2 \le \varepsilon$.

In this paper, we adopt a strong notion called the \textit{operator residual}, which applies to both the bounded and unbounded settings, for solving the general VI problem \eqref{VIP}.
Let us denote the normal cone of $X$ at $\bar x$ by
\beq \label{def_N_X}
N_X(\bar x) \coloneqq \{y \in \bbr^n \,|\, \langle y, x - \bar x\rangle \le 0, \forall\, x \in X \}.
\eeq
Noting that $\bar x \in X$ is an optimal solution for problem~\eqref{VIP} if and only if $F(\bar x) \in -N_X(\bar x)$.
We define the operator residual of $\bar x$ as
\beq \label{def_res}
\res_F(\bar x) \coloneqq \min_{y \in -N_X(\bar x)} \|y - F(\bar x)\|.
\eeq
In particular, if $X = \bbr^n$, then $N_X(\bar x) = \{0\}$ and $\res_F(\bar x) = \|F(\bar x)\|$, which is exactly
 the residual of solving the nonlinear equation $F(\bar x) = 0$.
 It can also be easily verified that in the bounded case,
 \beq \label{eq:Residual_stronger}
\res_F(\bar x) \le \varepsilon  \Longrightarrow \gap_W(\bar x) \le \gap_S(\bar x) \le D_X \varepsilon,
\eeq
where $D_X = \max_{x,z \in X} \|x -z\|$, and thus this is a stronger notion than
both the weak and strong gaps defined in \eqref{def_weak_gap} and \eqref{def_strong_gap},
respectively. Note that this criterion is also stronger than the modified two-parameter strong gap function studied by~\cite{MonSva10-1}. For strongly monotone problems, one may use both the operator residual $\res_F(\bar x)$ and the distance $\|\bar x - x^*\|$; in particular, a small residual $\res_F(\bar x)\leq \varepsilon$ implies a small distance $\|\bar x - x^*\|\leq \varepsilon/\mu$.
These observations motivate our choice of the operator residual as the solution criterion: it directly quantifies the violation of the optimality condition $F(\bar x)\in -N_X(\bar x)$ and remains well defined even when $X$ is unbounded. Consequently, a convergence guarantee in terms of the operator residual certifies approximate optimality and, in view of the implications above, automatically yields guarantees for the weaker solution criteria whenever they are applicable. 

The complexity of solving VIs has been well
studied, especially in the deterministic setting.
Nemirovski, in his seminal work \cite{Nem05-1}, introduced the mirror-prox method by modifying the extragradient scheme \cite{korpelevich1983extrapolation} and established an $\mathcal{O}(1/\varepsilon)$ iteration complexity in terms of the weak gap in \eqref{def_weak_gap} for smooth monotone VIs. This rate is optimal for first-order VI methods, as demonstrated by matching lower complexity bounds \cite{nemyud:83,Nem94,ouyang2021lower}. The mirror-prox framework subsequently inspired extensive developments  \cite{MonSva10-1,dang2015convergence,Malisky15,NJLS09-1,juditsky2011solving,CheLanOu14-1,yousefian2017smoothing,iusem2017extragradient,Shanbhag2019}.
In terms of the residual defined in \eqref{def_res}, Monteiro et al. \cite{MonSva10-1} showed that the hybrid proximal extragradient framework achieves an $\mathcal{O}(\varepsilon^{-2})$ complexity bound for monotone VIs with Lipschitz-continuous operators. Early uses of fixed regularization in convex programming can be found in Lan and  Monteiro 's works \cite{lan2013iteration,lan2016iteration}.
Kotsalis et al. \cite{kotsalis2022simple} showed that a simple operator extrapolation method attains the same complexity for the broader class of generalized monotone VIs with Lipschitz-continuous operators. For definitions of generalized monotonicity, see, e.g., \cite{facchinei_finite-dimensional_2004,dang2015convergence}. For monotone Lipschitz operators, Alves et al. \cite{alves2016regularized} showed that a regularized method achieves the near-optimal iteration complexity of $\mathcal{O}(\varepsilon^{-1}\log(1/\varepsilon))$; see also the related earlier works on convex optimization by Lan and Monteiro \cite{lan2013iteration,lan2016iteration}. This technique was later studied in the form of Halpern-type iterations \cite{diakonikolas2020halpern}. Under the additional assumption of cocoercivity, Halpern-type methods attain the optimal $\mathcal{O}(\varepsilon^{-1})$ complexity \cite{diakonikolas2020halpern,DiakonikolasWang2022Potential}. For unconstrained monotone Lipschitz equations, \cite{ioan2024extra} establishes an $\mathcal{O}(1/k)$ residual decay, where $k$ denotes the iteration count; this rate, however, relies on an asymptotic argument. Explicit nonasymptotic $\mathcal{O}(1/\varepsilon)$ residual bounds have been established for extragradient methods for unconstrained monotone VIs with Lipschitz-continuous operators \cite{yoon2021accelerated}.
There are a few works \cite{tran2024extragradient,cai2022accelerated} that establish an $\mathcal{O}(1/\varepsilon)$ complexity bound using fixed regularization methods.

The complexity of solving stochastic VIs depends critically on the noise model. 
A standard assumption is that we have access to a conditionally unbiased estimator (or stochastic oracle) of $F$~\cite{juditsky2011solving,kotsalis2022simple}. Specifically, for every $\mathcal{F}_k$-measurable $x\in X$,
\begin{align}\label{eq:fast-decreasing-bias-o}
\mathbb{E}[\tilde{F}(x,\zeta_k)\mid\mathcal{F}_k]
=F(x), \quad \textnormal{a.s.}
\end{align}
It should be noted that most of our developments can be extended straightforwardly to estimators with vanishing bias. For simplicity, we focus on the unbiased case in this paper.
In addition to \eqref{eq:fast-decreasing-bias-o}, we often need to impose different assumptions on the variance of the estimator to derive stochastic oracle complexity bounds.

Under the classical \emph{uniformly bounded variance} assumption, i.e., for every $\mathcal{F}_k$-measurable $x\in X$,
\begin{equation}\label{eq:intro-uniform-assumption}
\bbe\!\left[\|\tilde{F}(x,\zeta_k)-F(x)\|^2\mid\mathcal{F}_k\right]\leq  \sigma^2, \qquad \forall\,\, x\in X.
\end{equation}
there has been extensive work on stochastic VIs using the expected weak gap as the termination criterion. Juditsky et al.~\cite{juditsky2011solving} showed that stochastic mirror-prox attains the optimal $\mathcal{O}(1/\varepsilon^2)$ stochastic-oracle complexity under this criterion. More recently, Kotsalis et al.~\cite{kotsalis2022simple} showed that the simpler stochastic operator extrapolation (SOE) method not only achieves the same optimal weak-gap complexity for monotone VIs, but also attains the optimal
$\mathcal{O}((L/\mu)\log(1/\varepsilon)+ \sigma^2/(\mu^2\varepsilon))$
stochastic-oracle complexity under the mean-squared-distance criterion for strongly monotone VIs.
In contrast, the existing literature remains considerably less satisfactory in terms of operator-residual guarantees in the stochastic setting. For stochastic generalized monotone VIs, under a standard i.i.d.\ single-point stochastic oracle with uniformly bounded variance, SOE achieves a complexity of ${\cal O}(1/\varepsilon^4)$ for producing a point $x$ satisfying
$\mathbb{E}[\operatorname{res}(x)]\leq\varepsilon$
\cite{kotsalis2022simple}. For unconstrained Lipschitz-monotone problems, stochastic Halpern-type methods improve this complexity to ${\cal O}(1/\varepsilon^3)$ \cite{cai2022stochastic}. However, this latter result relies on a multi-point oracle that permits the same random sample to be reused at different query points, together with mean-square Lipschitz continuity of the stochastic operator, an assumption strictly stronger than Lipschitz continuity of the expected operator. Moreover, the constrained extension in \cite{cai2022stochastic} requires cocoercivity. Consequently, under the standard single-point oracle with uniformly bounded variance, the best-known general guarantee for stochastic Lipschitz-monotone VIs remains ${\cal O}(1/\varepsilon^4)$. This is substantially worse than the optimal ${\cal O}(1/\varepsilon^2)$ stochastic-oracle complexity under the weak-gap criterion. For the more restrictive class of unconstrained minimax problems, whose optimality conditions take the form of monotone equations rather than general VIs, the near optimal ${\mathcal O}(\sigma^2(\log 1/\varepsilon)^3/\varepsilon^2)$ dependence has been established in \cite{chen2024near}.

A complementary lower bound is provided by stochastic convex optimization, where an
$\Omega(\varepsilon^{-2}\log(1/\varepsilon))$ lower bound is established for stochastic gradient-norm
minimization under a local stochastic first-order oracle
\cite{foster2019complexity}. This problem is a special case of residual
minimization for monotone VIs: taking $X=\mathbb{R}^n$ and
$F=\nabla f$ for a smooth convex function $f$ gives
$\mathbb E[\operatorname{res}(x)]=\mathbb E[\|\nabla f(x)\|]$. Their result therefore yields a
$\Omega(\varepsilon^{-2}\log(1/\varepsilon))$ lower bound for the broader VI
class.
This naturally raises the question of whether a
$\widetilde{\mathcal{O}}(1/\varepsilon^2)$ stochastic-oracle complexity,
matching the lower bound in its polynomial dependence on $\varepsilon$ for residual reduction,
can be attained for stochastic monotone VIs with Lipschitz-continuous
operators.

Addressing this question under a more broadly applicable noise model presents an additional challenge. In many stochastic VIs, the uniform variance bound in \eqref{eq:intro-uniform-assumption} fails to hold \cite{iusem2017extragradient}, as the variance of the stochastic oracle may depend on the query point $x$ \cite{juditsky2023sparse}. This has motivated a growing body of work on weaker, \emph{state-dependent noise} assumptions that relax \eqref{eq:intro-uniform-assumption} \cite{iusem2017extragradient,kotsalis2022simple2,alacaoglu2025towards,alacaoglu2026solving}. Such assumptions typically take the form \cite{kotsalis2022simple2}
\begin{equation}\label{state-dependent-original}
\bbe\!\left[  \|\tilde F(x,\zeta_k)-F(x)\|^2\mid \mathcal{F}_k\right]\leq \sigma_*^2+\varsigma\|x-x^*\|^2,\qquad \forall x\in X,
\end{equation}
where $\sigma_*^2$ characterizes the variance at the solution $x^*$, and $\varsigma$ is the state-dependent variance parameter.
This distance-based relaxation is natural in the strongly monotone case, where the solution $x^*$ is unique. For merely monotone VIs, however, the solution need not be unique, making the distance to a particular solution less intrinsic to the problem.
In the special case of convex optimization, Ilandarideva et al. \cite{ilandarideva2025accelerated} considers state-dependent noise characterized by the function-value optimality gap and derives optimal convergence rates under that state-dependent noise assumption. For a general SVI, however, there is no corresponding notion of a function-value optimality gap, and hence this assumption is not directly applicable.

In this paper, we introduce a natural state-dependent noise model not only for strongly monotone VIs, but also for general monotone problems. Specifically,
we assume that for any $x^*\in X^*$, there exist constants
$\sigma_*^2 \equiv \sigma_*^2(x^*) \geq0$ and $\varsigma_* \equiv \varsigma_*(x^*) \geq0$ such that, for every $\mathcal{F}_k$-measurable $x\in X$,
\begin{align}\label{eqn:state-dependent-assumption}
    \mathbb{E}\!\left[
        \|\tilde F(x,\zeta_k)-F(x)\|^2\mid \mathcal{F}_k
    \right]
    \leq \sigma_*^2
    +\varsigma_* \langle F(x)-F(x^*),x-x^*\rangle.
\end{align}
Unlike the distance-based condition in
\eqref{state-dependent-original}, the state-dependent term in
\eqref{eqn:state-dependent-assumption} vanishes between any two solutions of a monotone VI
 when the solution is not unique.
Indeed,
if both $\bar x$ and $x^*$ solve
the VI ~\eqref{VIP}, their optimality conditions together with
monotonicity of $F$ imply
$\langle F(\bar x)-F(x^*),\bar x-x^*\rangle=0$. More generally, if
$\bar x$ is an $\varepsilon$-approximate strong gap solution,
then
$$0\leq\langle F(\bar x)-F(x^*),\bar x-x^*\rangle\leq\varepsilon.$$
Thus, \eqref{eqn:state-dependent-assumption} controls stochastic variation through an intrinsic VI optimality measure that itself is controlled by
the strong gap, rather than through the Euclidean distance to an arbitrarily
selected solution. 
Moreover, $\sigma_*^2$ still provides an upper bound on the variance on $x^*$.

In addition to important applications such as sparse regression \cite{ilandarideva2025accelerated}, policy evaluation in reinforcement learning provides a natural instance of this framework. For a fixed policy $\pi$, its value function $V^\pi$ is characterized as a fixed point of the Bellman operator and, under linear function approximation, can be formulated as a stochastic
VI~\cite{kotsalis2022simple2, li2023accelerated}. As the discount factor $\gamma$ approaches
one, the strong-monotonicity parameter $\mu$ associated with the projected Bellman
operator typically deteriorates, and the problem increasingly resembles a
merely monotone VI. Distance-based guarantees may consequently become poorly
conditioned, whereas the VI residual remains meaningful because it directly
measures the projected Bellman error without requiring a favorable
strong-monotonicity parameter. We show that a simple stochastic estimator of
the projected Bellman operator satisfies the state-dependent noise \eqref{eqn:state-dependent-assumption}, thereby
placing policy evaluation within our proposed noise model and allowing our residual-based convergence guarantees to apply directly. Interestingly, for strongly monotone problems, a direct application of strong monotonicity suggests that \eqref{eqn:state-dependent-assumption} holds with $\varsigma_* = \varsigma/\mu$ under \eqref{state-dependent-original}. However, for policy evaluation, we show that \eqref{eqn:state-dependent-assumption} in fact holds with $\varsigma_* \approx \varsigma$, independent of $\mu$. 
This highlights an important distinction between the two state-dependent noise assumptions.

The state-dependent noise model \eqref{eqn:state-dependent-assumption}, however, introduces a new challenge in the design and analysis of stochastic VI methods. To control the effect of the state-dependent variance term
 $\varsigma_*\langle F(x)-F(x^*),x-x^*\rangle$,
 the algorithm must simultaneously control this intrinsic optimality measure, which can in turn be related to the operator residual.
 Consequently, obtaining sharp residual complexity bounds is essential for handling the state-dependent noise: without sufficiently tight residual guarantees, the additional variance induced by the state-dependent term cannot be effectively controlled.
 Developing stochastic VI methods that exploit this relationship while attaining sharp residual complexity bounds requires new algorithmic and analytical techniques. To the best of our knowledge, this problem has not been studied previously, even under strong monotonicity.

\subsection{Contributions and organization}

This paper attempts to address the aforementioned challenges in stochastic VIs. Our main contributions are summarized as follows.

First, inspired by the accumulative regularization technique for convex optimization~\cite{lan_optimal_2024,ji2025high}, we develop a general AR framework for VIs that allows a broad class of algorithms to serve as subroutines for solving a sequence of regularized VI subproblems, and converts their convergence guarantees into operator-residual guarantees for the original problem. In the deterministic setting, we show that AR computes a point $x_S$ satisfying
$\res_F(x_S)\leq\varepsilon$ after
$\mathcal{O}\left({LD_0}/{\varepsilon}\right)$ evaluations of $F$
(cf.~\autoref{the:sublinear_AR}), where
$D_0\geq\|x_0-x^*\|$
is an upper bound on the distance between the initial point $x_0$ and the solution $x^*$. 
   This achieves the optimal dependence on the target accuracy. Notice that, unlike fixed regularization methods in the existing literature, AR is not tailored to a particular algorithm but instead provides a general framework. It requires only that its subroutine attain a relatively slow residual rate and a suboptimal convergence rate in distance for the regularized subproblems. Consequently, it can transfer many existing VI algorithms, such as OE and extragradient type methods, with optimal residual guarantees.
Unlike in convex optimization, where function values can serve as a proxy for monitoring algorithmic progress, no such proxy is generally available for VIs. Consequently, the analysis of our AR framework differs significantly from that developed for convex optimization in \cite{lan_optimal_2024,ji2025high}.
We further equip AR with stochastic subroutines under the classic
uniformly bounded variance assumptions. In this setting, AR computes a point $x_S$ satisfying $\sqrt{\bbe[\res_F(x_S)^2]}\leq\varepsilon$ after
\begin{equation}\label{eqn:major-1}
\widetilde{\mathcal{O}}\left(
    \tfrac{LD_0}{\varepsilon}
    +
    \tfrac{\sigma^2}{\varepsilon^2}
    \left(
        \log\tfrac{LD_0}{\varepsilon}
    \right)^3
\right)
\end{equation}
calls to the stochastic oracle $(\mathcal{SO})$ (cf. \autoref{thm:main-unbiased}), where $\widetilde{\mathcal{O}}$ hides factors polynomial in $\log\log(\tfrac{LD_0}{\varepsilon})$.
By Jensen's inequality, this guarantee implies $\bbe[\res_F(x_S)]\leq\varepsilon$.
The stochastic term exhibits the nearly optimal dependence on the target accuracy, up to logarithmic factors. 
By comparison, existing results require a sample complexity of order $\mathcal{O}(\sigma^2/\varepsilon^4)$ to guarantee $\bbe[\res_F(x_S)]\leq\varepsilon$.

Second, we extend the AR framework to stochastic monotone VIs under the
state-dependent noise assumption \eqref{eqn:state-dependent-assumption}. We show that AR computes a point $x_S$ satisfying
$\sqrt{\bbe[\res_F(x_S)^2]}\leq\varepsilon$ using
\begin{equation}\label{eqn:major-2}
\mathcal{\widetilde O}\left(
    \tfrac{LD_0}{\varepsilon}
    +
    \tfrac{\varsigma_* D_0}{\varepsilon}
    \left(
        \log\tfrac{LD_0}{\varepsilon}
    \right)^3
    +
    \tfrac{\sigma_*^2}{\varepsilon^2}
    \left(
        \log\tfrac{LD_0}{\varepsilon}
    \right)^3
\right)
\end{equation}
calls to the $\mathcal{SO}$
(cf.~\autoref{thm:ARVI-state-dependent}). Here, $\varsigma_*$
characterizes the state-dependent component of the noise, while
$\sigma_*^2$ denotes an upper bound on the noise variance at the solution.
This improves upon the
$\widetilde{\mathcal{O}}(\varepsilon^{-4})$ sample-complexity bound
established in \cite{kotsalis2022simple2} for finding a point with a slightly weaker guarantee $\bbe[\res_F(x_S)]\leq\varepsilon$  under the distance-based state-dependent noise
assumption \eqref{state-dependent-original}. Moreover, under the proposed
assumption, the state-dependent contribution has a more favorable dependence
on the target accuracy and the noise parameter. Specifically, the
state-dependent component of the existing bound scales as
$\widetilde{\mathcal{O}}({\varsigma}^4\varepsilon^{-4})$
\citep[Theorem~3.23]{kotsalis2022simple2}, whereas the corresponding component
in our bound scales as
$\widetilde{\mathcal{O}}(\varsigma_*\varepsilon^{-1})$.

Third, to achieve the complexity bound in \eqref{eqn:major-2}, we derive an improved convergence guarantee for the stochastic operator extrapolation (SOE) method \cite{kotsalis2022simple2} applied to strongly monotone VIs under the proposed state-dependent noise assumption. For a $\mu$-strongly monotone and $L$-Lipschitz continuous operator, our SOE method computes a point $\widehat{x}$ satisfying
$\bbe[\|\widehat{x}-x^*\|^2]\leq\varepsilon$ after
\begin{equation}\label{eqn:major-3}
\mathcal{O}\left(
    \tfrac{L+\varsigma_*}{\mu}
    \log\tfrac{\|x_0-x^*\|^2}{\varepsilon}
    +
    \tfrac{\sigma_*^2}{\mu^2\varepsilon}
\right)
\end{equation}
calls to the $\mathcal{SO}$; see
\autoref{state-dependent-thm-strongly-monotone}.
For comparison, \citep[Corollary~3.6]{kotsalis2022simple2} establishes the
sample complexity
\begin{equation*}
\mathcal{O}\left(
    \max\left\{
        \tfrac{\max\{L,\tilde L\}^2+\varsigma^2}{\mu^2}
        \log\tfrac{\|x_0-x^*\|}{\varepsilon},
        \tfrac{\sigma_*^2+\|F(x^*)\|^2}{\mu^2\varepsilon}
        \log\tfrac{1}{\varepsilon}
    \right\}
\right)
\end{equation*}
for achieving the same distance guarantee. This bound depends quadratically
on the condition number and involves the potentially larger sample-wise
Lipschitz constant $\tilde L$. Moreover, its stochastic term also contains an additional
logarithmic factor and depends on $\|F(x^*)\|^2$, which does not diminish
under variance reduction.
In contrast, \eqref{eqn:major-3} exhibits the optimal linear dependence on $L/\mu$ in the deterministic term and the optimal
$\mathcal{O}(\sigma_*^2/(\mu^2\varepsilon))$ stochastic term. Because
$\sigma_*^2$ bounds the oracle variance at a solution $x^*$, it can be
 substantially reduced through variance-reduction mechanisms, such as parallel or distributed implementations that aggregate multiple stochastic operator evaluations, mini-batching, and other averaging strategies. Such mechanisms reduce the effective noise level without increasing the number of algorithmic iterations.

Fourth, we extend the AR framework to obtain high-probability guarantees. We establish the sample complexity of AR under the uniform and state-dependent noise models for both monotone and strongly monotone VIs. Specifically, given any $p \in (0,1)$, we estimate the number of calls to $\mathcal{SO}$ for AR to compute a point $x_S$ satisfying $\res_F(x_S)\leq\varepsilon$ with probability at least $1-p$. 
To establish these high-probability sample-complexity guarantees, we also derive convergence rates for SOE applied to strongly monotone VIs under the new state dependent noise model. Under the uniform noise model, we establish slightly sharper convergence bounds for both AR and SOE.

Finally, we apply the proposed methods to online policy evaluation. 
We consider accumulatively regularized fast temporal-difference learning (AR-FTD), obtained by applying the proposed AR framework to the projected Bellman equation. We show that an unbiased minibatch stochastic estimator under a generative model naturally satisfies the new state-dependent noise assumption \eqref{eqn:state-dependent-assumption}. Consequently, with problem-dependent constants omitted for simplicity, AR-FTD computes a point whose value-function mean-squared error is at most $\varepsilon$ using
\begin{equation*}
\widetilde{\mathcal{O}}\left(\tfrac{1}{1-\gamma}\log\tfrac{1}{\varepsilon}+\tfrac{\sigma_*^2}{(1-\gamma)^2\varepsilon}\right)
\end{equation*}
samples (cf. \autoref{the:mult_epoch_FTD_burn}). Thus, AR-FTD achieves optimal deterministic and stochastic complexity. We further show that AR-FTD computes a point $\theta_S$ with a small Bellman residual, namely, $\bbe[\|F(\theta_S)\|]\leq\varepsilon$, using $\widetilde{\mathcal{O}}\left(\tfrac{\sigma_*^2}{\varepsilon^2}(\log\tfrac{1}{\varepsilon})^3\right)$ samples (cf. \autoref{prop:ARVI-state-dependent}). Moreover, whenever the projected Bellman operator is strongly monotone (i.e., discount factor $\gamma$ small), this residual guarantee can be translated into a distance-to-solution guarantee of nearly the same order as that in \autoref{the:mult_epoch_FTD_burn}. By contrast, to find a point $\theta_S$ with a small Bellman residual, AR-FTD does not require strong monotonicity of the projected Bellman operator and can therefore be implemented without knowledge of a strong monotonicity parameter.

\smallskip

The above result also reveals an interesting adaptivity property of the AR
framework for monotone VIs. For any prescribed residual accuracy, AR can be run without assuming strong monotonicity or knowing
the strong monotonicity parameter $\mu$. This is valuable because methods
whose parameters are tuned from an estimate of $\mu$ can converge extremely
slowly when $\mu$ is overestimated; see the well-known example in
\cite[page~1578]{NJLS09-1}. Existing parameter-adaptive procedures can avoid prior knowledge of $\mu$ when the total iteration budget is fixed and a lower bound on $\mu$ is available, but generally incur an additional logarithmic factor and do not directly accommodate a prescribed target accuracy when such a lower bound is unavailable \cite{juditsky2014deterministic}. Interestingly, we show that when the VI is in fact $\mu$-strongly monotone and the stochastic error dominates, one can simply ignore strong monotonicity, apply AR as a general monotone VI method, and then translate its residual guarantee into a distance-to-solution guarantee. Here, we consider the uniform noise model. A similar result can be derived under the state-dependent noise model. Specifically, for any $\varepsilon>0$, AR computes $x_S$ such that
$\mathbb{E}[\norm{x_S-x^*}^2]\leq\varepsilon$ using
\begin{equation*}
\widetilde{\mathcal{O}}\left(\tfrac{LD_0}{\mu\sqrt{\varepsilon}}+\tfrac{\sigma^2}{\mu^2\varepsilon}\left(\log\tfrac{LD_0}{\mu\sqrt{\varepsilon}}\right)^3\right)
\end{equation*}
samples (cf.\ \autoref{prop:mu-free}). The resulting complexity is nearly optimal: up
to logarithmic factors, it matches the dominant stochastic term achieved by
SOE \cite{kotsalis2022simple} for strongly monotone operators, without requiring prior knowledge of $\mu$. 
\vgap

This paper is organized as follows. Section 2 develops the AR framework for finding strong solutions of monotone VIs~\eqref{VIP} satisfying \eqref{eq:Lipschitz} and \eqref{strongly_monotone}. The framework applies to both deterministic problems and stochastic problems under the uniformly bounded variance assumption. Section 3 introduces the new state-dependent noise assumption, extends the AR framework to this setting, and analyzes the SOE method under the proposed assumption. Section 4 develops high-probability guarantees under different noise assumptions. Section 5 presents new results for policy evaluation based on the AR framework and the proposed state-dependent noise assumption. Finally, Section 6 provides some brief concluding remarks.

\subsection{Notation and terminology} \label{sec:notations}
Throughout this paper, we use the symbol $\mathbb{R}$, $\mathbb Z$, and $\mathbb N$ to denote the real numbers, integers, and natural numbers respectively. We correspondingly use $\mathbb{R}_+$ to denote the nonnegative real numbers, $\mathbb{Z}_+$ to denote the nonnegative integers, and $\mathbb{N}_+$ to denote the positive natural numbers respectively. We use $\mathbb{R}^n$ and $\mathbb{Z}^n$ to denote the sets of $n$-tuples of real numbers and integers, respectively. We denote the standard inner product on $\mathbb{R}^n$ by $\langle a,b\rangle = a^{\top}b$ for $a, b \in \mathbb{R}^n$. Additionally, we denote $\|\cdot\|$ as the Euclidean norm, i.e., $\|a\| = \sqrt{\langle a, a\rangle}$ for $a \in \mathbb{R}^n$. For any real number $s,$ $\lceil s\rceil$ and $\lfloor s\rfloor$ denote the nearest integers to $s$ from above and below, respectively.
Let $[m]\triangleq\{1,\dots,m\}$, with $m\in \mathbb{N}_{+}.$
{Let $\zeta_1,\dots,\zeta_{k-1}$ be i.i.d. random variables on $(\Omega,\mathcal{F})$. Set $\Omega_{k}\coloneqq \prod_{i=0}^{k-1}\Omega$ and define
\[
\mathcal{F}_{k}\coloneqq \sigma\!\left(\left\{A_0\times\cdots\times A_{k-1}:\; A_i\in\mathcal{F},\ i=0,\dots,k-1\right\}\right).
\]
We denote by $\mathbb{P}\coloneqq \prod_{i=0}^{k-1} \mu$ the corresponding product measure on $(\Omega_{k},\mathcal{F}_{k})$. For any event $A \in \mathcal F$, we denote $A^c$ as the complement of $A$. For any sub-$\sigma$-algebra $\mathcal{G}\subseteq \mathcal{F}_{k}$, we write $\mathbb{E}[\cdot\mid\mathcal{G}]$ for the conditional expectation given $\mathcal{G}$.}
We use $\mathcal{SO}$ to denote the stochastic oracle. The notation $\mathcal O(\cdot)$ refers to standard big-O notation, while the use of $\mathcal {\widetilde O}(\cdot)$ suppresses any factors polynomial in $\log \log (1/\varepsilon)$.

\section{An Accumulative Regularization Method for Monotone Variational Inequalities}\label{sec:deterministic-vi}

In this section, we present \textit{accumulative regularization} (AR) method for finding strong solutions of monotone VIs which satisfy \eqref{eq:Lipschitz} and \eqref{strongly_monotone} with $\mu=0$. As a preliminary result,
we first establish the convergence of AR when a general deterministic subroutine is used as the inner-loop method, with the goal of finding a strong solution in the sense that the residual defined in \eqref{def_res} is small.
We then establish the convergence of AR when a general stochastic subroutine is used as the inner-loop method, with the goal of finding a strong solution in the sense of a small expected residual $\bbe [\res_F(\bar x)]$.

\subsection{Algorithm}
The AR method for finding strong solutions of monotone VIs which satisfy \eqref{eq:Lipschitz} and \eqref{strongly_monotone} with $\mu=0$ targets a series of regularized VI subproblems and computes their approximate solutions $\{x_s\}$ through some subroutine $\mathcal{A}$. Specifically, it consists of an outer loop indexed by the epoch $s,$ where it solves
 the regularized VI subproblem:
\begin{equation}\label{proxVIP}
\textrm{find}\quad x_s^* \in X \quad \text{such that} \quad \langle F_s(x_s^*), x - x_s^* \rangle \ge 0, \quad \forall\,\, x \in X,
\end{equation}
where the regularized operator $F_s$ is defined as
\begin{equation}\label{eq:Fs}
F_s(x) \coloneqq F(x) +  r_s  (x - \bar x_s).
\end{equation}
Here, $ r_s  \geq 0$ is referred to as the \textit{regularization parameter},
and $\bar x_s$, defined in \eqref{eqn:center} as a convex combination of the previous epochs' approximate solutions, is referred to as the \textit{prox-center}. Within each epoch $s$, the AR method calls some subroutine $\mathcal{A}$ to approximately solve the regularized problem, and denotes the approximate solution by $x_{s}.$ Then, at the next epoch $s+1,$
the AR method restarts from the point $x_{s}$ and targets the next regularized VI problem.

It is immediate that when $ r_s  > 0$, the operator $F_s$ is $ r_s $-strongly monotone:
\begin{equation}
\langle F_s(x) - F_s(y), x - y\rangle \geq  r_s  \|x - y\|^2, \quad \forall \,\,x,y \in X.
\end{equation}
The subroutine $\mathcal{A}$ is used to solve this strongly monotone VI problem. Depending on the oracle model, the algorithm may have either deterministic access to the operator $F$ or only stochastic oracle access to it. Accordingly, the subroutine $\mathcal{A}$ used within AR can be either deterministic or stochastic.  In both cases, for AR to obtain a point with a residual at most $\varepsilon$ of the original VI problem \eqref{VIP},  which satisfy \eqref{eq:Lipschitz} and \eqref{strongly_monotone} with $\mu=0$, the subroutine $\mathcal{A}$ needs to satisfy analogous sublinear convergence guarantees, which we will specify in detail in the next two subsections.
\begin{algorithm}[th]
\caption{An accumulative regularization method for monotone VI}
\label{alg:ARVI}
\begin{algorithmic}[1]
\Require Total number of subproblems $S$, strictly increasing regularization parameters
$\{ r_s \}_{s=0}^S$ with $r_0 = 0$, iteration counts $\{N_s\}_{s=1}^S$, and
initial point $x_0 \in \mathbb{R}^n$.
\State Set initial prox-center to $\bar{x}_0 \coloneqq x_0$.
\For{$s = 1, \ldots, S$}
    \State Set
  \begin{equation}
       \bar{x}_s \coloneqq (1 - \gamma_s)\bar{x}_{s-1} + \gamma_s x_{s-1}
    \quad \text{with } \gamma_s \coloneqq 1 - r_{s-1}/ r_s .
    \label{eqn:center}
  \end{equation}
    \State Compute an approximate solution $x_s$ of the regularized VI subproblem
    \begin{equation}\label{eq:sub-problem}
        \langle F_s(x^*_s), x - x^*_s\rangle \geq 0, \qquad \forall\, x \in X,
    \end{equation}
    \State by running subroutine $\mathcal{A}$ for $N_s$ iterations with the initialization $x_{s-1}$.
\EndFor
\State \textbf{Output:} $x_S\coloneqq x_{S,N_S}.$
\end{algorithmic}
\end{algorithm}
The proposed AR method is formally described in Algorithm \ref{alg:ARVI}.
The name accumulative regularization refers both to the fact that the regularization strength $ r_s $ increases across epochs and that the prox-center $\bar x_s$ accumulates a convex combination of previous approximate solutions:
\begin{align}\label{eqn:center-convex}
 \bar{x}_s &=\tsum_{i=1}^s\tfrac{ r_{i}-{ r_{i-1}}}{ r_s }x_{i-1}.
\end{align}

AR starts with a small regularization parameter, with the first positive $r_s$ typically chosen on the order of the target accuracy $\varepsilon$, and an epoch length $N_s$ inversely proportional to $r_s$. As $r_s$ increases, the regularized subproblems become better conditioned and $N_s$ decreases geometrically. Consequently, the first few epochs dominate the computational cost, and $\tsum_{s=1}^S N_s$ remains of the same order as the first-epoch cost. Meanwhile, the prox-center update in \eqref{eqn:center-convex} assigns greater weight to recent approximate solutions, such as $x_{s-1},x_{s-2},\ldots$, while gradually reducing the influence of $x_0$. Compared with fixed-center regularization anchored at $x_0$, AR therefore adapts the centers while controlling the total complexity through the joint choice of $r_s$ and $N_s$. It can thus be viewed as a refined restarting scheme with adaptive centers.

The idea of accumulative regularization has long been widely adopted, for example, with a similar strategy appearing in \cite{AllenZhu2018}, although the resulting method did not attain the optimal deterministic oracle complexity. This optimal complexity in the deterministic setting was later achieved by the first order AR method in \cite{lan_optimal_2024}, and was subsequently extended to higher order methods in \cite{ji2025high}.
However, these methods focus on gradient norm minimization and rely on function value guarantees provided by the subroutine $\mathcal{A}$. This structure is unavailable for the VIs \eqref{VIP}, where the operator $F$ need not be the gradient of a scalar objective function. Consequently, it is not immediate how to instantiate the AR framework in this setting, what requirements should be imposed on the subroutine $\mathcal{A}$, and what type of residual guarantees the resulting method can provide. Our analysis addresses this gap by developing an AR scheme directly for monotone VIs, with convergence guarantees stated in terms of the residual defined in \eqref{def_res}.

We start with providing the residual decomposition.
By the definition of the residual in \eqref{def_res}, we relate the residual of the original VI problem \eqref{VIP} at $x_S$ to that of the final regularized subproblem. By the triangle inequality, we have
\begin{equation}\label{eqn:res}
    \begin{aligned}
    \res_F(x_S) &= \min_{y \in -N_X(x_S)} \|y - F(x_S)\|\\
&\leq \min_{y \in -N_X(x_S)} \|y - F_S(x_S)\| + \|F_S(x_S) - F(x_S)\| \\
&= \res_{F_S}(x_S) +  r_S \|x_S - \bar x_S\|\\
&\leq \res_{F_S}(x_S)  +  r_S \|x_S -x_S^*\|+ r_S \|x_S^*- \bar x_S\|.
    \end{aligned}
\end{equation}
This inequality provides the basic residual decomposition used in the analysis. The first term is the residual of $x_S$ with respect to the regularized operator $F_S$ and therefore measures how accurately the subroutine $\mathcal A$ solves the final regularized VI subproblem. The remaining terms bound the regularization error, which measures the discrepancy between $F_S$ and the original operator $F$. Thus, to obtain a point with a small residual for the original problem, it suffices to control both the residual of the final regularized subproblem and the regularization error.

Define the initial distance to the solution set as
\begin{equation}\label{def:initial-optimality-gap}
\operatorname{dist}(x_0, X^*) \coloneqq \min_{x^* \in X^*} \|x_0 - x^*\|.
\end{equation}

We next present a result showing that AR yields certain epoch-wise relations between the iterates $x_s$, the prox-centers $\bar x_s$, and the exact solutions $x_s^*$ of the regularized subproblems. In the deterministic setting, these relations reduce to two deterministic inequalities. In the stochastic setting, they hold  along every sample path. The detailed convergence results for the two cases are deferred to the next two subsections.
\begin{proposition}
\label{lemma_ARVI_proximity}
In \autoref{alg:ARVI}, for all $s\geq 1,$ for all $x^*\in X^*$
we have
    \begin{align}
     \|x_{0} - x_1^*\| &\leq \|x_0-x^*\|, \label{eq:epoch_proximity0}\\
        \|x_{s} - x_{s+1}^*\| &\leq \|x_{s } - x^*_{s }\|, \label{eq:epoch_proximity1} \\
         r_s  \|\bar x_s - x^*_s\| &\leq \tsum_{i = 1}^s ( r_{i} +  r_{i-1}) \|x_{i - 1} - x_{i - 1}^*\| \label{eq:epoch_proximity2},\\
        \|x_{s}^* - x^*\| &\leq 2 \tsum_{k= 0}^{s - 1} \|x_{k} - x_k^*\|.\label{eqn:induction-traingle}
    \end{align}
\end{proposition}
\begin{proof}
Observe that $r_0 = 0$, so that $F_0 = F$ and we denote $x_0^* \coloneqq x^*$, where $x^* \in X^*$.
Given that $x_{s-1}^*$ is the optimal solution of $ \langle F_{s-1}(x^*_{s-1}), x - x^*_{s-1}\rangle \geq 0$ for all $ x \in X,$ we obtain
\begin{equation}\label{eqn:optimality-1}
    \langle F(x^*_{s - 1}) +  r_{s - 1}(x^*_{s - 1} - \bar x_{s - 1}), x^*_s - x^*_{s - 1}\rangle \geq 0.
\end{equation}
By the definition of $F_s$ in \eqref{eq:sub-problem} and the center $\bar{x}_s$ in \eqref{eqn:center}, we obtain
\begin{align*}
F_s(x) = F(x) +  r_{s - 1}(x - \bar x_{s-1}) + ( r_s  -  r_{s - 1})(x - x_{s - 1}).
\end{align*}
Combining it with the fact that $x_{s}^*$ is the optimal solution of $ \langle F_{s}(x^*_{s}), x - x^*_{s}\rangle \geq 0$ for all $ x \in X,$ we obtain
\begin{equation}\label{eqn:optimality-2}
   \langle F(x^*_s)+  r_{s - 1}(x^*_s - \bar x_{s - 1}) + ( r_s  -  r_{s - 1})(x^*_s - x_{s - 1}), x^*_{s - 1} - x^*_s\rangle \geq 0.
\end{equation}
Combining \eqref{eqn:optimality-1} with \eqref{eqn:optimality-2}, we have
\begin{align}\label{eqn:optimality-3}
    r_{s - 1}\langle x^*_{s - 1} - x^*_s, x^*_s - x^*_{s - 1}\rangle + ( r_s  -  r_{s - 1})\langle x^*_s - x_{s - 1}, x^*_{s - 1} - x^*_s\rangle \geq  \langle  F(x^*_s)-F(x^*_{s - 1} ), x^*_s - x^*_{s - 1}\rangle.
\end{align}
By the monotonicity of $F$, it holds that $\langle F(x^*_s) - F(x^*_{s - 1}), x^*_s - x^*_{s - 1}\rangle \geq 0$. Moreover, for all $a, b \in\mathbb{R}^n,$ it holds that $2\langle a, b\rangle=\|a+b\|^2-\|a\|^2-\|b\|^2.$ Therefore, we have
\begin{align*}
\langle x^*_s - x_{s - 1}, x^*_{s - 1}-x^*_s\rangle &=\tfrac{1}{2} \|x^*_{s - 1} - x_{s - 1}\|^2 -\tfrac{1}{2} \|x_s^* - x_{s - 1}\|^2  - \tfrac{1}{2} \|x^*_s - x^*_{s- 1}\|^2.
\end{align*}
Thus, substituting it into \eqref{eqn:optimality-3}, we have
\begin{equation*}
    \tfrac{ r_s  -  r_{s - 1}}{2} \|x^*_{s - 1} - x_{s - 1}\|^2 \geq \tfrac{ r_s  -  r_{s - 1}}{2} \|x^*_s - x_{s - 1}\|^2 + \tfrac{ r_s }{2} \|x^*_s - x^*_{s- 1}\|^2.
\end{equation*}
Therefore, for every $s\geq 1$, $\|x_{s-1}-x_s^*\|\leq \|x_{s-1}-x_{s-1}^*\|.$
Taking $s=1$ and using the choice of $x_0^* \in X^*$ gives
\begin{equation*}
\|x_0-x_1^*\|\leq \operatorname{dist}(x_0,X^*),
\end{equation*}
which proves \eqref{eq:epoch_proximity0}. Replacing $s$ by $s+1$  proves \eqref{eq:epoch_proximity1}.
\smallskip
Furthermore, by the definition of the center $\bar{x}_s$ in \eqref{eqn:center}, it holds that
\begin{equation}\label{eqn:resursive}
    \begin{aligned}
         r_s  \|x_s^* - \bar x_s\|
&=  r_s  \left \|x^*_s - \tfrac{ r_{s - 1}}{ r_s } \bar x_{s - 1} - x_{s - 1} + \tfrac{ r_{s - 1}}{ r_s } x_{s - 1} \right \| \\
&\leq  r_s \|x_s^* - x_{s - 1}\| +  r_{s - 1}\|x_{s - 1} - \bar x_{s - 1}\| \\
&\leq  r_s \|x_s^* - x_{s - 1}\| +  r_{s - 1}\|x_{s - 1} - x^*_{s - 1}\| +  r_{s - 1} \|x^*_{s - 1} - \bar x_{s - 1}\| \\
&\!\!\!\overset{\eqref{eq:epoch_proximity1}}{\leq} ( r_s  +  r_{s - 1}) \|x^* _{s - 1} - x_{s - 1}\| +  r_{s - 1} \|x^* _{s - 1} - \bar x_{s - 1}\|.
    \end{aligned}
\end{equation}
Iteratively applying \eqref{eqn:resursive} yields the conclusion \eqref{eq:epoch_proximity2}.
\smallskip

We prove \eqref{eqn:induction-traingle} by induction.
The case $s=1$ follows from the following.
\begin{align*}
    \|x_1^* - x^*\|  \leq \|x_1^* - x_{0}\| + \|x_{0} - x^*\| \overset{\eqref{eq:epoch_proximity0}}{\leq} 2 \|x_{0} - x^*\|.
\end{align*}
Suppose \eqref{eqn:induction-traingle} holds for $s-1$. Then, for index $s$, it holds that
\begin{align*}
    \|x_s^* - x^*\|  \leq \|x_s^* - x_{s - 1}\| + \|x_{s - 1} - x_{s - 1}^*\| + \|x_{s - 1}^* - x^*\| \overset{\eqref{eq:epoch_proximity1}}{\leq} 2 \|x_{s - 1} - x^*_{s - 1}\| + \|x_{s - 1}^* - x^*\|.
\end{align*}
Combining this with \eqref{eqn:induction-traingle},
we have $\|x_{s}^* - x^*\| \leq 2 \tsum_{k= 0}^{s - 1} \|x_{k} - x_k^*\|$.

\end{proof}
\vgap

\subsection{Deterministic variational inequalities}

In this subsection, we consider the AR method for solving deterministic VIs with $\mu=0$, where convergence is measured in terms of the residual. Although we focus on the monotone case, the AR method also applies to strongly monotone VIs. The stochastic results developed in later sections cover both the monotone and strongly monotone cases.

We make  the following assumption regarding the convergence of the subroutine $\mathcal A=\mathcal A (F,  r_s , \bar{x}_{s}, x_{s-1})$.
\begin{assumption}\label{assumption_OE}
The subroutine $\mathcal A (F,  r_s , \bar{x}_{s}, x_{s-1})$ for solving the regularized problem \eqref{eq:sub-problem}
exhibits the following performance guarantee after $k$ iterations, for all $k\geq 1,$
\begin{align}
    \|x_{s,k}-x_s^*\|^2&\leq \tfrac{c_\mathcal{A}(L+ r_s )^2}{r_s^2k^2}\|x_{s-1}-x_s^*\|^2,\label{eqn:distance}\\
    \res_{F_s}(x_{s,k})&\leq{c_\mathcal{A}(L+ r_s )}\|x_{s-1}-x_s^*\|,\label{eqn:residual}
\end{align}
where $c_\mathcal{A}$ is a constant depending only on the subroutine $\mathcal A$.

\end{assumption}
The above assumption is satisfied by many algorithms. For example,
operator extrapolation (OE) \cite[Theorem~3.3]{kotsalis2022simple}
satisfies it with $c_{\mathcal A}=64$.  In particular, the residual
guarantee in \eqref{eqn:residual} is mild: it requires only a relatively
slow residual convergence rate, which depends on the initial
distance $\|x_{S-1}-x_S^*\|$ at epoch $S$.

We are now ready to state the convergence guarantee of AR for computing a
strong solution of the monotone VI~\eqref{VIP} under the Lipschitz condition
\eqref{eq:Lipschitz} and the monotonicity condition
\eqref{strongly_monotone} with $\mu=0$.

\begin{theorem}\label{the:sublinear_AR}
Assume that $0<\varepsilon \le LD_0$, where $D_0\geq \textnormal{dist}(x_0, X^*)$. In \autoref{alg:ARVI}, suppose 
\begin{equation}\label{eq:deterministic-S}
S = 1+\left \lceil\log_2  \tfrac{\tilde{c}_\mathcal{A}LD_0}{\varepsilon}\right \rceil,\quad
 r_s  = \tfrac{2^{s-1} \varepsilon}{\tilde{c}_\mathcal{A}D_0},\quad   N_s = \left\lceil\tfrac{4\sqrt{c_\mathcal{A}}(L+ r_s )}{r_s }\right\rceil,
\end{equation}
where $\tilde{c}_\mathcal{A}=\max\{8c_\mathcal{A}+1, 10/3\}.$
Then, \autoref{alg:ARVI} computes a point $x_S$ such that $\res_F(x_S) \leq \varepsilon$ after
\begin{equation}
 \tfrac{8\tilde{c}_\mathcal{A}\sqrt{c_\mathcal{A}}LD_0}{\varepsilon}+(4\sqrt{c_\mathcal{A}}+1)(S+1)=\mathcal{O}(\tfrac{LD_0}{\varepsilon})
\end{equation}
evaluations of $F$.
\end{theorem}

\begin{proof}
We first show that for each $s \in [S]$,
\begin{equation}\label{eq:deterministic-inductive-hypothesis}
\|x_s - x_s^*\|^2 \leq \tfrac{D_0^2}{16^s},
\end{equation}
where $ x_s=x_{s, N_s}$ for all $s\in[S]$.
\smallskip

When $s=1$, by \autoref{assumption_OE} and \autoref{lemma_ARVI_proximity}, it holds that
 \begin{equation*}
     \begin{aligned}
\|x_1-x_1^*\|^2&\overset{\eqref{eqn:distance}}{\leq} \tfrac{c_{\mathcal{A}}(L+ r_1 )^2\|x_{0}-x_1^*\|^2}{ r_1 ^2 N_1^2}\overset{\eqref{eq:epoch_proximity0}}{\leq}\tfrac{c_{\mathcal{A}}(L+ r_1 )^2D_0^2}{ r_1 ^2 N_1^2}\overset{\text{(i)}}{\leq}\tfrac{D_0^2}{  {16}},
     \end{aligned}
 \end{equation*}
 where in (i), we used the condition on the inner iteration number $N_s$ in \eqref{eq:deterministic-S}.

Let $s\in\{2,\ldots,S\}$ and suppose that the claim holds
at epoch $s-1$. Then
 by \autoref{assumption_OE} and \autoref{lemma_ARVI_proximity}, it holds that
\begin{align}\label{eqn:induction-s-1}
    \|x_s - x_s^*\|^2 &\overset{\eqref{eqn:distance}}{\leq} \tfrac{c_\mathcal{A}(L+ r_s )^2\|x_{s-1}-x_s^*\|^2}{ r_s ^2N_s^2} \overset{\eqref{eq:epoch_proximity1}}{\leq} \tfrac{c_\mathcal{A}(L+ r_s )^2\|x_{s-1}-x_{s-1}^*\|^2}{ r_s ^2N_s^2} \overset{\eqref{eq:deterministic-S}}{\leq} \tfrac{\|x_{s - 1} - x_{s-1}^*\|^2}{16}{\leq} \tfrac{D_0^2}{16^s}.
\end{align}
Hence, \eqref{eq:deterministic-inductive-hypothesis} holds.
\smallskip

By the residual decomposition
in \eqref{eqn:res} and \autoref{assumption_OE}, it holds that
\begin{equation}\label{eqn:residual-decom}
    \begin{aligned}
         \res_F(x_S)
&\,\overset{\eqref{eqn:res} }{\leq} \res_{F_S}(x_S)  +  r_S \|x_S -x_S^*\|+ r_S \|x_S^*- \bar x_S\|\\
&\overset{\eqref{eqn:residual}}{\leq}{c_\mathcal{A}(L+ r_S )}\|x_{S - 1} - x_{S}^*\| +  r_S \|x_S -x_S^*\|+ r_S \|x_S^*- \bar x_S\|\\
&\,\,\,\overset{\text{(ii)}}{\leq} {c_\mathcal{A}(L+ r_S )}\|x_{S - 1} - x_{S-1}^*\|  +  r_S \|x_S -x_S^*\|+\tsum_{s = 1}^S ( r_s  +  r_{s - 1}) \|x_{s - 1} - x_{s - 1}^*\|,
    \end{aligned}
\end{equation}
where in (ii), we used \autoref{lemma_ARVI_proximity}.
Substituting  \eqref{eq:deterministic-inductive-hypothesis} into \eqref{eqn:residual-decom}, we obtain
\begin{equation}\label{eqn:residual-decom'}
    \begin{aligned}
        \res_F(x_S)
&\leq \tfrac{c_\mathcal{A}(L+ r_S ) D_0}{4^{S-1}} + \tsum_{s = 1}^S \tfrac{( r_s  +  r_{s - 1})D_0}{4^{s - 1}}+ \tfrac{ r_S  D_0}{4^S}.
    \end{aligned}
\end{equation}
Observe that for the first term in \eqref{eqn:residual-decom'}, we have
\begin{equation*}
    \begin{aligned}
        \tfrac{c_\mathcal{A}(L+ r_S ) D_0}{4^{S-1}}\overset{\textnormal{(iii)}}{\leq}  \tfrac{2c_\mathcal{A} r_S  D_0}{4^{S-1}},
    \end{aligned}
\end{equation*}
where in (iii), we used $ r_S \geq L$ due to the choice of $S$ and $ r_S $ in \eqref{eq:deterministic-S}. Combining it with  \eqref{eqn:residual-decom'}, we obtain
\begin{equation*}
    \begin{aligned}
        \res_F(x_S)
&\leq \tfrac{2c_\mathcal{A} r_S  D_0}{4^{S-1}}+ \tsum_{s = 1}^S \tfrac{( r_s  +  r_{s - 1})D_0}{4^{s - 1}}+ \tfrac{ r_S  D_0}{4^S}=\left(8c_\mathcal{A}+1\right)\tfrac{ r_S  D_0}{4^{S}} + \tsum_{s = 1}^S \tfrac{( r_s  +  r_{s - 1})D_0}{4^{s - 1}} \leq\tfrac{\tilde{c}_\mathcal{A} r_S  D_0}{4^{S}} + \tfrac{5\varepsilon}{2\tilde{c}_{\mathcal{A}}}\leq \varepsilon.
    \end{aligned}
\end{equation*}
The total iteration complexity is bounded as
\begin{equation*}
    \tsum_{s = 1}^{S} N_s = \tsum_{s = 1}^{S}\left\lceil\tfrac{4\sqrt{c_\mathcal{A}}(L+ r_s )}{r_s }\right\rceil \leq\tfrac{8\tilde{c}_\mathcal{A}\sqrt{c_\mathcal{A}}LD_0}{\varepsilon}+(4\sqrt{c_\mathcal{A}}+1)(S+1).
\end{equation*}
This concludes the proof.
\end{proof}
\smallskip

\autoref{the:sublinear_AR} highlights the role of accumulative
regularization in improving the total complexity of finding strong
solutions. As the regularization strength $ r_s $ increases
geometrically across epochs, the regularized subproblems become better
conditioned, allowing the epoch lengths $N_s$ to decrease
geometrically until they reach constant order. Consequently, the
subproblems in later epochs need only be solved coarsely. Nevertheless,
the overall method requires only
$\mathcal{O}(LD_0/\varepsilon)=\mathcal{O}(1/\varepsilon)$ operator
evaluations to compute a strong solution $x_S$ satisfying
$\res_F(x_S)\leq\varepsilon$ for a monotone VI \eqref{VIP} satisfying
\eqref{eq:Lipschitz} and \eqref{strongly_monotone} with $\mu= 0$.

\subsection{Stochastic variational inequalities under uniform noise}\label{sec:Stochastic Variational Inequalities under Uniform Noise}

In this section, we study the problem of finding strong solutions of the
monotone VI~\eqref{VIP}, where $F$ satisfies \eqref{eq:Lipschitz} and
\eqref{strongly_monotone} with $\mu\geq 0$. We assume that $F$ is accessible only
through noisy evaluations returned by successive calls to a stochastic oracle
$(\mathcal{SO})$. Our goal is to bound the total number of oracle calls
required to produce an iterate $x_S$ satisfying
$\mathbb{E}[\res_F(x_S)]\leq\varepsilon$ for a prescribed accuracy
$\varepsilon>0$.

We assume that all data $\{\zeta_{k}\}$ are i.i.d random variables. Throughout this section, we impose the uniformly bounded noise assumption \eqref{eq:intro-uniform-assumption} and unbiased estimator assumption \eqref{eq:fast-decreasing-bias-o}.

We now supply \autoref{alg:ARVI} with a stochastic subroutine $\mathcal{A}$ for approximately solving the regularized problem \eqref{eq:sub-problem}. In the stochastic case,
due to the presence of noise, using mini-batches of samples will play an important role in deriving accelerated convergence rates.
Hence,
at iteration $k$ of epoch $s$, we construct a mini-batch estimator of $F$ from a block of $m_{s,k}$ i.i.d.\ samples, denoted by
$\{\zeta_{s,k}^{\,i}\}_{i=1}^{m_{s,k}}$. We fix the batch size within each epoch, setting $m_{s,k}\equiv m_s$; together with the i.i.d.\ sampling, this ensures that the unbiasedness and variance bounds are uniform across inner iterations, even though the mini-batch realization changes with $(s,k)$. Hence, for brevity, we suppress the batch indices and write $\tilde{F}_{s,k}(x_{s,k})$.
Specifically,
given the point $x_{s,k}$, the mini-batch estimator is defined as
\begin{equation}\label{eqn:minibatch}
\tilde{F}_{s,k}(x_{s,k}) \coloneqq \tfrac{1}{m_s} \tsum_{i = 1}^{m_s}\tilde{F}(x_{s,k},\zeta_{s, k}^{\,i}).
\end{equation}

We let $\mathcal{F}_{s,k}$ denote the filtration generated by the global trajectory up to, but not including, the samples used to form $\tilde{F}_{s,k}$. Accordingly, we assume that \eqref{eq:fast-decreasing-bias-o} and \eqref{eq:intro-uniform-assumption} hold with $\mathcal{F}_k$ and $\zeta_{k}$ replaced by their double-indexed counterparts $\mathcal{F}_{s,k}$ and $\zeta_{s,k}^{\,i}$. By the definition of $\mathcal{F}_{s,k}$, both $x_{s,k}$ and $x_s^*$ are $\mathcal{F}_{s,k}$-measurable.
Then, from the construction of the mini-batch estimator and assumptions \eqref{eq:fast-decreasing-bias-o} and \eqref{eq:intro-uniform-assumption},
it holds that
\begin{align}
    \mathbb E[\tilde{F}_{s,k}(x_{s,k}) \mid \mathcal{F}_{s,k}] = F(x_{s,k}), \quad \mathbb E\left[\|\tilde{F}_{s,k}(x_{s,k})-F(x_{s,k})\|^2 \mid \mathcal{F}_{s,k}\right] \leq \tfrac{ \sigma^2}{m_s}, \quad \textnormal{a.s.}
\end{align}

We make the following assumption regarding the convergence of the subroutine $\mathcal A=\mathcal A (\mathcal{SO},  r_s , \bar{x}_{s}, x_{s-1}, m_s)$.
\begin{assumption}\label{assumption_inner_sto}
The subroutine $\mathcal A (\mathcal{SO},  r_s , \bar{x}_{s}, x_{s-1}, m_s)$ for solving the regularized problem \eqref{eq:sub-problem}
exhibits the following performance guarantee after $k$ iterations:
\begin{align}
    \mathbb{E}\left[\|x_{s,k}-x_s^*\|^2\,\middle|\,\mathcal{F}_{s,0}\right]&\leq \tfrac{c_{\mathcal{A}}(L+ r_s )^2}{ (\mu+r_s) ^2k^2}\|x_{s-1}-x_s^*\|^2 +\tfrac{c_{\mathcal{A}} \sigma^2}{km_s (\mu+r_s) ^2},\label{eqn:distance-s}\\
 \sqrt{\mathbb{E}\left[\res_{F_s}(x_{s,k})^2\,\middle|\,\mathcal{F}_{s,0}\right]}&\leq {c_{\mathcal{A}}(L+ r_s )}\|x_{s-1}-x_s^*\|+\tfrac{c_{\mathcal{A}} \sigma\sqrt{k}}{\sqrt{m_s}},\label{eqn:residual-s}
\end{align}
where $c_{\mathcal{A}}$ is a constant depending only on the subroutine $\mathcal A$.
\end{assumption}
The above assumption can be readily verified for existing algorithms by
slightly modifying their analyses. For example, stochastic operator
extrapolation (SOE) \cite{kotsalis2022simple} satisfies it with
$c_{\mathcal A}=64$. The residual guarantee in \eqref{eqn:residual-s} is relatively mild and
easy to satisfy, as its deterministic
term is independent of $k$, while its stochastic term is not required to
decrease with $k$. We later establish such a guarantee for
SOE under the more general state-dependent noise model and omit its
derivation here for simplicity.

 \vgap
We are now ready to state the convergence guarantee of AR for computing a
strong solution of the monotone VI~\eqref{VIP}, where $F$ satisfies
\eqref{eq:Lipschitz} and \eqref{strongly_monotone} with $\mu\geq 0$. 

\begin{theorem}\label{thm:main-unbiased}
Assume that $0<\varepsilon \le LD_0$, where $D_0\geq \textnormal{dist}(x_0, X^*)$. In \autoref{alg:ARVI}, suppose 
\begin{align}\label{eqn:regularization-exp-unbiased}
         S&=\left\lceil
2 + \log_4\!\left(\tfrac{16(16c_{\mathcal A}+4)LD_0}{\varepsilon}\, S_0\log(4S_0)\right)
\right\rceil,\,\,  r_s =\tfrac{4^{s-1}\varepsilon}{(16c_{\mathcal A}+4)D_0s\log S},\,\,
\end{align}
where $S_0$ is defined as
\begin{align}\label{eqn:def-S0-unbiased}
S_0 \coloneqq \log_4\!\left(\tfrac{16(16c_{\mathcal A}+4)LD_0}{\varepsilon}\right).
\end{align}
For each epoch, the epoch length and the batch size satisfy 
\begin{align}\label{eqn:batch-epoch}
   N_s&={\left\lceil\tfrac{4\sqrt{2c_{\mathcal{A}}}(L+ r_s )}{r_s }\right\rceil},\,\,
m_s=\left\lceil\tfrac{16^{s-1}N_s \sigma^2}{(L+ r_s )^2 D_0^2}+1\right\rceil.
\end{align}
Then, at epoch $S$, \autoref{alg:ARVI} computes an approximate solution $x_S$ such that $\sqrt{\mathbb{E}\left[\res_F(x_S)^2\right]}
\le\varepsilon$ after
\begin{align}\label{eqn:final-sample-complexity}
   \widetilde{\mathcal{O}}\left(\tfrac{LD_0}{\varepsilon}+\tfrac{ \sigma^2}{\varepsilon^2}(\log\tfrac{LD_0}{\varepsilon})^3\right)
\end{align}
calls to the $\mathcal{SO}$, where $\widetilde{\mathcal{O}}$ hides factors polynomial in $\log\log(\tfrac{LD_0}{\varepsilon})$.
\end{theorem}
\begin{proof}
We first show that for each $s \in [S]$,
\begin{equation}\label{eq:stochastic-inductive-hypothesis}
    \mathbb{E}[\|x_s - x_s^*\|^2] \leq \tfrac{D_0^2}{16^s},
\end{equation}
where $x_s=x_{s, N_s}.$
\smallskip

When $s=1$, by \autoref{assumption_inner_sto} and \autoref{lemma_ARVI_proximity}, it holds that
 \begin{equation*}
     \begin{aligned}
          \mathbb{E}[\|x_1-x_1^*\|^2] &\overset{\text{(i)}}{\leq} \tfrac{c_{\mathcal{A}}(L+ r_1 )^2\|x_{0}-x_1^*\|^2}{ (\mu+r_1) ^2 N_1^2}+\tfrac{c_{\mathcal{A}}(L+ r_1 )^2D_0^2}{ (\mu+r_1) ^2 N_1^2}\overset{\eqref{eq:epoch_proximity0}}{\leq}\tfrac{2c_{\mathcal{A}}(L+ r_1 )^2D_0^2}{ (\mu+r_1) ^2 N_1^2}\overset{\text{(ii)}}{\leq}\tfrac{D_0^2}{  {16}},
     \end{aligned}
 \end{equation*}
 where in (i), we used \eqref{eqn:distance-s} from \autoref{assumption_inner_sto} and the condition on the batch size $m_s$
 in \eqref{eqn:batch-epoch}; in (ii), we used the condition on the inner iteration number $N_1$ in \eqref{eqn:batch-epoch} and $\mu\geq 0$.

 Suppose that \eqref{eq:stochastic-inductive-hypothesis} holds at epoch $s-1$. Then, for epoch $s$, similarly, by \autoref{assumption_inner_sto}, \autoref{lemma_ARVI_proximity}, and the condition on the batch size $m_s$
 in \eqref{eqn:batch-epoch}, it holds that
 \begin{equation*}
     \begin{aligned}
          \mathbb{E}[\|x_s-x_s^*\|^2] &\leq\tfrac{c_{\mathcal{A}}(L+ r_s )^2\mathbb{E}[\|x_{s-1}-x_{s-1}^*\|^2]}{N_s^2 (\mu+r_s) ^2}+\tfrac{1}{N_s^2 (\mu+r_s)^2}\cdot\tfrac{c_{\mathcal{A}}(L+ r_s )^2D_0^2}{  {16^{s-1}}}\\
          &\overset{\text{(iii)}}{\leq}\tfrac{c_{\mathcal{A}}(L+ r_s )^2D_0^2}{  {16^{s-1}}N_s^2 (\mu+r_s)^2}+\tfrac{1}{N_s^2 (\mu+r_s) ^2}\cdot\tfrac{c_{\mathcal{A}}(L+ r_s )^2D_0^2}{  {16^{s-1}}}\overset{\text{(iv)}}{\leq}\tfrac{D_0^2}{  16^{s}},
     \end{aligned}
 \end{equation*}
 where in (iii), we used the induction hypothesis; in (iv), we again used the condition on the inner iteration number $N_s$ in \eqref{eqn:batch-epoch} and $\mu\geq 0$.
 \smallskip

By the residual decomposition in \eqref{eqn:res}, it follows that
\begin{equation}\label{eqn:residual-decom-s}
\begin{aligned}
\sqrt{\mathbb{E}\left[\res_F(x_S)^2\right]}
&\overset{\eqref{eqn:res}}{\leq}
\sqrt{\mathbb{E}\left[\left(\res_{F_S}(x_S)+r_S\|x_S-x_S^*\|+r_S\|x_S^*-\bar x_S\|\right)^2\right]}\\
&\overset{\textnormal{(iv)}}{\leq}
\sqrt{\mathbb{E}\left[\res_{F_S}(x_S)^2\right]}
+r_S\sqrt{\mathbb{E}\left[\|x_S-x_S^*\|^2\right]}
+r_S\sqrt{\mathbb{E}\left[\|x_S^*-\bar x_S\|^2\right]},
\end{aligned}
\end{equation}
where in (iv), we used Minkowski's inequality. Furthermore, by the \eqref{eqn:residual-s} in \autoref{assumption_inner_sto}, we have
\begin{equation}\label{eqn:regularized-square}
\begin{aligned}
\sqrt{\mathbb{E}\left[\res_{F_S}(x_S)^2\right]}
&\overset{\eqref{eqn:residual-s}}{\leq}
\sqrt{\mathbb{E}\left[\left(c_{\mathcal{A}}(L+r_S)\|x_{S-1}-x_S^*\|+\tfrac{c_{\mathcal{A}}\sigma\sqrt{N_S}}{\sqrt{m_S}}\right)^2\right]}\\
&\,\,\overset{\textnormal{(v)}}{\leq}
c_{\mathcal{A}}(L+r_S)\sqrt{\mathbb{E}[\|x_{S-1}-x_{S-1}^*\|^2]}
+\tfrac{c_{\mathcal{A}}\sigma\sqrt{N_S}}{\sqrt{m_S}},
\end{aligned}
\end{equation}
where in (v), we used \autoref{lemma_ARVI_proximity} and Minkowski's inequality. Similarly, by \autoref{lemma_ARVI_proximity} and Minkowski's inequality, we have
\begin{equation}\label{eqn:regularized-sol-center}
r_S\sqrt{\mathbb{E}\left[\|x_S^*-\bar x_S\|^2\right]}
\leq
\tsum_{s=1}^S(r_s+r_{s-1})\sqrt{\mathbb{E}[\|x_{s-1}-x_{s-1}^*\|^2]}.
\end{equation}
Substituting \eqref{eq:stochastic-inductive-hypothesis}, \eqref{eqn:regularized-square}, \eqref{eqn:regularized-sol-center} into \eqref{eqn:residual-decom-s}, we obtain
\begin{equation}\label{eqn:residual-decom'-s}
    \begin{aligned}
          \sqrt{\mathbb{E}\left[\res_F(x_S)^2\right]}&\leq \tfrac{c_{\mathcal{A}}(L+ r_S ) D_0}{4^{S-1}}+\tfrac{c_{\mathcal{A}}\sigma\sqrt{N_S}}{\sqrt{m_S}}  + \tsum_{s = 1}^S \tfrac{( r_s  +  r_{s - 1})D_0}{4^{s - 1}}+ \tfrac{ r_S  D_0}{4^S}.
    \end{aligned}
\end{equation}
Observe that for the first term in \eqref{eqn:residual-decom'-s}, we have
\begin{equation}\label{eqn:first-term}
    \begin{aligned}
        \tfrac{c_{\mathcal{A}}(L+ r_S ) D_0}{4^{S-1}}\overset{\textnormal{(vi)}}{\leq}  \tfrac{2c_{\mathcal{A}} r_S  D_0}{4^{S-1}}.
    \end{aligned}
\end{equation}
Here in (vi),  we used $ r_S \geq L$, which follows from the choice of $S$ and $ r_S $ in \eqref{eqn:regularization-exp-unbiased} and will be established at the end of the proof.
\smallskip
For the second term in \eqref{eqn:residual-decom'-s}, by the choice of the mini-batch size $m_S$ in \eqref{eqn:batch-epoch}, we obtain
\begin{align}\label{eqn:second-term}
    \tfrac{c_{\mathcal{A}}\sigma\sqrt{N_S}}{\sqrt{m_S}} \leq   \tfrac{c_{\mathcal{A}}(L+ r_S ) D_0}{4^{S-1}} \overset{\eqref{eqn:first-term}}{\leq}\tfrac{2c_{\mathcal{A}} r_S  D_0}{4^{S-1}}.
\end{align}
Substituting \eqref{eqn:first-term} and \eqref{eqn:second-term} into \eqref{eqn:residual-decom'-s}, we obtain
\begin{equation*}
    \begin{aligned}
   \sqrt{\mathbb{E}\left[\res_F(x_S)^2\right]}
&\leq\tfrac{c_{\mathcal{A}} r_S  D_0}{4^{S-2}}+ \tsum_{s = 1}^S \tfrac{( r_s  +  r_{s - 1})D_0}{4^{s - 1}}+ \tfrac{ r_S  D_0}{4^S}\leq\tfrac{\left(16c_{\mathcal{A}}+1\right) r_S  D_0}{4^{S}}+ {\tfrac{5D_0}{4}}\tsum_{s=1}^{S}\tfrac{ r_s }{  {4}^{s-1}},
    \end{aligned}
\end{equation*}
 where in the last step, we used $r_0=0$. Observe that by the choice of the regularization $ r_s $ in \eqref{eqn:regularization-exp-unbiased}, we obtain
 \begin{align*}
\tfrac{\left(16c_{\mathcal{A}}+1\right) r_S  D_0}{4^{S}}\leq\tfrac{\varepsilon}{4},\quad \tsum_{s=1}^{S}\tfrac{ r_s }{4^{s-1}}
&\overset{\eqref{eqn:regularization-exp-unbiased}}{=}
\tsum_{s=1}^{S}\tfrac{4^{s-1}\varepsilon}{\left(16c_{\mathcal{A}}+4\right)D_0s\log S}\left(\tfrac{1}{4}\right)^{s-1}
\le
\tfrac{\varepsilon(1+\log S)}{\left(16c_{\mathcal{A}}+4\right)D_0\log S}
\le
\tfrac{\varepsilon}{2D_0},
\end{align*}
where in the last step, we used $\tfrac{1+\log S}{\log S}\leq 2$ and $
16c_{\mathcal A}+4\geq 4,$ due to $S\geq 3.$ Hence $\sqrt{\mathbb{E}\left[\res_F(x_S)^2\right]}\leq \tfrac{\varepsilon}{4}+ \tfrac{5}{4}\cdot\tfrac{\varepsilon}{2}\leq \varepsilon$.
\smallskip

It remains to prove $ r_S  \ge L$. Since $\varepsilon \le LD_0$ and $16(16c_{\mathcal A}+4) \ge 64$, we have $S_0 \ge \log_4 64 = 3$. Moreover,
\begin{equation}\label{eqn:bound-on-S-new}
    S \le 3 + \log_4\!\left(\tfrac{16(16c_{\mathcal A}+4)LD_0}{\varepsilon}\, S_0\log(4S_0)\right)
      = 3 + S_0 + \log_4 S_0 + \log_4\log(4S_0) \le 4S_0,
\end{equation}
where the last step uses $S_0 \ge 3$. Hence, by the definition of $S$ in \eqref{eqn:regularization-exp-unbiased} and its bound in \eqref{eqn:bound-on-S-new},
\begin{align*}
    4^S \ge \tfrac{16(16c_{\mathcal A}+4)LD_0}{\varepsilon}\, S_0\log(4S_0), \qquad S\log S \le 4S_0\log(4S_0),
\end{align*}
which together yield $\tfrac{4^S}{S\log S} \ge \tfrac{4(16c_{\mathcal A}+4)LD_0}{\varepsilon}$. Therefore, we have
\begin{align*}
     r_S  = \tfrac{4^{S-1}\varepsilon}{(16c_{\mathcal A}+4)D_0 S\log S}
           \ge \tfrac{\varepsilon}{(16c_{\mathcal A}+4)D_0} \cdot \tfrac{(16c_{\mathcal A}+4)LD_0}{\varepsilon}
             = L.
\end{align*}
The total sample complexity can be bounded as
\begin{equation*}
    \begin{aligned}
       \tsum_{s=1}^{S}\tsum_{i=1}^{N_s} m_s
&\le \tsum_{s=1}^{S}\left[2+\tfrac{16^{s-1}N_s \sigma^2}{(L+ r_s )^2D_0^2}\right]\left[2+\tfrac{4\sqrt{2c_{\mathcal{A}}}(L+ r_s )}{ r_s }\right]\\
&=\mathcal{O}\left(\tfrac{LD_0}{\varepsilon}\log\log\tfrac{LD_0}{\varepsilon}+\tfrac{ \sigma^2}{\varepsilon^2}(\log\tfrac{LD_0}{\varepsilon})^3(\log\log\tfrac{LD_0}{\varepsilon})^2\right).
    \end{aligned}
\end{equation*}

\end{proof}
It is worth noting that the stochastic setting requires a stronger regularization schedule than the deterministic one. In the deterministic case, the choice $ r_s  = \mathcal{O}(2^{s-1}\varepsilon/D_0)$ is sufficient, although this schedule is not unique. In the stochastic case, however, the operator information is noisy, and such a schedule is not sufficient for the present stochastic analysis. When $ r_s $ is too small, the oracle noise may be amplified by the ill-conditioning of the regularized subproblems and can dominate the epoch-wise progress. The stochastic analysis therefore uses larger regularization parameters to stabilize the iterates and compensate for the accumulated stochastic error. When the noise level vanishes, i.e., $\sigma=0$, the stochastic bound recovers the deterministic $\mathcal{O}(1/\varepsilon)$ polynomial dependence on the target accuracy, up to the additional polynomial in $\log\log(\tfrac{LD_0}{\varepsilon})$ induced by the stochastic regularization schedule.
\smallskip

The complexity bound in \autoref{thm:main-unbiased} is nearly optimal up to logarithmic factors. By Jensen's inequality, $\mathbb{E}\left[\res_F(x_S)\right]\leq\sqrt{\mathbb{E}\left[\res_F(x_S)^2\right]}$. Consequently, the sample complexity guarantee in \autoref{thm:main-unbiased} for finding a point $x_S$ satisfying $\mathbb{E}\left[\res_F(x_S)\right]\leq\varepsilon$ is $\widetilde{\mathcal O}\bigl(\sigma^2\varepsilon^{-2}(\log\tfrac{1}{\varepsilon})^3\bigr)$. Existing residual guarantees for stochastic monotone VIs typically yield a stochastic complexity term of order $\mathcal O(\sigma^2\varepsilon^{-4})$ for directly controlling the strong residual; see, for example, \cite{kotsalis2022simple}. In the stochastic regime $\sigma>0$, this dependence matches the $\Omega\bigl(\sigma^2\varepsilon^{-2}\log(LD/\varepsilon)\bigr)$ stochastic optimization lower bound of \cite{foster2019complexity} up to logarithmic factors.

\begin{proposition}\label{prop:mu-free}
Suppose the conditions of \autoref{thm:main-unbiased} hold and $F$ is $\mu$-strongly
monotone. Let $x^*$ be the unique solution of the VI. Then, for any
$\tilde \varepsilon>0$, AR computes $x_S$ such that
$\mathbb{E}[\norm{x_S-x^*}^2]\leq\tilde \varepsilon$ after
\begin{equation}\label{eqn:mu-free}
\widetilde{\mathcal{O}}\left(\tfrac{LD_0}{\mu\sqrt{\tilde\varepsilon}}+\tfrac{\sigma^2}{\mu^2\tilde\varepsilon}\left(\log\tfrac{LD_0}{\mu\sqrt{\tilde\varepsilon}}\right)^3\right)
\end{equation}
calls to the stochastic oracle.
\end{proposition}

\begin{proof}
For any $x\in X$ and $y\in-N_X(x)$, we have $\langle y,x-x^*\rangle\leq0$,
while the optimality of $x^*$ implies $\langle F(x^*),x-x^*\rangle\geq0$.
Hence, by the strong monotonicity of $F$,
\begin{align*}
\mu\norm{x-x^*}^2
&\leq\langle F(x)-F(x^*),x-x^*\rangle
\leq\langle F(x)-y,x-x^*\rangle
\leq\norm{F(x)-y}\norm{x-x^*}.
\end{align*}
Taking the infimum over $y\in-N_X(x)$ gives $\mu\norm{x-x^*}\leq\res_F(x).$
Therefore, applying \autoref{thm:main-unbiased} with residual tolerance $\varepsilon$ yields
\begin{equation*}
\mathbb{E}[\norm{x_S-x^*}^2]
\leq\tfrac{1}{\mu^2}\mathbb{E}\left[\res_F(x_S)^2\right]
\leq \tfrac{\varepsilon^2}{\mu^2}\coloneqq\tilde\varepsilon.
\end{equation*}
Substituting the sample-complexity bound from \eqref{eqn:final-sample-complexity} in terms of $\tilde \varepsilon$ completes the proof.
\end{proof}
\smallskip

Notice that the stochastic complexity term in \eqref{eqn:mu-free} matches the optimal dependence on the target accuracy established in \citep[Corollary 3.4]{kotsalis2022simple} up to logarithmic factors.  

{In general, it is difficult to certify that the distance to the optimal solution is less than a prescribed tolerance without knowing $\mu$. However, for our method AR, even if we do not know $\mu$, we can use an estimate while preserving the nearly optimal stochastic sample complexity. Specifically, let $\tilde{\varepsilon}$ be the desired mean squared distance accuracy. If we overestimate $\mu$ by $\tilde{\mu}>\mu$, then the targeted residual accuracy is set as $\varepsilon^2\coloneqq\tilde{\mu}^2\tilde{\varepsilon}>\mu^2\tilde{\varepsilon}$, i.e.,  we can only guarantee the mean squared distance accuracy $\tfrac{\tilde{\mu}^2}{\mu^2}\tilde{\varepsilon}>\tilde{\varepsilon}$, which is less accurate than the desired $\tilde{\varepsilon}$. By \autoref{prop:mu-free}, the number of calls to the $\mathcal{SO}$ is bounded by
\begin{equation}\label{eqn:new-mu}
\widetilde{\mathcal{O}}\left(\tfrac{LD_0}{\tilde{\mu}\sqrt{\tilde{\varepsilon}}}+\tfrac{\sigma^2}{\tilde{\mu}^2\tilde{\varepsilon}}\left(\log\tfrac{LD_0}{\tilde{\mu}\sqrt{\tilde{\varepsilon}}}\right)^3\right),
\end{equation}
which is correspondingly lower than the one using the true value of $\mu$ c.f. \eqref{eqn:mu-free}. On the other hand, if we underestimate $\mu$ by $\tilde{\mu}<\mu$, then the targeted residual accuracy satisfies $\varepsilon^2=\tilde{\mu}^2\tilde{\varepsilon}<\mu^2\tilde{\varepsilon}$, i.e., the guaranteed accuracy is higher since $\tfrac{\tilde{\mu}^2}{\mu^2}\tilde{\varepsilon}<\tilde{\varepsilon}$, but the bound on the number of calls to the $\mathcal{SO}$ is also higher according to \eqref{eqn:new-mu}. In both cases, AR converges and remains nearly optimal in its stochastic part. } Thus, using AR, we can obtain an optimal convergence guarantee in terms of the distance to the optimal solution without knowing $\mu.$

{Notice that the above $\mu$-free property does not generally hold for other algorithms that directly control the distance to the optimal solution, for which an incorrect estimate of $\mu$ may cause the algorithm not to converge or to converge extremely slowly. The classical example in \cite{NJLS09-1} illustrates the latter behavior when $\mu$ is overestimated; see also SOE method in \cite{kotsalis2022simple}, when $\mu$ is overestimated, SOE may not converge. When $\mu$ is underestimated, SOE still converges but with a slower convergence rate dictated by the underestimated value $\tilde \mu$ of $\mu.$
}

Observe that the $\mu$-free property above does not recover the optimal
deterministic term for strongly monotone problems. When $\mu$ is known,
however, this limitation can be removed by directly restarting AR. We next
show that the resulting restarted scheme achieves the optimal convergence
rate for strongly monotone VIs and, in particular, recovers the deterministic
linear convergence rate in terms of the distance to the unique solution.

\begin{corollary}\label{thm:main-unbiased-mu}
Assume that $F$ is $L$-Lipschitz continuous and $\mu$-strongly
monotone, and
$0<\varepsilon \le \|x_0-x^*\|^2 \leq D_0^2$. In \autoref{alg:ARVI}, suppose the parameters satisfy
\begin{align}\label{eqn:regularization-exp-unbiased-mu}
         S&=\left\lceil
\log_{16}\!\left(\tfrac{D_0^2}{{\varepsilon}}\right)
\right\rceil,\quad r_s =0.
\end{align}
For each epoch, the epoch length $N_s$ and the batch size $m_s$ satisfy
\begin{align}\label{eqn:batch-epoch-mu}
   N_s&={\left\lceil\tfrac{4\sqrt{2c_{\mathcal{A}}}L}{{\mu}}\right\rceil},\quad
m_s=\max\left\{\left\lceil\tfrac{32\cdot16^{s-1}c_{\mathcal{A}} \sigma^2}{N_s\mu^2 D_0^2}\right\rceil,\, 1\right\}.
\end{align}
Then, at epoch $S$, \autoref{alg:ARVI} computes an approximate solution $x_S$ such that $\mathbb{E}[\|x_S - x^*\|^2]
\le\varepsilon$ after
\begin{align*}
\mathcal{O}\left(\tfrac{L}{\mu}\left\lceil
\log_{16}\!\left(\tfrac{D_0^2}{{\varepsilon}}\right)
\right\rceil+\tfrac{ \sigma^2}{\mu^2\varepsilon}\right)
\end{align*}
calls to the $\mathcal{SO}$.
\end{corollary}
\begin{proof}
We show that for each $s \in [S]$,
\begin{equation}\label{eq:stochastic-inductive-hypothesis-mu}
    \mathbb{E}[\|x_s - x^*\|^2] \leq \tfrac{\|x_0 - x^*\|^2}{16^s},
\end{equation}
where $x_s=x_{s, N_s}$ for all $s\in[S]$.
\smallskip

Since $ r_s =0$, the regularized operator $F_s=F$, and hence $x_s^*=x^*$. Therefore, by \autoref{assumption_inner_sto}, we have
\begin{align}
    \mathbb{E}\left[\|x_{s,k}-x^*\|^2\,\middle|\,\mathcal{F}_{s,0}\right]&\leq \tfrac{c_{\mathcal{A}}L^2}{  {\mu^2}k^2}\|x_{s-1}-x^*\|^2 +\tfrac{c_{\mathcal{A}} \sigma^2}{km_s  {\mu^2}},\label{eqn:distance-s-strong}
\end{align}
Hence, for $s=1$, after $N_1$ iterations, it holds that
 \begin{equation*}
     \begin{aligned}
          \mathbb{E}[\|x_1-x^*\|^2] &\,\,\,\overset{\text{(i)}}{\leq} \tfrac{c_{\mathcal{A}}L^2\|x_{0}-x^*\|^2}{N_1^2  {\mu^2} }+\tfrac{D_0^2}{32}\overset{\text{(ii)}}{\leq}\tfrac{D_0^2}{  {16}},
     \end{aligned}
 \end{equation*}
 where in (i), we used \eqref{eqn:distance-s-strong} and the condition on the batch size $m_s$
 in \eqref{eqn:batch-epoch-mu}; in (ii), we used the condition on the inner iteration number $N_1$ in \eqref{eqn:batch-epoch-mu}.

 Suppose that \eqref{eq:stochastic-inductive-hypothesis-mu} holds at epoch $s-1$. Then, for epoch $s$, similarly, by \eqref{eqn:distance-s-strong}, and the condition on the batch size $m_s$
 in \eqref{eqn:batch-epoch-mu}, it holds that
 \begin{equation*}
     \begin{aligned}
          \mathbb{E}[\|x_s-x^*\|^2] &\leq\tfrac{c_{\mathcal{A}}L^2\mathbb{E}[\|x_{s-1}-x^*\|^2]}{N_s^2  {\mu^2}}+\tfrac{c_{\mathcal{A}} \sigma^2}{N_sm_s  {\mu^2}}\overset{\text{(iii)}}{\leq}\tfrac{c_{\mathcal{A}}L^2D_0^2}{  {16^{s-1}}N_s^2  {\mu^2}}+\tfrac{c_{\mathcal{A}} \sigma^2}{N_sm_s  {\mu^2}}\overset{\text{(iv)}}{\leq}\tfrac{c_{\mathcal{A}}L^2D_0^2}{  {16^{s-1}}N_s^2  {\mu^2}}+\tfrac{1}{2}\cdot\tfrac{D_0^2}{16^s}\overset{\text{(v)}}{\leq}\tfrac{D_0^2}{  16^{s}},
     \end{aligned}
 \end{equation*}
 where in (iii), we used the induction hypothesis;
 in (iv),  we  substituted $m_s$ from \eqref{eqn:batch-epoch-mu}; and in (v)
 we again used the condition on the inner iteration number $N_s$ in \eqref{eqn:batch-epoch-mu}. Therefore, \eqref{eq:stochastic-inductive-hypothesis-mu} holds for all $s\in [S]$.
 \smallskip

By the choices of $N_s,$ and $s\leq S$ hence $16^{s-1}\leq 16^{S-1},$ we can bound the batch size $m_s$
as follows
\begin{equation*}
    \begin{aligned}
        m_s&\leq\tfrac{32\cdot16^{s-1}c_{\mathcal{A}} \sigma^2}{N_s\mu^2 D_0^2}+1\overset{\eqref{eqn:batch-epoch-mu}}{\leq}\tfrac{4\sqrt{2}\cdot16^{s-1}\sqrt{c_{\mathcal{A}}} \sigma^2}{\mu L D_0^2}+1 \overset{\eqref{eqn:regularization-exp-unbiased-mu}}{\leq} \tfrac{4\sqrt{2}\sqrt{c_{\mathcal A}} \sigma^2}{\mu L \varepsilon}+1.
    \end{aligned}
\end{equation*}
Hence, the total sample complexity can be bounded as
\begin{equation*}
    \begin{aligned}
    \tsum_{s=1}^{S}\tsum_{i=1}^{N_s} m_s&\le \tsum_{s=1}^{S}\left[\tfrac{8\cdot16^{s-1}\sqrt{c_{\mathcal{A}}} \sigma^2}{\mu L D_0^2}{\tfrac{1}{\sqrt{2}}}+1\right]\left[\tfrac{4\sqrt{2c_{\mathcal{A}}}L}{{\mu}}+1\right]=\mathcal{O}\left(\tfrac{L}{\mu}\left\lceil
\log_{16}\!\left(\tfrac{D_0^2}{{\varepsilon}}\right)
\right\rceil+\tfrac{ \sigma^2}{\mu^2\varepsilon}\right).
    \end{aligned}
\end{equation*}

\end{proof}
We next show that AR also achieves a near-optimal residual complexity when $F$ is $\mu$-strongly monotone with known $\mu>0$. The algorithm remains unchanged: we only modify the regularization parameters $r_s$ and the number of epochs, while using the same choices of epoch lengths and batch sizes choices as before. The resulting complexity is optimal up to logarithmic factors for strongly monotone problems.
\begin{corollary}
\label{thm:main-unbiased-mu-n}
Assume that $0<\varepsilon \le LD_0$, where $D_0\geq \textnormal{dist}(x_0, X^*)$. In \autoref{alg:ARVI}, let 
\begin{align}\label{eqn:regularization-exp-unbiased-mu-n}
S_1&\coloneqq
1+\left\lceil
\log_4\left(\tfrac{(16c_{\mathcal A}+1)\mu D_0}{\varepsilon}\right)
\right\rceil,
\qquad
S\coloneqq
S_1+\left\lceil\log_4\left(\tfrac{8L}{\mu}\log_4\left(\tfrac{16(16c_{\mathcal A}+4)L}{\mu}\right)\right)\right\rceil.
\end{align}
For each epoch $s$, the regularization parameter $r_s$, the epoch length $N_s$, and the batch size $m_s$ are given by 
\begin{align}\label{eqn:batch-epoch-mu-n}
 r_s=
\begin{cases}
0, & s\leq S_1,\\[1mm]
\tfrac{4^{s-S_1}\mu}{8\log_4\left(\tfrac{16(16c_{\mathcal A}+4)L}{\mu}\right)}, & s>S_1,
\end{cases}\quad N_s=\left\lceil\tfrac{4\sqrt{2c_{\mathcal{A}}}(L+ r_s )}{r_s+\mu}\right\rceil,\quad
m_s=\left\lceil\tfrac{16^{s-1}N_s \sigma^2}{(L+ r_s )^2 D_0^2}+1\right\rceil.
\end{align}
Then, at epoch $S$, \autoref{alg:ARVI} computes an approximate solution $x_S$ such that $\sqrt{\bbe\left[\res_F(x_S)^2\right]}
\le\varepsilon$ after
\begin{align}\label{eqn:final-sample-complexity-mu-n}
{\mathcal{O}}\left(\tfrac{L}{\mu}\left(\log\tfrac{\mu D_0}{\varepsilon}+\log \log \tfrac{L}{\mu}\right)+\log \tfrac{L}{\mu}+\tfrac{\sigma^2}{\varepsilon^2}\left(\log\tfrac{L}{\mu}\right)^3\right)
\end{align}
calls to the $\mathcal{SO}$.
\end{corollary}
\begin{proof}
We first show that for each $s \in [S]$, $ \mathbb{E}[\|x_s - x_s^*\|^2] \leq \tfrac{D_0^2}{16^s},$
where $x_s=x_{s, N_s}.$ This is the same as \autoref{thm:main-unbiased}, thus omitted for simplicity.
The residual decomposition is also the same, it remains to bound the residual as follows.
\begin{equation*}
    \begin{aligned}
   \sqrt{\mathbb{E}\left[\res_F(x_S)^2\right]}
&\leq\tfrac{c_{\mathcal{A}} r_S  D_0}{4^{S-2}}+ \tsum_{s = 1}^S \tfrac{( r_s  +  r_{s - 1})D_0}{4^{s - 1}}+ \tfrac{ r_S  D_0}{4^S}\leq\tfrac{\left(16c_{\mathcal{A}}+1\right) r_S  D_0}{4^{S}}+ {\tfrac{5D_0}{4}}\tsum_{s=1}^{S}\tfrac{ r_s }{  {4}^{s-1}},
    \end{aligned}
\end{equation*}
Denote $S_0\coloneqq\log_4\left(\tfrac{16(16c_{\mathcal A}+4)L}{\mu}\right),$
since $S_0\geq3$, the choices of $S$ and $S_1$ imply
\begin{align*}
S-S_1
&\leq 1+\log_4\left(\tfrac{8S_0L}{\mu}\right)
=\tfrac{5}{2}+\log_4S_0+\log_4\tfrac{L}{\mu}\leq \tfrac{5}{2}+2S_0
\leq 4S_0,
\end{align*}
where we used $\log_4S_0\leq S_0$ and
$\log_4(L/\mu)\leq S_0$. Hence,
\begin{align*}
\tfrac{(16c_{\mathcal A}+1)r_SD_0}{4^S}
&=
\tfrac{(16c_{\mathcal A}+1)\mu D_0}{C4^{S_1}}
\leq\tfrac{\varepsilon}{4C}
\leq\tfrac{\varepsilon}{4},\\
\tsum_{s=1}^S\tfrac{r_s}{4^{s-1}}
&\leq
\tfrac{4(S-S_1)\mu}{C4^{S_1}}
\leq
\tfrac{16S_0\mu}{C4^{S_1}}
=
\tfrac{2\mu}{4^{S_1}}
\leq
\tfrac{\varepsilon}{2(16c_{\mathcal A}+1)D_0}
\leq
\tfrac{\varepsilon}{2D_0},
\end{align*}
It remains to prove $ r_S  \ge L$. By the choice of $S$,
\begin{align*}
r_S
=
\tfrac{4^{S-S_1}\mu}{8\log_4\left(\tfrac{16(16c_{\mathcal A}+4)L}{\mu}\right)}
\geq L.
\end{align*}
The total sample complexity can be bounded as
\begin{equation*}
    \begin{aligned}
       \tsum_{s=1}^{S}{N_s} m_s&=  \tsum_{s=1}^{S_1}{N_s} m_s+\tsum_{s=S_1+1}^{S}\tsum_{i=1}^{N_s} m_s\\
&\le \tsum_{s=1}^{S_1}\left[2+\tfrac{16^{s-1}N_s \sigma^2}{L^2D_0^2}\right]\left[1+\tfrac{4\sqrt{2c_{\mathcal{A}}}L}{\mu}\right]+\tsum_{s=S_1+1}^{S}\left[2+\tfrac{16^{s-1}N_s \sigma^2}{(L+ r_s )^2D_0^2}\right]\left[1+\tfrac{4\sqrt{2c_{\mathcal{A}}}(L+ r_s )}{ r_s +\mu}\right]\\
&=\mathcal{O}\left(\tfrac{L}{\mu}S_1+S_0+\tfrac{L}{\mu}\log S_0+\tfrac{\sigma^2}{\varepsilon^2}S_0^3\right)\\
&={\mathcal{O}}\left(\tfrac{L}{\mu}\left(\log\tfrac{\mu D_0}{\varepsilon}+\log \log \tfrac{L}{\mu}\right)+\log \tfrac{L}{\mu}+\tfrac{\sigma^2}{\varepsilon^2}\left(\log\tfrac{L}{\mu}\right)^3\right).
    \end{aligned}
\end{equation*}

\end{proof}
\section{Stochastic Variational Inequalities under State Dependent Noise}\label{sec:Stochastic Variational Inequalities under State-Dependent Noise}

In this section, we study the problem of finding a strong solution of the
monotone VI~\eqref{VIP}, where $F$ satisfies \eqref{eq:Lipschitz} and
\eqref{strongly_monotone} with $\mu\geq 0$, and only stochastic evaluations
satisfying the state-dependent variance condition are available. In \autoref{sec:Assumptions}, we introduce the assumptions imposed on the stochastic operators. In \autoref{sec:An accumulative regularization method for stochastic monotone VI}, we establish the convergence rate of the AR method when equipped with a generic subroutine $\mathcal{A}$ for monotone VIs. In \autoref{sec:Stochastic strongly monotone VI}, we provide a concrete instantiation of this subroutine, namely stochastic operator extrapolation, and analyze its convergence rate for strongly monotone VIs. We also establish an improved convergence guarantee for stochastic operator extrapolation (SOE) \cite{kotsalis2022simple2} for strongly monotone VIs under state-dependent noise.

\subsection{Assumptions}\label{sec:Assumptions}

Our goal in this subsection is to present assumptions on the stochastic estimators that characterize their accuracy
and play an essential role in the design and analysis of the algorithms for solving the stochastic monotone VIs
developed in subsequent sections.

We assume that all data $\{\zeta_k\}$ are i.i.d. random variables. Throughout this section, we impose the state-dependent noise assumption \eqref{eqn:state-dependent-assumption} and unbiased estimator assumption \eqref{eq:fast-decreasing-bias-o}.

The variance condition  in \eqref{eqn:state-dependent-assumption} is state-dependent: the right-hand side depends on the queried point $x$, rather than on a uniform constant, which is  commonly assumed in the stochastic VI literature. In particular, the variance is controlled by the quantity $\varsigma_*\langle F(x)-F(x^*),x-x^*\rangle$. Since $F$ is monotone, this quantity is nonnegative and is motivated by the strong gap of the VI.
Furthermore, for any two solutions $\bar{x}$ and $x^*$ of a monotone VI, monotonicity and their corresponding VI imply $\langle F(\bar{x})-F(x^*),\,\bar{x}-x^*\rangle=0.$
Thus, the state-dependent term vanishes even when the solution is not unique.

Unlike distance-based state-dependent noise conditions expressed in terms of $\|x-x^*\|^2$ \cite{kotsalis2022simple2,alacaoglu2025towards}, this new formulation does not require uniqueness of the solution, and hence does not rely on strong monotonicity, while remaining meaningful for merely monotone VIs.  It is new to the literature and is also the key structural assumption underlying our sharper sample-complexity bounds and, in the strongly monotone setting, an improved convergence guarantee for stochastic operator extrapolation (SOE) \cite{kotsalis2022simple2}.

In the application section (see \autoref{sec:Applications: Online Policy Evaluation via Variational Inequalities}), we show that these assumptions are readily satisfied. In particular, we consider an i.i.d.\ sampling model in which a black-box simulator generates samples from the transition kernel of the reinforcement learning problem.

\subsection{Accumulative regularization for stochastic monotone VIs}\label{sec:An accumulative regularization method for stochastic monotone VI}
In this subsection, we analyze AR with subroutine $\mathcal{A}$ for computing a strong solution
of the monotone VI~\eqref{VIP}, where $F$ satisfies \eqref{eq:Lipschitz} and
\eqref{strongly_monotone} with $\mu\ge0$, under the assumptions stated in
\autoref{sec:Assumptions}.

Recall that the regularized operator $F_s$ is defined in \eqref{eq:Fs}. We define its stochastic counterpart as
\begin{equation}\label{eqn:regulairzed-stochastic-estimator}
   \widehat{F}_{s,k}(x)\coloneqq \tilde{F}_{s,k}(x) +  r_s  (x - \bar x_s),
\end{equation}
where $\tilde{F}_{s,k}(x)$ is defined in \eqref{eqn:minibatch}.
We assume that the stochastic estimator $\tilde{F}_{s,k}$ satisfies the assumptions in \autoref{sec:Assumptions}.
Furthermore, recall that $\mathcal{F}_{s,k}$ denote the filtration generated by the global trajectory up to, but not including, the samples used to form $\tilde{F}_{s,k}$.
Hence, by the definition of $\mathcal{F}_{s,k}$, both $\bar{x}_s$ and $x_s^*$ are $\mathcal{F}_{s,k}$-measurable.
\medskip

We start by exploring the properties of the stochastic operator $\widehat{F}_{s,k}$ under the assumptions in \autoref{sec:Assumptions}.
\begin{lemma}\label{lemma:state-dependent-noise-regularized}
Suppose assumptions \eqref{eq:fast-decreasing-bias-o} and \eqref{eqn:state-dependent-assumption}. Let $x_s^*$ be a solution of the regularized VI associated with $F_s$, and let $x^*\in X^*$. Then the following properties hold almost surely.
\begin{itemize}
    \item[\textendash] {\it Unbiasedness:} for any $\mathcal{F}_{s,k}$-measurable point $x \in X$,
    \begin{align}\label{eq:fast-decreasing-bias}
         \mathbb{E}[\widehat{F}_{s,k}(x)\mid \mathcal{F}_{s,k}]=F_s(x).
    \end{align}
    \item[\textendash] {\it Variance at the regularized solution:}
    \begin{align}\label{eq:bounded-variance-at-opt-regularized}
    \mathbb E[\|\widehat{F}_{s,k}(x_{s}^*) - F_s(x_{s}^*)\|^2\mid\mathcal{F}_{s,k}]
    \leq \tfrac{\varsigma_{*}  r_s }{m_s}\|x_s^* - \bar x_s\| \|x_s^* - x^*\| + \tfrac{\sigma_*^2}{m_s}.
    \end{align}
\end{itemize}
\end{lemma}
\begin{proof}
By the definitions of $F_s$ and $\widehat{F}_{s,k}$,
it holds that
\begin{align*}
        \mathbb E[\widehat{F}_{s,k}(x)\mid\mathcal{F}_{s,k}] &= \mathbb E[\tilde{F}_{s,k}(x)\mid\mathcal{F}_{s,k}] +  r_s  (x - \bar x_s),\quad
        F_s(x)= F(x)+  r_s  (x - \bar x_s).
\end{align*}
Therefore, it holds that
\begin{equation*}
    \begin{aligned}
         &\mathbb{E}[\widehat{F}_{s,k}(x)\mid \mathcal{F}_{s,k}]-F_s(x)=\mathbb{E}[\tilde{F}_{s,k}(x)\mid \mathcal{F}_{s,k}]-F(x)\overset{\text{(i)}}{=}0,
    \end{aligned}
\end{equation*}
where in (i), we used \eqref{eq:fast-decreasing-bias-o}.
Furthermore, since $x_s^*$ and $x^*$ are $\mathcal{F}_{s,k}$-measurable, and $\bar x_s$ is also $\mathcal{F}_{s,k}$-measurable by its definition in \eqref{eqn:center}, it holds that
\begin{align*}
    \mathbb E[\|x_s^* - \bar x_s\| \|x^* - x^*_s\|\mid\mathcal{F}_{s,k}]
    =\|x_s^* - \bar x_s\| \|x^* - x^*_s\|.
\end{align*}
Therefore, by the definitions of $F_s$ in \eqref{eq:Fs}, $\tilde{F}_{s,k}$ in \eqref{eqn:minibatch}, $\widehat{F}_{s,k}$ in \eqref{eqn:regulairzed-stochastic-estimator},
it holds that
\begin{align*}
&\,\,\mathbb{E}\big[\|\widehat{F}_{s,k}(x_{s}^*) - F_s(x_{s}^*)\|^2\mid \mathcal{F}_{s,k}\big]\\
       &\,\,=\mathbb E[\|\tilde{F}_{s,k}(x_{s}^*) - F(x_s^*)\|^2\mid\mathcal{F}_{s,k}] \\
        &\,\,\overset{\text{(ii)}}{\leq} \tfrac{\varsigma_{*}}{m_s}\langle F(x_s^*) - F(x^*), x^*_s - x^*\rangle + \tfrac{ \sigma^2_*}{m_s} \\
        &\,\overset{\eqref{eq:Fs}}{=} \tfrac{\varsigma_{*}}{m_s} \langle - r_s (x_s^* - \bar x_s),  x^*_s - x^*\rangle  + \tfrac{\varsigma_{*}}{m_s}\langle F(x^*), x^*-x^*_s \rangle+\tfrac{\varsigma_{*}}{m_s}\langle F_s(x_s^*), x^*_s - x^*\rangle + \tfrac{ \sigma^2_*}{m_s} \\
        &\,\,\overset{\text{(iii)}}{\leq} \tfrac{ \varsigma_{*}}{m_s} \langle - r_s (x_s^* - \bar x_s), x^*_s - x^*\rangle + \tfrac{ \sigma^2_*}{m_s} \\
        &\,\,\,\,\leq \tfrac{\varsigma_{*} r_s }{m_s}  \|x^*_s - \bar x_s\| \|x_s^* - x^*\| + \tfrac{ \sigma^2_*}{m_s}.
\end{align*}
Here, in (ii), we used the state-dependent variance assumption \eqref{eqn:state-dependent-assumption}, the definition of $\tilde{F}_{s,k}$, and the minibatch property of the i.i.d.\ samples;
 in (iii), we used the definition of $x^*$ as a solution of the VI problem \eqref{VIP}, namely
$\langle F(x^*), x^* - x_s^* \rangle \leq 0$, and the fact that $x_s^*$
solves the regularized VI problem associated with $F_s$ in \eqref{proxVIP}, namely
$\langle F_s(x_s^*), x_s^* - x^*\rangle \leq 0$.
Therefore, \eqref{eq:bounded-variance-at-opt-regularized} holds.
This concludes the proof.
\smallskip

\end{proof}

In view of \eqref{eq:bounded-variance-at-opt-regularized}, we define the variance term of the regularized operator at the regularized solution $x_s^*$ as
\begin{align}\label{eqn:delta_s}
     \sigma^2_{s,*}\coloneqq {\varsigma_{*} r_s } \|x^*_s - \bar x_s\| \|x_s^* - x^*\| + { \sigma^2_*},
\end{align}
which incorporates the regularization through $ r_s $ and the distances $\|x_s^* - \bar x_s\|$ and $\|x_s^* - x^*\|$, the state-dependent noise level $\varsigma_*$, and the variance bound $\sigma_*^2$ at the solution $x^*$. Observe that $ \sigma^2_{s,*}$ is random, as it depends on the prox-center $\bar x_s$ and the regularized solution $x_{s}^*$; nevertheless, it is immediate that $ \sigma^2_{s,*}$ is $\mathcal{F}_{s,0}$-measurable.
\vgap

We make the following assumption regarding the convergence of the subroutine $\mathcal A (\mathcal{SO},  r_s , \bar{x}_{s}, x_{s-1}, m_s)$.
\begin{assumption}\label{assumption_inner_sto-state-dependent}
The subroutine
$\mathcal A(\mathcal{SO},  r_s , \bar{x}_s, x_{s-1}, m_s)$
for solving the regularized problem \eqref{eq:sub-problem}
generates iterates satisfying the following performance guarantees for any
mini-batch size
$m_s\geq \tfrac{\tilde{c}_{\mathcal A}\varsigma_*}{L+ r_s }$
and all $k\geq 2$:
\begin{align}
    \mathbb{E}\left[\|x_{s,k}-x_s^*\|^2\,\middle|\,\mathcal{F}_{s,0}\right]&\leq \tfrac{c_{\mathcal{A}}(L+ r_s )^2}{(  \mu+r_s )^2k^2}\|x_{s-1}-x_s^*\|^2 +\tfrac{c_{\mathcal{A}} \sigma^2_{s,*}}{km_s(  \mu+r_s )^2},\label{eqn:distance-s-state-dependent}\\
 \sqrt{\mathbb{E}\left[\res_{F_s}(x_{s,k})^2\,\middle|\,\mathcal{F}_{s,0}\right]}&\leq {c_{\mathcal{A}}(L+ r_s )}\|x_{s-1}-x_s^*\|+\tfrac{c_{\mathcal{A}} \sigma_{s,*}\sqrt{k}}{\sqrt{m_s}},\label{eqn:residual-s-state-dependent}
\end{align}
where $c_{\mathcal{A}}$ is a constant depending only on the subroutine $\mathcal A$.
\end{assumption}
Similar to \autoref{assumption_inner_sto}, the deterministic terms in the
two guarantees remain unchanged, and the stochastic error terms have the
same dependence on the number of iterations and the mini-batch size. The
main differences lie in the treatment of the stochastic oracle noise.
First, in \autoref{assumption_inner_sto}, the stochastic terms are
controlled by a uniform noise bound $\sigma$, whereas
\eqref{eqn:distance-s-state-dependent} and
\eqref{eqn:residual-s-state-dependent} depend on the effective noise level
$ \sigma_{s,*}$ at the solution $x_s^*$ of the regularized subproblem.
Second, while the guarantees in the uniform-noise setting hold for any
mini-batch size, the guarantees under the state-dependent noise model are
imposed only for mini-batch sizes satisfying
$m_s\geq \tilde{c}_{\mathcal A}\varsigma_*/(L+ r_s )$. This condition ensures that the state-dependent variance does not dominate the deterministic Lipschitz-dependent terms.

Although stochastic operator extrapolation (SOE)
\cite{kotsalis2022simple} provides an example of a subroutine satisfying
guarantees of this type under the distance-dependent noise condition
\eqref{state-dependent-original}, its performance under the new
state-dependent noise assumption
\eqref{eqn:state-dependent-assumption} is unknown. We therefore refine the
analysis of SOE under this new assumption and prove that it satisfies
\autoref{assumption_inner_sto-state-dependent}
(cf. \autoref{prop:subroutine}).

It remains to control the variance proxy $ \sigma^2_{s,*}$ in \eqref{eqn:delta_s}. Unlike the uniform variance case, this quantity may grow with the regularization level $ r_s $ through the term
$ r_s \|x_s^*-\bar x_s\|\|x_s^*-x^*\|$.
The next lemma shows that this growth remains controlled. Specifically, if the approximate solutions from previous epochs remain geometrically close to their corresponding regularized solutions, then $  \sigma^2_{s,*}$ remains bounded by the accumulative regularization parameters, $\sigma_*$ and $\varsigma_*.$

\begin{lemma}\label{lem:state-dependent-iterate-boundedness}
    Suppose that $\mathbb E[\|x_j - x_j^*\|^2] \leq \tfrac{D_0^2}{16^j}$ for all $j\in\{0,\ldots,s-1\}$. Then, for $s\geq 2$, it holds that
\begin{equation}\label{eqn:boundedness}
            \mathbb{E}[ \sigma^2_{s,*}]
   \leq  {\tfrac{8\varsigma_*D_0^2}{3}
   \,\tsum_{i = 1}^s
   {\tfrac{ r_{i} +  r_{i-1}}{4^{i-1}}}}
   +\sigma_*^2.
    \end{equation}
\end{lemma}
\begin{proof}
By \autoref{lemma_ARVI_proximity}, \eqref{eqn:induction-traingle}, we obtain
\begin{align}\label{induction-0}
    \mathbb E[ \|x_{s}^* - x^*\| ]&\leq 2 \tsum_{k = 0}^{s - 1} \mathbb E[\|x_{k} - x_k^*\|]\leq 2\tsum_{k = 0}^{s - 1}\tfrac{D_0}{4^k}\leq \tfrac{8D_0}{3}.
\end{align}
 {Hence,  we have}
{
\begin{align}\label{eqn:boundednedd-regulairzed-true}
   \sqrt{\mathbb E[\|x_s^* - x^*\|^2]}
    &\overset{\eqref{eqn:induction-traingle}}{\leq}
    \sqrt{\mathbb E
    \left[
    \left(
    \tsum_{k=0}^{s-1}\|x_k-x_k^*\|
    \right)^2
    \right]}\overset{\text{(i)}}{\leq}
    \tsum_{k=0}^{s-1}
    \sqrt{\mathbb E
    [\|x_k-x_k^*\|^2]}\overset{\text{(ii)}}{\leq}
    \tsum_{k=0}^{s-1}\tfrac{D_0}{4^k}
    \leq \tfrac{4D_0}{3},
\end{align}
where in (i), we used Minkowski's inequality; and in (ii), we used Jensen's inequality and \eqref{induction-0}.
}
\smallskip

Similarly, by \autoref{lemma_ARVI_proximity}
 {and Minkowski's inequality}, we obtain
{
\begin{equation*}
\begin{aligned}
      r_s
    \sqrt{\mathbb E[\|x_s^*-\bar{x}_s\|^2]}
    &\overset{\eqref{eq:epoch_proximity2}}{\leq}
    \sqrt{\mathbb E\left[
    \left(
    \tsum_{i=1}^s( r_{i}+ r_{i-1})
    \|x_{i-1}^*-x_{i-1}\|
    \right)^2
    \right]}\\
    &\,\,\overset{\text{(iii)}}{\leq}
    \tsum_{i = 1}^s( r_{i}+ r_{i-1})
    \sqrt{\mathbb E
    [\|x^*_{i-1}-x_{i-1}\|^2]},
\end{aligned}
\end{equation*}
where in (iii), we used again the Minkowski's inequality.
}
Hence, by substituting the condition
$\mathbb E[\|x_j - x_j^*\|^2] \leq \tfrac{D_0^2}{16^j}$
for all $j\in\{0,\ldots,s-1\}$, we have
{
\begin{equation}\label{eqn:bias-2'}
\begin{aligned}
     r_s
    \sqrt{ \mathbb E[\|x_s^*-\bar{x}_s\|^2]}
    &\leq D_0\tsum_{i = 1}^s
    \tfrac{ r_{i}+ r_{i-1}}{4^{i-1}}.
\end{aligned}
\end{equation}
}
To bound $ \sigma_{s,*}$ in
\eqref{eqn:delta_s},
by Cauchy--Schwarz inequality, we have
\begin{align}\label{eqn:bias-0'}
    \mathbb E\left[\|x_s^* - \bar x_s\| \|x_s^* - x^*\|\right]
    &\leq  \sqrt{\mathbb E\left[\|x_s^* - \bar x_s\|^2\right]}
   \sqrt{\mathbb E\left[\|x_s^* - x^*\|^2\right]}
    \overset{\eqref{eqn:boundednedd-regulairzed-true}}{\leq}
     {\tfrac{4D_0}{3}
    \sqrt{\mathbb E\left[\|x_s^* - \bar x_s\|^2\right]}}.
\end{align}
Hence, substituting \eqref{eqn:bias-2'} into \eqref{eqn:bias-0'}, we have
\begin{equation*}
\begin{aligned}
    \varsigma_* r_s
    \mathbb E\left[\|x_s^* - \bar x_s\| \|x_s^* - x^*\|\right]
    &\leq
     {\tfrac{8}{3}\varsigma_* r_s D_0\sqrt{
    \mathbb E
    \left[\|x_s^*-\bar x_s\|^2\right]} }{\leq
    \tfrac{8}{3}\varsigma_*D_0^2
    \tsum_{i = 1}^s
    \tfrac{ r_{i}+ r_{i-1}}{4^{i-1}}}.
\end{aligned}
\end{equation*}
This concludes the proof.
\end{proof}
\smallskip

We are now ready to state the convergence guarantee for the {AR method} in \autoref{alg:ARVI} with $\mu\geq 0$ measured in terms of the residual.  
\begin{theorem}\label{thm:ARVI-state-dependent}
Assume that $\varepsilon \le LD_0$, where
$D_0\geq \textnormal{dist}(x_0, X^*)$. In \autoref{alg:ARVI},
suppose $S$ and $\{ r_s \}_{s=1}^S$ satisfy
\eqref{eqn:regularization-exp-unbiased}, and
$\{N_s\}_{s=1}^S$ satisfies \eqref{eqn:batch-epoch}. For each
$s\in[S]$, choose $m_s$ according to
\begin{equation}\label{eqn:regularization-exp-unbiased-state-dependent}
m_s
=
\left\lceil
\max\left\{
\tfrac{16^{s-1}N_s}{(L+ r_s )^2D_0^2}
\left(
\tfrac{8\varsigma_*}{3}D_0^2
\tsum_{i=1}^s
\tfrac{ r_{i}+ r_{i-1}}{4^{i-1}}
+\sigma_*^2
\right)
+1,\,
\tfrac{\tilde{c}_{\mathcal A}\varsigma_*}{L+ r_s }
\right\}
\right\rceil.
\end{equation}
Then, at epoch $S$, \autoref{alg:ARVI} computes an approximate 
solution $x_S$ satisfying $\sqrt{\mathbb{E}\left[\res_F(x_S)^2\right]}
\le\varepsilon$ after
\begin{align}\label{eqn:state-dependent-final}
\widetilde{\mathcal{O}}\left(
\tfrac{LD_0}{\varepsilon}
+
\tfrac{\varsigma_* D_0}{\varepsilon}
\left(\log\tfrac{LD_0}{\varepsilon}\right)^{ {3}}
+
\tfrac{\sigma_*^2}{\varepsilon^2}
\left(\log\tfrac{LD_0}{\varepsilon}\right)^3
\right)
\end{align}
calls to the $\mathcal{SO}$.
\end{theorem}
\begin{proof}
The proof follows the same induction as that of
\autoref{thm:main-unbiased}, with the uniform noise bound $ \sigma^2$
replaced by the upper bound on
$\mathbb E[ \sigma_{s,*}^2]$ established in
\autoref{lem:state-dependent-iterate-boundedness}. In particular, by \autoref{assumption_inner_sto-state-dependent}, we have
 \begin{equation}\label{eqn:induction-thm3}
     \begin{aligned}
          \mathbb E\left[\|x_s-x_s^*\|^2\right]
&\leq
\tfrac{c_{\mathcal A}(L+ r_s )^2}
{(  \mu+r_s )^2N_s^2}
\mathbb E\left[\|x_{s-1}-x_s^*\|^2\right]
+
\tfrac{c_{\mathcal A}\mathbb E[ \sigma_{s,*}^2]}
{N_sm_s(  \mu+r_s )^2}.
     \end{aligned}
 \end{equation}
By the batch size  condition \eqref{eqn:regularization-exp-unbiased-state-dependent} and \autoref{lem:state-dependent-iterate-boundedness}, we have
\begin{align*}
\tfrac{c_{\mathcal A}\mathbb E[ \sigma_{s,*}^2]}
{N_sm_s(  \mu+r_s )^2}
&\leq
\tfrac{c_{\mathcal A}(L+ r_s )^2D_0^2}
{16^{s-1}N_s^2(  \mu+r_s )^2}.
\end{align*}
Substituting this bound into \eqref{eqn:induction-thm3}, the
remaining steps follow similarly and are omitted for brevity.
It remains to specify the resulting sample complexity. Observe that
\begin{equation*}
    \begin{aligned}
      \tfrac{\tfrac{8\varsigma_*}{3}D_0^2\,\tsum_{i = 1}^s \tfrac{ r_{i} +  r_{i-1}}{4^{i-1}} + \sigma_*^2}{\varepsilon^2}\overset{\textnormal{(i)}}{=} \mathcal{O}\left( \tfrac{\varsigma_* D_0}{\varepsilon}+ \tfrac{\sigma_*^2}{\varepsilon^2} \right).
    \end{aligned}
\end{equation*}
Here, in $\textnormal{(i)}$, we used the definition of $ r_{i}$ in \eqref{eqn:regularization-exp-unbiased}, which gives $ \tsum_{i = 1}^s \tfrac{ r_{i} +  r_{i-1}}{4^{i-1}} = \mathcal{O}\left( \tfrac{\varepsilon}{D_0} \right).$
Hence, the total sample complexity reads
\begin{equation*}
    \begin{aligned}
\tsum_{s=1}^{S}N_sm_s
&\le \tsum_{s=1}^{S}\left[2+\tfrac{16^{s-1}N_s}{(L+ r_s )^2D_0^2}\cdot\left(  {\tfrac{8\varsigma_*}{3}D_0^2\,\tsum_{i = 1}^s \tfrac{ r_{i} +  r_{i-1}}{4^{i-1}} + \sigma_*^2}\right)+\tfrac{\tilde{c}_{\mathcal A}\varsigma_*}{L+ r_s }\right]\left[1+\tfrac{4\sqrt{2c_{\mathcal{A}}}(L+ r_s )}{ r_s }\right]\\
&=\mathcal{O}\left(
\tsum_{s=1}^{S}
\left[
\tfrac{L+ r_s }{ r_s }
+\tfrac{\varsigma_*}{ r_s }
+\tfrac{16^{s}}{ r_s ^2D_0^2}
\left(
\varsigma_*D_0^2 \tsum_{i = 1}^s \tfrac{ r_{i} +  r_{i-1}}{4^{i-1}}+\sigma_*^2
\right)
\right]
\right)\\
&=
\widetilde{\mathcal{O}}\left(
\tfrac{LD_0}{\varepsilon}
+
\tfrac{\varsigma_* D_0}{\varepsilon}
\left(\log\tfrac{LD_0}{\varepsilon}\right)^{ {3}}
+
\tfrac{\sigma_*^2}{\varepsilon^2}
\left(\log\tfrac{LD_0}{\varepsilon}\right)^3
\right).
\end{aligned}
\end{equation*}
\end{proof}

\smallskip

By Jensen's inequality,
$\mathbb{E}\left[\res_F(x_S)\right]\leq\sqrt{\mathbb{E}\left[\res_F(x_S)^2\right]}$.
Consequently, the sample complexity guarantee in \autoref{thm:ARVI-state-dependent} for finding a point $x_S$ satisfying $\mathbb{E}\left[\res_F(x_S)\right]\leq\varepsilon$ is
$\widetilde{\mathcal O}\bigl(\sigma_*^2\varepsilon^{-2}(\log\tfrac{1}{\varepsilon})^3\bigr)$.
This improves upon the $\mathcal O(\varepsilon^{-4})$ sample complexity established in \citep[Theorem~3.23]{kotsalis2022simple2} for achieving a small expected residual under the distance-based state-dependent noise assumption \eqref{state-dependent-original}.
\begin{remark}\label{remark:choice-reference-solution}
The convergence guarantee holds for any $x^*\in X^*$. Hence, the reference solution can be selected according to the dominant term in the sample complexity bound. If the initial-distance term dominates, one may choose
$x_{\mathrm{dist}}^*\in\argmin_{x\in X^*}\|x_0-x\|$, so that
$D_0\geq\operatorname{dist}(x_0,X^*)$. If the variance term dominates, one may instead choose
$x_{\mathrm{var}}^*\in\argmin_{x\in X^*}\sigma_*^2(x)$
and take $D_0\geq\|x_0-x_{\mathrm{var}}^*\|$. This choice trades a potentially larger initial distance for a smaller variance.
\end{remark}
\smallskip

 As in the uniform-noise setting, AR can be applied as a general monotone VI method without exploiting strong monotonicity. For a prescribed residual tolerance, its implementation does not require knowledge of $\mu$; $\mu$ enters only when translating residual accuracy into distance accuracy.
\begin{remark}
 Suppose the conditions of \autoref{thm:ARVI-state-dependent} hold and $F$ is $\mu$-strongly monotone. Setting the residual tolerance to $\mu\sqrt{\varepsilon}$, AR computes $x_S$ such that $\mathbb{E}[\|x_S-x^*\|^2]\leq\varepsilon$ after
\begin{equation*}
\widetilde{\mathcal{O}}\left(
\tfrac{LD_0}{\mu\sqrt{\varepsilon}}
+\tfrac{\varsigma_*D_0}{\mu\sqrt{\varepsilon}}
\left(\log\tfrac{LD_0}{\mu\sqrt{\varepsilon}}\right)^3
+\tfrac{\sigma_*^2}{\mu^2\varepsilon}
\left(\log\tfrac{LD_0}{\mu\sqrt{\varepsilon}}\right)^3
\right)
\end{equation*}
calls to the stochastic oracle.
\end{remark}

In the following, we show that AR can also be used to derive the optimal convergence rate for strongly monotone problems ($\mu>0$) under the state-dependent noise model.
\begin{corollary}\label{state-dependent-thm-strongly-monotone}
Assume that $F$ is $L$-Lipschitz continuous and $\mu$-strongly
monotone, and
$0<\varepsilon \le \|x_0-x^*\|^2 \leq D_0^2$. In \autoref{alg:ARVI}, suppose the parameters satisfy \eqref{eqn:regularization-exp-unbiased-mu}.
For each epoch, the epoch length $N_s$ satisfies \eqref{eqn:batch-epoch-mu} and the batch size $m_s$ satisfy
\begin{align}\label{eqn:batch-epoch-mu-state}
m_s=\max\left\{\left\lceil\tfrac{32\cdot16^{s-1}c_{\mathcal{A}} \sigma^2_*}{N_s\mu^2 D_0^2}\right\rceil,\, \tfrac{\tilde{c}_{\mathcal A}\varsigma_*}{L},\,1\right\}.
\end{align}
Then, at epoch $S$, \autoref{alg:ARVI} computes an approximate solution $x_S$ such that $\mathbb{E}[\|x_S - x^*\|^2]
\le\varepsilon$ after
\begin{align}\label{eq:FTD_Complexity_error}
\mathcal{O}\left(\tfrac{L+{\varsigma_*}}{\mu}\left\lceil 1+
\log_{16}\!\left(\tfrac{D_0^2}{{\varepsilon}}\right)
\right\rceil+\tfrac{ \sigma^2_*}{\mu^2\varepsilon}\right)
\end{align}
calls to the $\mathcal{SO}$.
\end{corollary}
\begin{proof}
The proof follows the same induction as that of \autoref{thm:main-unbiased-mu}, after replacing $ \sigma^2$ by a deterministic upper bound on $\sigma_*^2$ in \autoref{lem:state-dependent-iterate-boundedness} with $ r_{i}=0,$ for all $i\in[s]$, and is therefore omitted for brevity.
The total sample complexity can be bounded as
\begin{equation*}
    \begin{aligned}
       \tsum_{s=1}^{S}\tsum_{i=1}^{N_s} m_s
&\le \tsum_{s=1}^{S}\left[\tfrac{8\cdot16^{s-1}\sqrt{c_{\mathcal{A}}} \sigma^2_*}{\mu L D_0^2}{\tfrac{1}{\sqrt{2}}}+\tfrac{\tilde{c}_{\mathcal A}\varsigma_*}{L}+1\right]\left[\tfrac{4\sqrt{2c_{\mathcal{A}}}L}{{\mu}}+1\right]=\mathcal{O}\left(\tfrac{L+{\varsigma_*}}{\mu}\left\lceil 1+
\log_{16}\!\left(\tfrac{D_0^2}{{\varepsilon}}\right)
\right\rceil+\tfrac{ \sigma^2_*}{\mu^2\varepsilon}\right).
    \end{aligned}
\end{equation*}
\end{proof}
This result improves the distance-to-solution guarantee for strongly monotone operators established in \cite{kotsalis2022simple2}.
In the latter result (cf. \citep[Corollary 3.6]{kotsalis2022simple2}), under the distance-based state-dependent noise assumption \eqref{state-dependent-original}, the sample complexity to find a solution $\widehat{x} \in X$ such that $\bbe\left[
\|\widehat{x}-x^*\|^2
\right]
\leq\varepsilon$ is
$$
\mathcal{O}\{ \max( \tfrac{\max\{L, \tilde{L}\}^2+\varsigma_*^2}{\mu^2}
\log{\tfrac{\|x_0-x^*\|}{\varepsilon}} , \tfrac{\sigma_*^2+\|F(x^*)\|^2}{\mu^2 \varepsilon} \log{\tfrac{1}{\varepsilon}}) \},
$$
where it depends quadratically on the condition number (also involving the potentially larger sample $L$-Lipschitz continuity parameter $\tilde{L}$), and also depends on the state-dependent noise parameter $\varsigma$ and on $\| F(x^*)\|$, and therefore does not capture the benefit of incorporating variance-reduction mechanisms. In contrast, even in the presence of the state-dependent noise, the deterministic error decays linearly, as shown in \eqref{eq:FTD_Complexity_error}; furthermore, the second term in \eqref{eq:FTD_Complexity_error} can be substantially reduced through variance-reduction mechanisms, such as parallel or distributed implementations that aggregate multiple stochastic operator evaluations, mini-batching, or related averaging strategies. These approaches reduce the effective noise level without increasing the number of algorithmic iterations.

We next state a parallel residual guarantee for strongly monotone problems under state dependent noise. The result shows that AR attains a near optimal residual complexity. Since the proof is similar to that of \autoref{thm:main-unbiased-mu-n}, we only include the result here and omit its proof for simplicity.
\begin{corollary}\label{thm:ARVI-state-dependent-mu}
Assume that $\mu>0$ and $0<\varepsilon\leq LD_0$, where
$D_0\geq\textnormal{dist}(x_0,X^*)$. In \autoref{alg:ARVI},
suppose $S$, $\{r_s\}_{s=1}^S$, and $\{N_s\}_{s=1}^S$ are chosen
as in \autoref{thm:main-unbiased-mu}. For each $s\in[S]$, choose
$m_s$ according to
\eqref{eqn:regularization-exp-unbiased-state-dependent}.
Then, at epoch $S$, \autoref{alg:ARVI} computes an approximate
solution $x_S$ satisfying
$\sqrt{\mathbb E[\res_F(x_S)^2]}\leq\varepsilon$ after
\begin{align}\label{eqn:state-dependent-final-mu}
\mathcal O\left(
\tfrac{L}{\mu}
\left(
\log\tfrac{\mu D_0}{\varepsilon}
+\log\log\tfrac{L}{\mu}
\right)
+\log\tfrac{L}{\mu}
+\tfrac{\varsigma_*D_0}{\varepsilon}
\left(\log\tfrac{L}{\mu}\right)^3
+\tfrac{\sigma_*^2}{\varepsilon^2}
\left(\log\tfrac{L}{\mu}\right)^3
\right)
\end{align}
calls to the $\mathcal{SO}$.
\end{corollary}
\subsection{Stochastic operator extrapolation for strongly monotone VIs}\label{sec:Stochastic strongly monotone VI}
In this subsection,
we focus on providing a subroutine $\mathcal{A}$ that satisfies the convergence guarantee in \autoref{assumption_inner_sto-state-dependent} used in the AR method.
We consider
the stochastic operator extrapolation method as $\mathcal{A}$ for solving
VI problems  \eqref{VIP} which satisfy $L$-Lipschitz continuity \eqref{eq:Lipschitz} and $\mu$-strong monotonicity \eqref{strongly_monotone} with $\mu>0$, under the assumptions in \autoref{sec:Assumptions}. We provide both suboptimal result required in \autoref{assumption_inner_sto-state-dependent} and an optimal result as a byproduct.
\begin{algorithm}[t]
\caption{Stochastic operator extrapolation}
\label{alg:SOE}
\begin{algorithmic}
\State Let $x_{0}=x_{-1}\in\bbr^n$, and let the nonnegative parameters ${\eta_k}$ and ${\lambda_k}$ be given with $\lambda_0=0$.
\For{$k=0,1,\dots,N-1$}
\State Draw a mini-batch of size $m_k$ and define
\begin{equation*}
\tilde F_k(x)
\coloneqq
\tfrac{1}{m_k}\tsum_{i=1}^{m_k}\tilde F(x;\zeta_i).
\end{equation*}
\If{$k=0$}
\State Set $\tilde F_{-1}(\cdot)\equiv \tilde F_{0}(\cdot)$.
\EndIf
\begin{equation}\label{eq:stochastic_algorithm_step}
x_{k+1}
=
\argmin_{x \in X}
\left\{
\eta_{k}
\left\langle
\tilde F_k(x_k)
+
\lambda_k\big[
\tilde F_k(x_k)
-
\tilde F_{k-1}(x_{k-1})
\big],
x
\right\rangle
+
\tfrac{1}{2}\|x-x_k\|^2\right\}.
\end{equation}
\EndFor
\end{algorithmic}
\end{algorithm}

We define the mini-batch estimator as follows
\begin{equation}\label{eqn:mini-batch-estimator}
\tilde F_k(x)
\coloneqq
\tfrac{1}{m_k}\tsum_{i=1}^{m_k}\tilde F(x;\zeta_{k,i}),
\end{equation}
where $m_k$ is the batch size.
We let $\mathcal{F}_{k}$ denote the filtration generated by the global trajectory up to, but not including, the samples used to form $\tilde F_k(\cdot)$; thus, $\mathcal{F}_{k}$ contains the entire sample history observed before $\tilde F_k(\cdot)$ is defined.
Since the sampling strategy is fixed within each epoch, the batch size is uniform across inner iterations, and we write $m_k \equiv m$.

Building upon the Assumptions \eqref{eq:fast-decreasing-bias-o} and \eqref{eqn:state-dependent-assumption},
we have the following batch-level properties on $\tilde F_k$.
\begin{lemma}
Suppose assumptions \eqref{eq:fast-decreasing-bias-o} and \eqref{eqn:state-dependent-assumption}.
The following properties hold almost surely.
\begin{itemize}
    \item[\textendash] {\it Unbiasedness:} for any $\mathcal{F}_{k}$-measurable point $x \in X$,
    \begin{align}\label{eq:fast-decreasing-bias-unregularized}
        \mathbb{E}[\tilde F_k(x)\mid\mathcal{F}_{k}]=F(x),\quad \quad \textnormal{a.s.}
    \end{align}
    \item[\textendash] {\it State-dependent variance:}  for every $\mathcal F_k$-measurable point $x\in X$, $\tilde F_k$ satisfies
\begin{align}\label{eq:state-dep-diff-op}
\mathbb{E}\big[\|\tilde F_k(x)-F(x)\|^2\mid \mathcal{F}_{k}\big]
\leq \tfrac{\sigma_*^2}{m}
+\tfrac{\varsigma_*}{m}\langle F(x)-F(x^*),\,x-x^* \rangle, \quad \textnormal{a.s.}
\end{align}
\end{itemize}
\end{lemma}
These are the $1/m$-averaged counterparts of the per-sample conditions and reduce to the latter when $m=1$. The claim follows from independence and the construction of $\tilde{F}_k$.

Consider the stochastic operator extrapolation (SOE) algorithm \cite{kotsalis2022simple}.
Starting from the initial point $x_{0}=x_{-1}\in \bbr^n$, the SOE iteration is given in \autoref{alg:SOE}, where $\tilde F_k$ is defined in \eqref{eqn:mini-batch-estimator}. Define
\begin{equation}\label{eqn:delta-error}
    \Delta_{k} \coloneqq\tilde{F}_{k}(x_{k})- F(x_{k}).
\end{equation}

We start with a general error bound for the SOE method.
Before stating it, we describe the role of the conditions imposed.
Conditions \eqref{eq:inner-loop-B}--\eqref{eq:inner-loop-A} are the standard matching conditions between the stepsizes $\{\eta_k\}$, the extrapolation weights $\{\lambda_k\}$, and the weighting sequence $\{\theta_k\}$, ensuring that the extrapolation step is correctly balanced for convergence; see, for example, \cite{kotsalis2022simple,kotsalis2022simple2}.

\smallskip
\begin{lemma}\label{lem:per-itrate}
Let $\{x_{ k}\}_{k\geq -1}$ denote the iterates generated by the SOE method in \autoref{alg:SOE}, and let $\{\theta_{ k}\}_{k\geq 0}$ be a sequence of positive numbers. Let $\{\eta_k\}_{k\geq 0}$ and $\{\lambda_k\}_{k\geq 0}$ denote the stepsizes and extrapolation weights, with $\lambda_0=0$. Suppose that, for all $k\geq 0$, they satisfy
\begin{align}
 16\theta_{ k+1}(\lambda_{ k+1}\eta_{ k+1})^2L^2
  &\leq  \theta_{ k},\label{eq:inner-loop-B}\\
  \theta_{ k+1}\eta_{ k+1}\lambda_{ k+1}
  &=\theta_{ k}\eta_{ k}.\label{eq:inner-loop-A}
\end{align}
Then, for any $x\in X,$
for all $N\geq 1$, it holds that
\begin{equation}\label{eqn:iterate-wise}
 \begin{aligned}
&\tsum_{k = 0}^{N-1}\theta_{ k}\left[\tfrac{1}{2}(\|x_{ k} - x\|^2-\|x_{  k + 1} - x\|^2)-\eta_{  k}\langle F(x_{  k + 1}), x_{  k + 1}-x \rangle\right] +\tfrac{3\theta_{N}}{8}\| x_{N} - x \|^2+\theta_{N}(\eta_{ N} \lambda_{ N})^2\|\Delta_{N - 1}\|^2\\
&\geq\tsum_{k = 0}^{N-1}\left[\tfrac{\theta_k}{8}\|x_{  k + 1} - x_{ k}\|^2-4\theta_k(\eta_{  k} \lambda_{ k})^2(\|\Delta_{  k}\|^2 + \|\Delta_{  k - 1}\|^2)\right]+  {\tsum_{k = 0}^{N-2}}\theta_{ k}\eta_{  k} \langle \Delta_{ {k + 1}}, x_{  k + 1} - x\rangle.
\end{aligned}
\end{equation}
\end{lemma}
\begin{proof}
By the initialization $x_{-1}=x_0$ and the choice of using the same batch to construct $\tilde{F}_{-1}$ and $\tilde{F}_0$, we have
\begin{align}\label{eqn:initial-p1}
\tilde{F}_{ -1}(x_{ -1})-\tilde{F}_{ 0}(x_{ 0})
    \overset{\eqref{eqn:delta-error}}{=}
  \tilde{F}_{ -1}(x_{ -1})-\tilde{F}_{ 0}(x_{ -1})
   = 0,
\end{align}
where the first equality follows from $x_{-1}=x_0$, while the second follows because $\tilde{F}_{-1}(\cdot)=\tilde{F}_0(\cdot)$, as the two stochastic operators are constructed using the same sample batch. Moreover, by the definition of $\Delta_k$ in \eqref{eqn:delta-error}, it follows that $\Delta_{-1}=\Delta_0$. Therefore, for SOE, the three-point lemma \citep[Lemma 3.1]{lan2020first} implies that, for each $k\geq 0$ and all $x\in X,$ it holds almost surely that
\begin{equation}\label{eq:state-dependent-optimality-condition}
\begin{aligned}
\tfrac{1}{2}\|x_{k} - x\|^2
&\geq \tfrac{1}{2}\|x_{k + 1} - x\|^2 + \eta_{k}\langle F(x_{k + 1}), x_{ k + 1} - x\rangle
\\
&\quad + \eta_{ k} \lambda_{k} \langle \Delta_{ k}-\Delta_{ k-1},  x_{ k + 1} -x_{k}\rangle+\eta_{ k}\langle \Delta_{ k+1}, x_{ k + 1} - x\rangle\\
&\quad + \eta_{ k} \langle \tilde{F}_{k}(x_{k})-\tilde{F}_{k+1}(x_{k+1}), x_{ k + 1} - x\rangle-\eta_{ k}\lambda_{k} \langle\tilde{F}_{k-1}(x_{k-1})-\tilde{F}_{k}(x_{k}), x_{k }- x\rangle \\
&\quad+ \tfrac{1}{2}\| x_{k}-x_{k + 1} \|^2 + \eta_{ k} \lambda_{k} \langle F(x_{k}) - F(x_{ k - 1}),  x_{ k + 1}-x_{k} \rangle.
\end{aligned}
\end{equation}
Multiplying \eqref{eq:state-dependent-optimality-condition} by $\theta_k$, and summing over $k=0,\ldots,N-1$,  we obtain
\begin{equation}\label{eq:state-dependent-optimality-condition-mult-theta-new}
\begin{aligned}
&\tsum_{k = 0}^{N-1}\theta_{ k}\left[\tfrac{1}{2}\|x_{ k} - x\|^2-\tfrac{1}{2}\|x_{  k + 1} - x\|^2-\eta_{  k}\langle F(x_{  k + 1}), x_{  k + 1} - x\rangle\right] \\
&\geq{\tsum_{k = 0}^{N-1}\theta_{ k}\eta_{  k} \lambda_{ k} \langle \Delta_{  k}-\Delta_{  k-1}, x_{  k + 1} - x_{ k}\rangle}+ {  {\tsum_{k = 0}^{N-1}}\theta_{ k}\eta_{  k} \langle \Delta_{ {k + 1}}, x_{  k + 1} - x\rangle}\\
&\quad + {\theta_{ N-1}\eta_{  N-1} \langle{F}(x_{ N-1})-{F}(x_{ N})+\Delta_{N-1}-\Delta_N, x_{N} - x\rangle}\\
&\quad  +{\tsum_{k = 0}^{N-1}\tfrac{\theta_{ k}}{2}\|x_{  k + 1} - x_{ k}\|^2+ \tsum_{k = 0}^{N-1}\theta_{ k}\eta_{  k} \lambda_{ k} \langle F(x_{ k}) - F(x_{  k - 1}), x_{  k + 1} - x_{ k}\rangle},
\end{aligned}
\end{equation}
where the inequality follows from the telescoping of the inner-product terms, using \eqref{eq:inner-loop-A} and \eqref{eqn:initial-p1}. We first bound the deterministic cross terms in the above relation. By Cauchy-Schwarz inequality, Young's inequality, the $L$-Lipschitz continuity  of $F,$ and the condition on the stepsize $\eta_k$ in \eqref{eq:inner-loop-B}, we have
\begin{equation}\label{eq:step4bound}
\begin{aligned}
&\,\,\,\tsum_{k = 0}^{N - 1} \tfrac{\theta_{  k}}{2}\|x_{  k + 1} - x_{ k}\|^2 + \tsum_{k = 0}^{N - 1} \theta_{  k} \eta_{  k} \lambda_{ k} \langle F(x_{ k}) - F(x_{  k - 1}), x_{  k + 1} - x_{ k}\rangle \\
&\,\,\,\geq\tsum_{k = 0}^{N - 1} \tfrac{3\theta_{  k}}{8}\|x_{  k + 1} - x_{ k}\|^2  - 2\tsum_{k = 0}^{N - 1} \theta_{ k}(\eta_{  k} \lambda_{ k})^2 L^2 \|x_{ k} - x_{  k - 1}\|^2 \\
&\overset{\eqref{eq:inner-loop-B}}{\geq}\tsum_{k = 0}^{N - 1} \tfrac{\theta_{  k}}{4}\|x_{  k + 1} - x_{ k}\|^2+\tfrac{\theta_{N-1}}{8}\| x_{ N-1}-x_{ N} \|^2 - 2\theta_{ 0}(\eta_{ 0} \lambda_{ 0})^2L^2 \|x_{  -1} - x_{ 0}\|^2 \\
&\,\,\,= \tsum_{k = 0}^{N - 1} \tfrac{\theta_{  k}}{4} \|x_{  k + 1} - x_{ k}\|^2+\tfrac{\theta_{  N-1}}{8}\| x_{ N-1}-x_{ N} \|^2,
\end{aligned}
\end{equation}
Similarly, by the Cauchy--Schwarz and Young's inequalities, the $L$-Lipschitz continuity  of $F,$ and the condition on the extrapolation weight $\lambda_k$ in \eqref{eq:inner-loop-A}, we obtain
\begin{equation}\label{eqn:thm-final-K}
    \begin{aligned}
        &{\theta_{ N-1}\eta_{  N-1} \langle{F}(x_{ N-1})-{F}(x_{ N}), x_{  N} - x\rangle} +2\theta_{  N} (\eta_{  N} \lambda_{ N})^2L^2 \| x_{  N - 1}- x_{ N}\|^2+\tfrac{\theta_{  N}}{8}\| x_{ N} - x \|^2\geq 0.
    \end{aligned}
\end{equation}
Furthermore,
by the condition on the stepsize $\eta_{N}$ in \eqref{eq:inner-loop-B}, it holds that
\begin{equation}\label{thm-inner-bias-2}
    \begin{aligned}
     &\tfrac{\theta_{  N-1}}{8}\| x_{ N-1}-x_{ N} \|^2 -2\theta_{N} (\eta_{N} \lambda_{ N})^2L^2 \|x_{ N - 1} - x_{ N}\|^2\geq0.
    \end{aligned}
\end{equation}
Next, we bound the stochastic cross terms in \eqref{eq:state-dependent-optimality-condition-mult-theta-new}. By the Cauchy--Schwarz and Young inequalities, we obtain
\begin{equation}\label{eq:step1bound}
\begin{aligned}
\theta_k\eta_{  k} \lambda_{ k} \langle \Delta_{  k} - \Delta_{  k - 1}, x_{  k + 1} - x_{ k}\rangle + \tfrac{\theta_k}{8}\|x_{  k + 1} - x_{ k}\|^2&\,\,\,\,\geq-4\theta_k(\eta_{  k} \lambda_{ k})^2(\|\Delta_{  k}\|^2 + \|\Delta_{  k - 1}\|^2)\\
\theta_{ N-1}\eta_{N-1} \langle\Delta_{N-1}, x_{  N} - x\rangle+\tfrac{\theta_{N}}{4} \|x_{N} - x  \|^2&\overset{\eqref{eq:inner-loop-A}}{\geq} -\theta_{N} (\eta_{ N} \lambda_{ N})^2\|\Delta_{N - 1}\|^2.
\end{aligned}
\end{equation}
Substituting \eqref{eq:step4bound}--\eqref{eq:step1bound} into \eqref{eq:state-dependent-optimality-condition-mult-theta-new} concludes the proof.
\end{proof}
\smallskip

We now establish the main convergence properties of SOE under the state-dependent noise assumption. In particular, \autoref{prop:main-theorem-bias-exp} provides bounds on the expected squared distance to the solution, the aggregated successive differences between the iterates, and the state-dependent variance term. Condition \eqref{eq:inner-loop-D} requires the effective state-dependent noise level $\varsigma_*/m$ is small enough.
Here the problem-dependent constants $L$, $\mu$, $\sigma_*^2$, and $\varsigma_*$ are fixed, while the schedules $\{\theta_k\}$, $\{\eta_k\}$, $\{\lambda_k\}$, and the batch size $m$ are free to be chosen.
\begin{theorem}\label{prop:main-theorem-bias-exp}
Let $\{x_{ k}\}_{k\geq -1}$ denote the iterates generated by the SOE method in \autoref{alg:SOE}, and let $\{\theta_{ k}\}_{k\geq 0}$ be a sequence of nonnegative numbers. Let $\{\eta_k\}_{k\geq 0}$ and $\{\lambda_k\}_{k\geq 0}$ denote the stepsizes and extrapolation weights, with $\lambda_0=0$. Suppose that, for all $k\geq 0$, they satisfy  \eqref{eq:inner-loop-B}, \eqref{eq:inner-loop-A} and
\begin{align}
\tfrac{\theta_{k+1}}{2}&\leq \tfrac{    \theta_{k}}{2}
    + \tfrac{\mu\theta_{k}\eta_{ k}}{2}.
    \label{eq:inner-loop-C}
\end{align}
Furthermore,
suppose that for all $k\geq 0,$ the batch size $m$ is large enough that
\begin{align}
   \tfrac{12\varsigma_*}{m}
\left[
\theta_{k+1}(\eta_{k+1}\lambda_{k+1})^2
+
\theta_{k+2}(\eta_{k+2}\lambda_{k+2})^2
\right]\leq  {\theta_{k}\eta_{k}}.
    \label{eq:inner-loop-D}
\end{align}
Then, for all $N\geq 1$, it holds that
\begin{equation}\label{eqn:main-theorem-bias-exp}
\begin{aligned}
&
\tfrac{\theta_{  N}}{8} \mathbb{E}[\|x_{ N} - x^* \|^2]
+\tsum_{k = 0}^{N - 1} \tfrac{\theta_{  k}}{8} \mathbb E[\|x_{  k + 1} - x_{ k}\|^2]+\tsum_{k = 1}^{N}
\tfrac{\theta_{k-1}\eta_{k-1}}{6}\mathbb E\left[\langle F(x_{k}) - F(x^*), x_{k} - x^*\rangle\right]\\
&\leq\left\{
\tfrac{\theta_{ 0}}{2}+\tfrac{4\varsigma_* L\theta_{  1}(\eta_{   1} \lambda_{ 1})^2}{m}\right\}
\mathbb E\left[\|x_{ 0}-x^*\|^2\right]+\tfrac{8\sigma_*^2}{m}\tsum_{k = 0}^{N} { \theta_{  k} (\eta_{  k} \lambda_{ k})^2 }.
\end{aligned}
\end{equation}
\end{theorem}
\begin{proof}
Taking $x=x^*$ in \eqref{eqn:iterate-wise} and applying the following conditional expectation, we obtain
\begin{equation}\label{eq:step1bound'}
    \begin{aligned}
\mathbb E[\|\Delta_{  k}\|^2\,|\,\mathcal{F}_{ k}]
&\overset{\eqref{eqn:delta-error}}{=} \mathbb E[\|\tilde{F}_{ k}(x_{ k}) - F(x_{ k})\|^2\,|\,\mathcal{F}_{ k}]
\overset{\text{(i)}}{\leq} \tfrac{\varsigma_*}{m}\langle F(x_{ k}) - F(x^*), x_{ k} - x^*\rangle+
\tfrac{\sigma_*^2}{m},
    \end{aligned}
\end{equation}
where, in (i), we used the state dependent variance property
 \eqref{eq:state-dep-diff-op}.
 Similarly, for all $k\geq 1,$ we have
\begin{equation}\label{eq:step1bound''}
    \begin{aligned}
        &\,\,\mathbb E[\|\Delta_{  k-1}\|^2\,|\,\mathcal{F}_{ k-1}] \leq \tfrac{\varsigma_*}{m}\langle F(x_{ k-1}) - F(x^*), x_{ k-1} - x^*\rangle+ \tfrac{\sigma_*^2}{m}.
    \end{aligned}
\end{equation}
For $k=0$, there holds $\Delta_{-1}=\Delta_0$.
Furthermore, by the unbiasedness assumption and $x_{  k + 1}\in \mathcal{F}_{ k+1}$, we obtain
\begin{equation}\label{eqn: term II}
\begin{aligned}
&\,\,\tsum_{k = 0}^{N - 1} \theta_{  k} \eta_{  k} \mathbb E[\langle \Delta_{  k+1}, x_{  k + 1} - x^*\rangle\,|\,\mathcal{F}_{ k+1}]=0.
\end{aligned}
\end{equation}
Substituting  \eqref{eq:step1bound'}, \eqref{eq:step1bound''} and \eqref{eqn: term II} into \eqref{eqn:iterate-wise}, we obtain
\begin{equation}\label{eqn:term I}
     \begin{aligned}
&\tsum_{k = 0}^{N-1}\theta_{ k}\mathbb E\left[\tfrac{1}{2}\|x_{ k} -x^*\|^2-\tfrac{1}{2}\|x_{  k + 1} -x^*\|^2-\eta_{  k}\langle F(x_{  k + 1}), x_{  k + 1}-x^* \rangle\right]+\tfrac{8\sigma_*^2}{m}\tsum_{k = 0}^{N} \theta_{  k} (\eta_{  k} \lambda_{ k})^2 \\
&\overset{\text{(ii)}}{\geq} \tsum_{k = 0}^{N-1}\tfrac{\theta_k}{8}\mathbb E[\|x_{  k + 1} - x_{ k}\|^2]-\tfrac{3\theta_{N}}{8}\mathbb E[\| x_{ N} - x^* \|^2]\\
&\quad- \tfrac{4\varsigma_*}{m}\tsum_{k = 0}^{N - 1}  [\theta_{  k} (\eta_{  k} \lambda_{ k})^2  + \theta_{  k + 1}(\eta_{  k + 1} \lambda_{  k + 1})^2]\mathbb  E\left[\langle F(x_{ k}) - F(x^*), x_{ k} - x^*\rangle\right] \\
&\quad-\tfrac{4\theta_{ 0}(\eta_{ 0} \lambda_{ 0})^2 \varsigma_*}{m} \mathbb E\left[\langle F(x_{ 0}) - F(x^*), x_{ 0} - x^*\rangle\right] + \tfrac{4\theta_{  N}(\eta_{  N}\lambda_{ N})^2 \varsigma_*}{m} \mathbb  E\left[\langle F(x_{N- 1}) - F(x^*), x_{  N - 1} - x^*\rangle\right]\\
 &\quad- \tfrac{\theta_{ N}(\eta_{N} \lambda_{ N})^2\varsigma_*}{m} \mathbb E[ \langle F(x_{ N-1}) - F(x^*), x_{ N-1} - x^*\rangle],
\end{aligned}
\end{equation}
where, in (ii), we used $x_{ -1}=x_{ 0}$ and $\tilde{F}_{-1}(\cdot)=\tilde{F}_{0}(\cdot).$
By the fact that $x^*$ is the solution to \eqref{VIP}, it holds that
\begin{equation}\label{eq:x-star-solves-the-VI}
\begin{aligned}
\left\langle F(x_{k+1}),x_{k+1}-x^*\right\rangle
&\geq
\left\langle F(x_{k+1})-F(x^*),x_{k+1}-x^*\right\rangle.
\end{aligned}
\end{equation}
Thus, combining \eqref{eq:x-star-solves-the-VI} with \eqref{eqn:term I}, we
define $\mathcal{I}_N$ as follows.
\begin{equation*}
\begin{aligned}
     \mathcal{I}_{ N}
      &\,\,\coloneqq \tsum_{k = 1}^{N} \left\{\theta_{k-1}\eta_{ k-1}-\tfrac{4\varsigma_*}{m}
\left[
\theta_{k}(\eta_{k}\lambda_{k})^2
+
\theta_{k+1}(\eta_{k+1}\lambda_{k+1})^2
\right]\right\}\mathbb E\left[\langle F(x_{k}) - F(x^*), x_{k} - x^*\rangle\right].
    \end{aligned}
\end{equation*}
Substituting $\mathcal{I}_N$ into
\eqref{eqn:term I},  we obtain
\begin{equation}\label{eqn:final-1}
      \begin{aligned}
&\tsum_{k = 0}^{N-1}\theta_{ k}\mathbb E\left[\tfrac{1}{2}\|x_{ k} -x^*\|^2-\tfrac{1}{2}\|x_{  k + 1} -x^*\|^2\right] +\tfrac{8\sigma_*^2}{m}\tsum_{k = 0}^{N} \theta_{  k} (\eta_{  k} \lambda_{ k})^2+\tfrac{3\theta_{N}}{8}\mathbb E[\| x_{ N} - x^* \|^2]\\
&\geq \tsum_{k = 0}^{N-1}\tfrac{\theta_k}{8}\mathbb E[\|x_{  k + 1} - x_{ k}\|^2]+\mathcal{I}_N- \tfrac{\varsigma_*}{m}[8\theta_{  0} (\eta_{  0} \lambda_{ 0})^2  + 4\theta_{   1}(\eta_{  1} \lambda_{1})^2]\mathbb  E\left[\langle F(x_{0}) - F(x^*), x_{ 0} - x^*\rangle\right].
\end{aligned}
\end{equation}
By the condition of batch size $m$ in \eqref{eq:inner-loop-D} and strong monotonicity, we have
\begin{align*}
             \mathcal{I}_{ N}
&\overset{\eqref{eq:inner-loop-D}}{\geq}\tsum_{k = 1}^{N} \left\{\tfrac{\theta_{k-1}\eta_{ k-1}}{2}+\tfrac{\theta_{k-1}\eta_{ k-1}}{6}\right\}\mathbb E\left[\langle F(x_{k}) - F(x^*), x_{k} - x^*\rangle\right]\\
&\,\,\,\geq\tsum_{k = 0}^{N-1} \tfrac{\mu\theta_{k}\eta_{ k}}{2}\mathbb E[\|x_{k+1}-x^*\|^2]+\tsum_{k = 1}^{N} \tfrac{\theta_{k-1}\eta_{ k-1}}{6}\mathbb E\left[\langle F(x_{k}) - F(x^*), x_{k} - x^*\rangle\right].
\end{align*}
Applying \eqref{eq:inner-loop-C} to \eqref{eqn:final-1} for
$k=0,\ldots,N-1$ telescopes the corresponding terms. Furthermore,
for $k=N-1$,
\begin{align*}
\tfrac{\theta_N}{8}
=
\tfrac{\theta_N}{2}-\tfrac{3\theta_N}{8}
\leq
\tfrac{\theta_{N-1}}{2}
+\tfrac{\mu\theta_{N-1}\eta_{N-1}}{2}
-\tfrac{3\theta_N}{8}.
\end{align*}
Moreover, the $L$-Lipschitz continuity of $F$ implies $\left\langle F(x_0)-F(x^*),x_0-x^*\right\rangle
\leq L\|x_0-x^*\|^2.$
Applying these bounds to \eqref{eqn:final-1} completes the proof.
\end{proof}

The right-hand side consists of the initial optimality gap $\|x_0-x^*\|^2$, and a noise term $\tfrac{16 \sigma_*^2}{m}\tsum_{k=0}^{K-1}[\theta_{k} (\eta_{k}\lambda_{k})^2+\theta_{k+1} (\eta_{k+1}\lambda_{k+1})^2]$ which decreases with the batch size $m$.
Concrete convergence rates can be obtained by specializing the weights, stepsizes, and extrapolation parameters.

\smallskip
We next specify such a choice in \autoref{prop:main-theorem-bias-exp} to derive a sublinear convergence rate for SOE under the assumptions in \autoref{sec:Assumptions}.
\begin{corollary}\label{cor:subroutine-guarantee}
    Suppose $\lambda_0=0,$ and for all $k\geq 1,$ $ \lambda_{ k}=\tfrac{\eta_{  k-1}\theta_{  k-1}}{\eta_{  k}\theta_{ k}},   m\geq
\left\lceil\tfrac{6\varsigma_*}{L}\right\rceil.$ Let
    \begin{align}
        k_{ 0}&=\tfrac{16L}{\mu}+1,\quad
        \eta_{ k}=\tfrac{4}{\mu(k+k_{ 0}-1)},
       \quad
        \theta_{ k}=(k+k_{ 0})(k+k_{ 0}-1)\label{eqn:stepsize-inner},\quad
    \forall\, k\geq 0.
    \end{align}
    Then, for all $N\geq 1,$ SOE satisfies
\begin{equation}\label{eqn:subroutine-guarantee}
    \begin{aligned}
    &{(N+k_{ 0})(N+k_{ 0}-1)}\mathbb{E}[\|x_{N} - x^*\|^2]+ \tsum_{k = 0}^{N - 1} (k+k_{ 0})(k+k_{ 0}-1)\mathbb E[\|x_{  k + 1} - x_{ k}\|^2]\\
    &\quad+\tfrac{16}{3\mu}
\tsum_{k=1}^{N}(k+k_0-1)
\mathbb{E}\left[
\left\langle F(x_k)-F(x^*),x_k-x^*\right\rangle
\right]\\
    &\leq \tfrac{9k_0^2}{2}\mathbb E[\|x_{ 0} - x^*\|^2]+\tfrac{1024(N-1)\sigma_*^2}{\mu^2 m}.
    \end{aligned}
    \end{equation}
\end{corollary}
\begin{proof}
It suffices to verify that, under the choices of $\eta_k, \lambda_k, \theta_k,$ and $m,$ the conditions of \autoref{prop:main-theorem-bias-exp} hold.
For all $k\geq 0,$ setting $\lambda_{ k+1}=\tfrac{\theta_{ k}\eta_{ k}}{\theta_{ k+1}\eta_{ k+1}}$ ensures that \eqref{eq:inner-loop-A} holds. To verify \eqref{eq:inner-loop-B},
note that for all $k\geq 0,$ we have
\begin{equation*}
    \begin{aligned}
       16 \theta_{ k+1}(\lambda_{ k+1} \eta_{ k+1})^2L^2
       = 16 \theta_{ k+1}\left(\tfrac{\theta_{ k}\eta_{ k}}{\theta_{ k+1}}\right)^2L^2  \overset{\text{(i)}}{\leq} 16 \theta_{ k}\eta_{ k}^2L^2
       \overset{\text{(ii)}}{\leq} \theta_{ k},
    \end{aligned}
\end{equation*}
where (i) follows from the monotonicity of $\{\theta_{ k}\}$ in \eqref{eqn:stepsize-inner}; and (ii) follows from the stepsize choice $\eta_k$ in \eqref{eqn:stepsize-inner}, which implies $ \eta_{ k}\leq \tfrac{1}{4L}$, for all $k\geq0.$
Therefore, \eqref{eq:inner-loop-B} holds. Furthermore, notice that
\begin{equation*}
    \begin{aligned}
     \tfrac{12\theta_{ k+1}(\lambda_{ k+1}\eta_{ k+1})^2 \varsigma_*}{m}&=\tfrac{12\theta_{ k+1}}{m}\left(\tfrac{\theta_{ k}\eta_{ k}}{\theta_{ k+1}}\right)^2 \varsigma_*\overset{\text{(iii)}}{\leq}\tfrac{12\theta_{ k}\eta_{ k}^2\varsigma_*}{m}\overset{\text{(iv)}}{\leq}\tfrac{\theta_{ k}\eta_k\lambda_k}{2},\quad
        \tfrac{12\theta_{ k+2}(\lambda_{ k+2}\eta_{ k+2})^2 \varsigma_*}{m}\leq\tfrac{\theta_{ k+1}\eta_{k+1}\lambda_{k+1}}{2}=\tfrac{\theta_{ k}\eta_{ k}}{2},
    \end{aligned}
\end{equation*}
where (iii) follows again from the monotonicity of $\{\theta_{k}\},$ and (iv)
follows by the condition on the state-dependent variance parameter $\varsigma_*$ as $\tfrac{\eta_k \varsigma_*}{m} \leq \tfrac{\lambda_k}{24}$. Hence,
\eqref{eq:inner-loop-D} holds due to $\lambda_k\leq 1$.
\smallskip

We proceed with \eqref{eq:inner-loop-C}.
Since $\eta_k, \theta_{k}$ are chosen as in \eqref{eqn:stepsize-inner}, it holds that
\begin{align*}
    \tfrac{    \theta_{k}}{2}
    + \tfrac{\mu\theta_{k}\eta_{ k}}{2}&\,{=}\,\tfrac{\theta_{ k}}{2}\left(1+\tfrac{2}{(k+k_{ 0}-1)}\right)\overset{\eqref{eqn:stepsize-inner}}{\geq} \tfrac{\theta_{ k+1}}{2}.
\end{align*}
Therefore, \autoref{prop:main-theorem-bias-exp} holds.
Substituting the choices for
$\theta_{  k}, \eta_{  k}, \lambda_{ k}$ into \eqref{eqn:main-theorem-bias-exp},
we obtain
\begin{equation}\label{eqn:cor-2-final-2}
\begin{aligned}
\tfrac{\theta_{ 0}}{2}+\tfrac{4\varsigma_* L}{m}[2\theta_{ 0} (\eta_{ 0} \lambda_{ 0})^2 + \theta_{ 1}(\eta_{ 1} \lambda_{ 1})^2]
&=
\tfrac{\theta_{ 0}}{2}+\tfrac{4\varsigma_* L}{m} \theta_{ 1}(\eta_{ 1} \lambda_{ 1})^2\le \tfrac{\theta_{0}}{2}+\tfrac{L}{4}\theta_{1}\eta_{1}\lambda_{1}^3
 \leq \tfrac{9k_0^2}{16}.
\end{aligned}
\end{equation}
Furthermore, notice that
\begin{equation}\label{eq:concrete-noise-dep-lhs-bound}
    \begin{aligned}
        &\tsum_{k = 1}^{N}
\tfrac{8\theta_{k-1}\eta_{k-1}}{6}\mathbb E\left[\langle F(x_{k}) - F(x^*), x_{k} - x^*\rangle\right]=
\tfrac{16}{3\mu}
\tsum_{k=1}^{N}(k+k_0-1)
\mathbb{E}\left[
\left\langle F(x_k)-F(x^*),x_k-x^*\right\rangle
\right].
    \end{aligned}
\end{equation}
Furthermore, by substituting the choices of $\lambda_k$, $\theta_k$, and $\eta_k$, we have
\begin{equation*}
m\geq
\tfrac{96\varsigma_*(k+k_0)}
{\mu(k+k_0-1)^2}.
\end{equation*}
Since the right-hand side is decreasing in $k$, its maximum is attained at $k=1$. Hence, using $k_0=\tfrac{16L}{\mu}+1$, it suffices to choose $m\geq
\left\lceil\tfrac{6\varsigma_*}{L}\right\rceil$ to guarantee $\tfrac{\eta_k \varsigma_*}{m} \leq \tfrac{\lambda_k}{24}.$
It remains to bound the variance term in \eqref{eqn:main-theorem-bias-exp}.
\begin{equation}\label{eqn:cor-2-final-5}
    \begin{aligned}
     \tfrac{8\sigma_*^2}{m}\tsum_{k = 0}^{N-1} { \theta_{  k} (\eta_{  k} \lambda_{ k})^2 }&\overset{\eqref{eqn:stepsize-inner}}{=}\tsum_{k = 1}^{N -1} \tfrac{128\sigma_*^2}{\mu^2 m}\cdot\tfrac{{k+k_{ 0}-1}{}}{k+k_{0}}\leq \tfrac{128(N-1)\sigma_*^2}{\mu^2 m}.
    \end{aligned}
\end{equation}
Substituting \eqref{eqn:cor-2-final-2}  and \eqref{eqn:cor-2-final-5} into \eqref{eqn:main-theorem-bias-exp} concludes the proof.
\end{proof}
\medskip

It is interesting to note that \autoref{cor:subroutine-guarantee} also provides a guarantee on the weighted cumulative inner products.
Let $R$ be sampled according to
\begin{equation}\label{eqn:iterative-averaging}
\mathbb{P}(R=k)
=
\tfrac{k+k_0-1}
{\sum_{j=1}^{N}(j+k_0-1)}
=
\tfrac{2(k+k_0-1)}
{N(N+2k_0-1)}.
\end{equation}

\begin{remark}\label{remark:another-output}
Under the conditions of \autoref{cor:subroutine-guarantee}, the random output $x_R$ satisfies
\begin{equation*}
\begin{aligned}
\tfrac{16N}{\mu(N+k_{0})}
\mathbb{E}\left[
\left\langle F(x_R)-F(x^*),x_R-x^*\right\rangle
\right]
&\leq
\tfrac{9k_0^2\mathbb{E}[\|x_{0}-x^*\|^2] }{2(N+k_{0})(N+k_{0}-1)}
+\tfrac{1024(N-1)\sigma_*^2}
{\mu^2 m(N+k_{0})(N+k_{0}-1)}.
\end{aligned}
\end{equation*}
Indeed, by the definition of the random output in \eqref{eqn:iterative-averaging}, we have
\begin{equation*}
\tfrac{16}{3\mu}
\tsum_{k=1}^{N}(k+k_0-1)
\mathbb{E}\left[
\left\langle F(x_k)-F(x^*),x_k-x^*\right\rangle
\right]
=
\tfrac{8N(N+2k_0-1)}{3\mu}
\mathbb{E}\left[
\left\langle F(x_R)-F(x^*),x_R-x^*\right\rangle
\right].
\end{equation*}
Combining it with \autoref{cor:subroutine-guarantee},
\begin{equation*}
\begin{aligned}
\tfrac{16N}{\mu(N+k_{0})}
\mathbb{E}\left[
\left\langle F(x_R)-F(x^*),x_R-x^*\right\rangle
\right]
&\leq
\tfrac{64N(N+2k_0-1)}
{3\mu(N+k_{0})(N+k_{0}-1)}
\mathbb{E}\left[
\left\langle F(x_R)-F(x^*),x_R-x^*\right\rangle
\right] \\
&\leq
\tfrac{9k_0^2\mathbb{E}[\|x_{0}-x^*\|^2] }{2(N+k_{0})(N+k_{0}-1)}
+\tfrac{1024(N-1)\sigma_*^2}
{\mu^2 m(N+k_{0})(N+k_{0}-1)}.
\end{aligned}
\end{equation*}
This concludes the proof.
\end{remark}

\smallskip
We next establish a residual guarantee for the last iterate. Together with \autoref{cor:subroutine-guarantee}, this result will later be used to verify that SOE satisfies \autoref{assumption_inner_sto-state-dependent}.
\begin{corollary}\label{lem:slow-residual-last}
Suppose the conditions in \autoref{cor:subroutine-guarantee} hold. Let $N\geq 2.$ Then, it holds that
\begin{align}
\sqrt{\mathbb{E}\left[\res_F(x_N)^2\right]}
&\leq
28L\|x_0-x^*\|
+
\tfrac{29\sqrt{N}\,\sigma_*}{\sqrt{m}}.
\label{eqn:res-last}
\end{align}
\end{corollary}
\begin{proof}
For SOE \autoref{alg:SOE}, by the optimality condition of $x_k$ in \eqref{eq:stochastic_algorithm_step}, for any $k\geq 2$, for all $x\in X,$ we have $\left\langle F(x_k)+\varepsilon_{k-1},x-x_k\right\rangle\geq 0,$ where
\begin{equation*}
\varepsilon_{k-1}
\coloneqq
\tilde F_{k-1}(x_{k-1})-F(x_k)
+\lambda_{k-1}[\tilde F_{k-1}(x_{k-1})-\tilde F_{k-2}(x_{k-2})]
+\tfrac{x_k-x_{k-1}}{\eta_{k-1}}.
\end{equation*}
Hence, by the definition of the residual in \eqref{def_res}, it holds that $\res_F(x_k)\leq \|\varepsilon_{k-1}\|.$
Using the Lipschitz continuity of $F$, the triangle inequality, for all $k\geq 2$, we obtain
\begin{equation*}
\begin{aligned}
\res_F(x_k)
&\leq
(1+\lambda_{k-1})\|\tilde F_{k-1}(x_{k-1})-F(x_{k-1})\|
+\lambda_{k-1}\|\tilde F_{k-2}(x_{k-2})-F(x_{k-2})\| 
+(L+\tfrac{1}{\eta_{k-1}})\|x_k-x_{k-1}\|\\
&\quad
+\lambda_{k-1}L\|x_{k-1}-x_{k-2}\|.
\end{aligned}
\end{equation*}
By Minkowski's inequality, for all $k\geq 2$, we obtain
\begin{equation}\label{eqn:epsilon-residual}
    \begin{aligned}
        \sqrt{\mathbb{E}\left[\res_F(x_k)^2\right]}
&\leq
(1+\lambda_{k-1})
\sqrt{\mathbb{E}\left[
\|\tilde F_{k-1}(x_{k-1})-F(x_{k-1})\|^2
\right]}
+\lambda_{k-1}
\sqrt{\mathbb{E}\left[
\|\tilde F_{k-2}(x_{k-2})-F(x_{k-2})\|^2
\right]}\\
&\quad
+\left(L+\tfrac{1}{\eta_{k-1}}\right)
\sqrt{\mathbb{E}\left[\|x_k-x_{k-1}\|^2\right]}
+\lambda_{k-1}L
\sqrt{\mathbb{E}\left[\|x_{k-1}-x_{k-2}\|^2\right]}.
    \end{aligned}
\end{equation}
For any $k\geq 0$, by the state-dependent variance bound \eqref{eq:step1bound''}, we have
\begin{equation}\label{eqn:var-state-dependent}
\begin{aligned}
\mathbb{E}[\|\tilde F_k(x_k)-F(x_k)\|^2]
\leq
s_k+\tfrac{\sigma_*^2}{m},
\quad\text{where}\quad
s_k\coloneqq
\tfrac{\varsigma_*}{m}
\mathbb{E}\left[
\langle F(x_k)-F(x^*),x_k-x^*\rangle
\right].
\end{aligned}
\end{equation}
By the choices of the parameters in \autoref{cor:subroutine-guarantee},
we have
\begin{equation}\label{eqn:residual-inter-1}
  k_{0}=\tfrac{16L}{\mu}\geq 16,\,\,\lambda_{k-1}\leq 1, \,\,L+\tfrac{1}{\eta_{k-1}}\leq\tfrac{(4k+5k_0)\mu}{16},\,\,\text{and}\,\,(2k+3k_0)^2\leq 16(k-1+k_{0})(k+k_{0}-2)=16\theta_{k-1}.
\end{equation}
Substituting \eqref{eqn:var-state-dependent} and \eqref{eqn:residual-inter-1} into \eqref{eqn:epsilon-residual}, and using Minkowski's inequality again, for all $k\geq 2$, we obtain
\begin{equation}\label{eqn:per-iterate bound}
\begin{aligned}
\sqrt{\mathbb{E}\left[\res_F(x_k)^2\right]}
&\leq
\tfrac{(4k+5k_0)\mu}{16}
\sqrt{\mathbb{E}\left[\|x_k-x_{k-1}\|^2\right]}
+\tfrac{\mu k_0}{16}
\sqrt{\mathbb{E}\left[\|x_{k-1}-x_{k-2}\|^2\right]}
+2\sqrt{s_{k-1}}+\sqrt{s_{k-2}}
+\tfrac{3\sigma_*}{\sqrt{m}}.
\end{aligned}
\end{equation}
We proceed with bounding each term in \eqref{eqn:per-iterate bound}.
By \autoref{cor:subroutine-guarantee}, we obtain
\begin{equation}\label{eqn:cor}
    \mathbb{E}[\|x_N-x_{N-1}\|^2]\leq \tfrac{B}{\theta_{N-1}},\quad \mathbb{E}[\|x_{N-1}-x_{N-2}\|^2]\leq \tfrac{B}{\theta_{N-2}},\quad \tsum_{k = 1}^{N}
\tfrac{4\theta_{k-1}\eta_{k-1}\mathbb E\left[\langle F(x_{k}) - F(x^*), x_{k} - x^*\rangle\right]}{3}\leq B,
\end{equation}
where $B$ is defined as follows
\begin{equation}\label{eqn:B}
B
\coloneqq
\tfrac{9k_0^2}{2}\|x_0-x^*\|^2
+\tfrac{1024(N-1)\sigma_*^2}{\mu^2 m},
\qquad
\sqrt{B}
\leq
\tfrac{3k_0\|x_0-x^*\|}{\sqrt{2}}
+\tfrac{32\sqrt{N}\,\sigma_*}{\mu\sqrt{m}}.
\end{equation}
It remains to bound $s_k$ in \eqref{eqn:per-iterate bound}, where $0\leq k\leq N.$ For all $1\le k\le N,$ by the batch size condition $m$ \eqref{eq:inner-loop-D} and parameter choices \eqref{eqn:stepsize-inner}, we have
\begin{align*}
\tfrac{\theta_{k-1}\eta_{k-1}}{12}
&\overset{\eqref{eq:inner-loop-D}}{\geq}
\tfrac{\varsigma_*}{m}
\left[
\theta_{k}(\eta_{k}\lambda_{k})^2
+
\theta_{k+1}(\eta_{k+1}\lambda_{k+1})^2
\right]
\overset{\eqref{eqn:stepsize-inner}}{=}
\tfrac{16\varsigma_*}{m\mu^2}
\left[
\tfrac{k+k_{0}-1}{k+k_{0}}
+
\tfrac{k+k_{0}}{k+k_{0}+1}
\right]
\geq \tfrac{16\varsigma_*}{m\mu^2}.
\end{align*}
Hence, we have
\begin{equation*}
\begin{aligned}
s_k
&\leq
\tfrac{\mu^2}{16}\cdot\tfrac{\theta_{k-1}\eta_{k-1}}{12}
\mathbb{E}\left[
\langle F(x_k)-F(x^*),x_k-x^*\rangle
\right]
\overset{\eqref{eqn:cor}}{\leq}
\tfrac{\mu^2B}{256},
\quad\forall\,\,k\in\{1,\dots,N\}.
\end{aligned}
\end{equation*}
For $s_0$, by the Lipschitz continuity of $F$,
$\tfrac{\eta_1\varsigma_*}{m}\leq\tfrac{\lambda_1}{24}$,
$\lambda_1\leq1$,
$\eta_1=\tfrac{4}{\mu k_0}$, and
$k_0=\tfrac{16L}{\mu}$, we have
\begin{equation*}
\begin{aligned}
s_0
\leq \tfrac{\varsigma_* L}{m}\|x_0-x^*\|^2
\leq \tfrac{\lambda_1L}{24\eta_1}\|x_0-x^*\|^2
=\tfrac{\lambda_1\mu Lk_0}{96}\|x_0-x^*\|^2
\leq \tfrac{L^2}{6}\|x_0-x^*\|^2
\leq \tfrac{\mu^2B}{256}.
\end{aligned}
\end{equation*}
Hence $\sqrt{s_k}\leq \mu\sqrt{B}/16,$ for all $k\in\{0,\dots,N\}.$
Combining it with \eqref{eqn:cor} and \eqref{eqn:per-iterate bound},
we obtain
\begin{equation*}
\sqrt{\mathbb{E}\left[\res_F(x_N)^2\right]}
\leq \tfrac{\mu\sqrt{B}}{16}
\left(
\tfrac{4N+5k_0}{\sqrt{\theta_{N-1}}}
+\tfrac{k_0}{\sqrt{\theta_{N-2}}}
+3
\right)
+\tfrac{3\sigma_*}{\sqrt{m}}
\overset{\text{(i)}}{\leq}
\tfrac{13\mu\sqrt{B}}{16}
+\tfrac{3\sigma_*}{\sqrt{m}},
\quad\forall\,\,N\geq 2,
\end{equation*}
where in (i), we used \eqref{eqn:residual-inter-1}, which implies
\begin{equation*}
\tfrac{4N+5k_0}{\sqrt{\theta_{N-1}}}\leq 8
\quad\text{and}\quad
\tfrac{k_0}{\sqrt{\theta_{N-2}}}\leq 2.
\end{equation*}
Combining it with the bound on $B$ in \eqref{eqn:B} concludes the proof.
\end{proof}
\vgap

For any $s\geq 1$ and $k\geq 2$, let $R_s$ be independent of the samples generated during epoch $s$ and satisfy
\begin{equation}\label{eqn:iterative-averaging-per-epoch}
\mathbb{P}(R_s=j)
=
\tfrac{j+k_{s,0}-1}
{\tsum_{\ell=1}^{k}(\ell+k_{s,0}-1)}
=
\tfrac{2(j+k_{s,0}-1)}
{k(k+2k_{s,0}-1)},
\qquad
j\in\{1,\ldots,k\}.
\end{equation}
Applying \autoref{cor:subroutine-guarantee} and \autoref{lem:slow-residual-last} to $F_s$ shows that SOE can serve as the subroutine $\mathcal{A}$ in \autoref{assumption_inner_sto-state-dependent}. In fact, the following result establishes stronger guarantees for SOE than those required by \autoref{assumption_inner_sto-state-dependent}.
\begin{proposition}\label{prop:subroutine}
Suppose that the conditions of \autoref{cor:subroutine-guarantee} are satisfied when the Lipschitz and strongly monotone parameters therein are specialized to $L+r_s$ and $\mu+r_s$, respectively. Then SOE exhibits the following performance guarantee after $k$ iterations: for all $k\geq 2$, it holds that
\begin{equation}
    \begin{aligned}
        &\tfrac{16 k}{(\mu+r_s)(k+k_{s,0})}\mathbb{E}\left[\left\langle F_s(x_{s,R_s})-F_s(x_s^*), x_{s,R_s}-x_s^*\right\rangle\,\middle|\,\mathcal{F}_{s,0}\right]+ \mathbb{E}\left[\|x_{s,k}-x_s^*\|^2\,\middle|\,\mathcal{F}_{s,0}\right]\\
 &\leq \tfrac{1301(L+ r_s )^2}{(\mu+r_s)^2k^2}\|x_{s-1}-x_s^*\|^2 +\tfrac{1024 \sigma^2_{s,*}}{km_s (\mu+r_s) ^2}.   \label{eqn:distance-s'} 
    \end{aligned}
\end{equation}
Furthermore, we have 
\begin{align}
      &\sqrt{\mathbb{E}\left[\res_{F_s}(x_{s,k})^2\,\middle|\,\mathcal{F}_{s,0}\right]}\leq {28(L+ r_s )}\|x_{s-1}-x_s^*\|+\tfrac{29\sqrt{k} \sigma_{s,*}}{\sqrt{m_s}},\label{eqn:residual-s''}
  \end{align}
where $ \sigma^2_{s,*}\coloneqq \varsigma_{*} r_s  \|x^*_s - \bar x_s\| \|x_s^* - x^*\| +  \sigma^2_*$.
\end{proposition}
\begin{proof}
Observe that $F_s$ is Lipschitz continuous with constant $L+r_s$ and $(\mu+r_s)$-strongly monotone. Moreover, \autoref{lemma:state-dependent-noise-regularized} shows that the regularized operator $F_s$ satisfies the monotone-operator assumptions in \autoref{sec:Assumptions} with the corresponding modified parameters. Applying \autoref{cor:subroutine-guarantee} together with \autoref{remark:another-output} to $F_s$ using these parameters and the definition of $R_s$ in \eqref{eqn:iterative-averaging-per-epoch}, dividing both sides of \eqref{eqn:subroutine-guarantee} by
$(k+k_{s,0})(k+k_{s,0}-1)$, and using
\begin{equation*}
(k+k_{s,0})(k+k_{s,0}-1)\geq k^2,\qquad
k_{s,0}=\tfrac{16(L+r_s)}{\mu+r_s}+1,\qquad
x_{s,0}=x_{s-1},\qquad
k-1\leq k,
\end{equation*}
we obtain \eqref{eqn:distance-s'}.
It remains to verify the residual guarantees. When
\autoref{lem:slow-residual-last} is applied to the regularized operator $F_s$ with Lipschitz constant $L+r_s$, strong monotonicity parameter $\mu+r_s$, initial point $x_{s-1}$, solution $x_s^*$, batch size $m_s$, and variance proxy $\sigma_{s,*}^2$, its last-iterate bound \eqref{eqn:res-last} yields \eqref{eqn:residual-s''}.

\end{proof}

\section{High-Probability Convergence Guarantees for Accumulative Regularization}

In this section, we will derive the complexities of AR framework for  solving the VI \eqref{VIP} with $\mu\ge0$ with high probability under light-tailed noise assumptions. We show that for any $p \in (0,1)$, AR produces a search point $x_S$ such that $\mathbb P(\res_F(x_S) \leq \varepsilon) \geq 1 -p$. 

In \autoref{sec:hp-ar-state-dep}, we introduce the light-tailed state-dependent noise assumption and establish the high-probability convergence rate of the AR framework when equipped with a generic subroutine $\mathcal A$ satisfying subroutine guarantees. Accordingly, in \autoref{sec:hp-soe-state-dependent-new} we verify that SOE satisfies these subroutine guarantees. Moreover, in \autoref{sec:hp-uniform}, we extend the result to the simpler light-tailed uniform noise case.

\subsection{Accumulative regularization for state-dependent light-tailed noise}\label{sec:hp-ar-state-dep}

In this section, we establish a high-probability convergence guarantee under the state-dependent noise model. To this end, we assume that the stochastic oracle $\mathcal{SO}$ returns the unbiased estimator in \eqref{eq:fast-decreasing-bias-o} and satisfies the following light-tailed assumption on the stochastic gradient error.
\smallskip

For any fixed solution $x^*\in X^*$, there exist constants $\sigma_*^2\geq 0$ and $\varsigma_*\geq 0$ such that, for any $\mathcal F_k$-measurable point $x\in X$, the following sub-Gaussian state-dependent variance condition holds:
\begin{align}\label{eq:hp-bounded-var-at-opt-o}
\mathbb{E}\left[\exp\left\{\tfrac{\|\tilde{F}(x,\zeta_k)-F(x)\|^2}{\sigma_*^2+\varsigma_*\langle F(x)-F(x^*),x-x^*\rangle}\right\}\,\bigg|\,\mathcal{F}_k\right]\leq \exp\{1\}, \quad \textnormal{a.s.}
\end{align}

\medskip

We start with the following result for the stochastic regularized operator $\widehat F_{s,k}$ define in \eqref{eqn:regulairzed-stochastic-estimator}.
\begin{lemma}\label{lem:hp-regularized-operator-properties}
For every $\mathcal F_{s,k}$-measurable point $x\in X,$ the estimator $\widehat F_{s,k}(x)$ satisfies the unbiasedness condition \eqref{eq:fast-decreasing-bias}; namely, $\mathbb E[\widehat F_{s,k}(x)\mid\mathcal F_{s,k}]=F_s(x).$ Furthermore, let
\begin{equation*}
\sigma_{s,*}^2\coloneqq 9\left[\varsigma_*r_s\|x_s^*-\bar x_s\|\cdot\|x_s^*-x^*\|+\sigma_*^2\right].
\end{equation*}
Then the oracle error at $x_s^*$ satisfies the conditional exponential-moment bound
\begin{equation}\label{eq:hp-state-dep-var-at-subproblem-opt}
\mathbb E\left[\exp\left\{\tfrac{m_s\|\widehat F_{s,k}(x_s^*)-F_s(x_s^*)\|^2}{\sigma_{s,*}^2}\right\}\,\middle|\,\mathcal F_{s,k}\right]\leq\exp(1),\quad\textnormal{a.s.}
\end{equation}
\end{lemma}

\begin{proof}
The unbiasedness condition \eqref{eq:fast-decreasing-bias} follows from \autoref{lemma:state-dependent-noise-regularized}. Moreover, by the definition of $\widehat F_{s,k}$,
\begin{equation*}
\widehat F_{s,k}(x_s^*)-F_s(x_s^*)=\tilde F_{s,k}(x_s^*)-F(x_s^*).
\end{equation*}
Under the light-tail assumption \eqref{eq:hp-bounded-var-at-opt-o}, \autoref{prop:moment-tail-bound} implies that, for every $\mathcal F_{s,k}$-measurable point $x\in X,$
\begin{equation*}
\mathbb E\left[\exp\left\{\tfrac{m_s\|\tilde F_{s,k}(x)-F(x)\|^2}{9[\sigma_*^2+\varsigma_*\langle F(x)-F(x^*),x-x^*\rangle]}\right\}\,\middle|\,\mathcal F_{s,k}\right]\leq\exp(1),\quad\textnormal{a.s.}
\end{equation*}
By the definitions of $\mathcal F_{s,0}$ and $x_s^*$, the point $x_s^*$ is $\mathcal F_{s,0}$-measurable. Since $\mathcal F_{s,0}\subseteq\mathcal F_{s,k}$, the preceding bound applies with $x=x_s^*$.

We next bound the resulting sub-Gaussian parameter. By the definitions of the variational inequalities for $x^*$ and $x_s^*$, the definition of $F_s$, and the Cauchy--Schwarz inequality,  we have
\begin{equation*}
\begin{aligned}
\sigma_*^2+\varsigma_*\langle F(x_s^*)-F(x^*),x_s^*-x^*\rangle
&\leq \sigma_*^2+\varsigma_*\langle F(x_s^*),x_s^*-x^*\rangle\\
&\leq \sigma_*^2+\varsigma_*\langle F(x_s^*)-F_s(x_s^*),x_s^*-x^*\rangle\\
&=\sigma_*^2+r_s\varsigma_*\langle\bar x_s-x_s^*,x_s^*-x^*\rangle\\
&\leq \sigma_*^2+r_s\varsigma_*\|x_s^*-\bar x_s\|\cdot\|x_s^*-x^*\|.
\end{aligned}
\end{equation*}
Since $\bar x_s$ is $\mathcal F_{s,0}$-measurable, so is $\sigma_{s,*}^2$. In particular, $\sigma_{s,*}^2$ is $\mathcal F_{s,k}$-measurable. Therefore, the denominator in the preceding exponential-moment bound at $x=x_s^*$ is at most $\sigma_{s,*}^2$. This concludes the proof.
\end{proof}
\medskip

Unlike under the uniform noise model, the conditional sub-Gaussian variance proxy $\sigma_*^2+\varsigma_*\langle F(x)-F(x^*),x-x^*\rangle$ depends on the iterate $x$ and is therefore random. Hence, the uniform-noise high-probability argument based on the deterministic variance proxy $\sigma^2$ cannot be applied directly; we instead develop a state-dependent analysis for stochastic AR and its inner subroutine, with SOE as a concrete example.  Controlling the cumulative squared oracle error and the associated martingale sum requires self-normalized concentration, which produces logarithmic variance-ratio terms encoded by $q_s$. These terms capture the additional error induced by the randomness of the variance proxy.

To formulate the required subroutine guarantee for AR, fix $\Lambda\geq0$ and $N_s\geq2$. Let $c_{\mathcal A}>1$ be a constant depending only on $\mathcal A$, and let $\{q_s\}_{s\geq0}$ satisfy $q_s\leq1+\Lambda+\log(N_s+\tfrac{L+r_s}{\mu+r_s})$. For every $2\leq k\leq N_s$, define the events
\begin{align}
E_{s,k}^1&\coloneqq\left\{\|x_{s,k}-x_s^*\|^2\leq c_{\mathcal A}q_s\left[\tfrac{(L+r_s)^2}{(\mu+r_s)^2k^2}\|x_{s-1}-x_s^*\|^2+\tfrac{(1+\Lambda)\sigma_{s,*}^2}{km_s(\mu+r_s)^2}\right]\right\},\label{eqn:hp-distance-s-state-dependent}\\
E_{s,k}^2&\coloneqq\left\{\res_{F_s}(x_{s,k})\leq c_{\mathcal A}\sqrt{q_s}\left[(L+r_s)\|x_{s-1}-x_s^*\|+\tfrac{\sqrt{k}\sqrt{1+\Lambda}\sigma_{s,*}}{\sqrt{m_s}}\right]\right\}.\label{eqn:hp-residual-s-state-dependent}
\end{align}
Compared with their counterparts under the uniform noise model, the distance and residual bounds in \eqref{eqn:hp-distance-s-state-dependent} and \eqref{eqn:hp-residual-s-state-dependent} contain the additional factors $q_s$ and $\sqrt{q_s}$, respectively. These factors arise from controlling the random sub-Gaussian parameter along the iterates.

With this notation, we state the following performance assumption on the subroutine $\mathcal A$ used by AR.

\begin{assumption}\label{as:hp-assumption_inner_sto-state-dependent}
For any mini-batch size satisfying $m_s\geq\lceil\tfrac{(1+\Lambda)\varsigma_*c_{\mathcal A}q_{s}}{L+r_s}\rceil$, consider the iterates generated by the subroutine $\mathcal A(\mathcal{SO},r_s,\bar x_s,x_{s-1},m_s)$ for solving the regularized problem \eqref{eq:sub-problem}. Suppose the following holds:
\begin{equation}
\mathbb P\left(E_{s,k}^1\cap E_{s,k}^2\,\middle|\,\mathcal F_{s,0}\right)\geq1-2k\exp\{-\Lambda\},\quad\textnormal{a.s.} \quad\forall\,\,2\leq k\leq N_{s}
\end{equation}
\end{assumption}
If $\sigma_*=0$ and $\varsigma_*=0$, we set $q_s\equiv1$, and the analysis reduces to the deterministic analysis in \autoref{the:sublinear_AR}. In \autoref{sec:hp-soe-state-dependent-new}, we show that SOE satisfies \autoref{as:hp-assumption_inner_sto-state-dependent}.

As in the in-expectation setting, \autoref{as:hp-assumption_inner_sto-state-dependent} requires a sufficiently large mini-batch size $m_s$ to control the state-dependent noise. Moreover, since the random sub-Gaussian parameter $\sigma_{s,*}^2$ contains a term proportional to the regularization parameter $r_s$, the following lemma provides a deterministic upper bound on $\sigma_{s,*}^2$ over the good event $\mathcal E_{s-1}$.
\begin{lemma}\label{lem:hp-state-dependent-iterate-boundedness}
For any $s\geq2,$ define the event
\begin{equation}\label{eq:hp-induction-event}
\mathcal E_s := \left\{\|x_i - x_i^*\|^2 \leq \tfrac{D_0^2}{16^i}, \quad \forall i \in [s]\right\}
\end{equation}
for each $s \geq 1.$ Moreover, define
\begin{equation*}
\mathcal E_{\sigma_{s,*}}\coloneqq\left\{\sigma_{s,*}^2\leq24\varsigma_*D_0^2\tsum_{i=1}^s\tfrac{r_i+r_{i-1}}{4^{i-1}}+9\sigma_*^2\right\}.
\end{equation*}
Then $\mathcal E_{s-1}\subseteq\mathcal E_{\sigma_{s,*}}$.
\end{lemma}

\begin{proof}
On the event $\mathcal E_{s-1}$, \autoref{lemma_ARVI_proximity} and the definition of $\mathcal E_{s-1}$ yield
\begin{align*}
\|x_s^*-x^*\|&\overset{\eqref{eqn:induction-traingle}}{\leq}2\tsum_{k=0}^{s-1}\|x_k-x_k^*\|\leq2\tsum_{k=0}^{s-1}\tfrac{D_0}{4^k}\leq\tfrac{8D_0}{3},\\
r_s\|x_s^*-\bar x_s\|&\overset{\eqref{eq:epoch_proximity2}}{\leq}\tsum_{i=1}^s(r_i+r_{i-1})\|x_{i-1}-x_{i-1}^*\|\leq\tsum_{i=1}^s\tfrac{(r_i+r_{i-1})D_0}{4^{i-1}}.
\end{align*}
Consequently, by the definition of the sub-Gaussian parameter $\sigma_{s,*}^2$,
\begin{align*}
\sigma_{s,*}^2&=9\left[\varsigma_*r_s\|x_s^*-\bar x_s\|\cdot\|x_s^*-x^*\|+\sigma_*^2\right]\leq24\varsigma_*D_0^2\tsum_{i=1}^s\tfrac{r_i+r_{i-1}}{4^{i-1}}+9\sigma_*^2.
\end{align*}
Thus, $\mathcal E_{s-1}\subseteq\mathcal E_{\sigma_{s,*}}$.
\end{proof}
\smallskip
We next present the convergence rate for AR.
Since the subroutine guarantees in \autoref{as:hp-assumption_inner_sto-state-dependent} contain an additional factor $q_s$ in the distance bound and $\sqrt{q_s}$ in the residual bound compared with the uniform-noise case, we need to enlarge the AR epoch length $N_s$, batch size $m_s$, and the number of total epochs accordingly to account for that. For simplicity, we first introduce a few quantities to reflect the effect of $q_s.$
Let $p\in(0,1)$, $0<\varepsilon\leq LD_0$ and $D_0 \ge \operatorname{dist}(x_0, X^*)$, define
\begin{equation}\label{eq:parameter-bounds}
\begin{aligned}
\bar S&:=\left\lceil3+\log_4\left(\tfrac{16(16c_{\cA}+4)LD_0}{p\varepsilon}S_0\log(4S_0)\right)+\tfrac{1}{2}\log\left(\tilde c_{\cA}\left(1+\log\tfrac{LD_0}{p\varepsilon}\right)\right)\right\rceil,
\end{aligned}
\end{equation}
where
\begin{equation}\label{eq:hp-state-dependent-def-tildecA}
S_0 \coloneqq \left\lceil\log_4 \tfrac{16(16c_{\cA} +4)LD_0}{p\varepsilon}\right\rceil\quad\text{and}\quad \tilde c_{\cA}\coloneqq16\log_4(16(16c_{\cA}+4)).
%+\tfrac{8}{3}\left[4+2\log(8\sqrt{2c_{\cA}})+4\log(32(c_{\cA}+4))\right].
\end{equation}
Additionally, define the base iteration count by
\begin{equation*}
\begin{aligned}
N_{\textnormal{base}}&:=\left\lceil\tfrac{4\sqrt{2c_{\cA}}(L+\underline r)}{\mu + \underline r}\right\rceil,\quad\text{where}\quad \underline r:=\tfrac{\varepsilon}{16(16c_{\cA}+4)D_0\log\bar S}.
\end{aligned}
\end{equation*}
Furthermore, define
\begin{equation}\label{eq:Qbar-Nbar-Lambda}
\begin{aligned}
\bar q&:=2\left(1+\log_4\left(\tfrac{16\bar S N_{\textnormal{base}}^2}{p}\right)\right)\quad\text{and}\quad \bar N:=\left\lceil N_{\textnormal{base}}\sqrt{\bar q}\right\rceil.
\end{aligned}
\end{equation}
We are now ready to provide the high-probability convergence guarantee result for AR when $\mu \ge 0$.

\begin{theorem}\label{thm:hp-state-dep-main-ar-new} 
Assume that $0<\varepsilon \le LD_0$, where $D_0\geq \operatorname{dist}(x_0, X^*)$. In \autoref{alg:ARVI}, suppose $S$ and $r_s$ satisfy
\begin{equation}\label{eq:hp-state-dependent-ar-regularization}
S = \left\lceil 2 + \log_4 \left( \tfrac{16(16c_{\cA} + 4) LD_0}{p\varepsilon}S_0 \log(4S_0) \sqrt{\bar q} \right)\right\rceil,\quad r_s:=\tfrac{4^{s-1}\varepsilon}{16(16c_{\cA}+4)D_0s\log S},
\end{equation}
where $S_0$ is defined in \eqref{eq:parameter-bounds} and $\bar q$ is defined in \eqref{eq:Qbar-Nbar-Lambda}.
 For each epoch, the epoch length satisfies
\begin{align}\label{hp-state-dep-epoch-length}
N_s = \left\lceil N_{s,0} \sqrt{\bar q} \right\rceil\quad\text{where}\quad N_{s,0} = \left\lceil\tfrac{4\sqrt{2c_{\cA}}(L+ r_s)}{r_s}\right\rceil,
\end{align}
and the batch size satisfies 
\begin{align}\label{eqn:hp-state-dep-batch-epoch}
m_s=\max\left\{\left\lceil\tfrac{16^{s-1}N_s(1+\Lambda_{p, \varepsilon})}{(L+ r_s)^2 D_0^2} \left[24\varsigma_* D_0^2\tsum_{i = 1}^s \tfrac{r_{i} + r_{i-1}}{4^{i - 1}}+9\sigma_*^2\right]+1\right\rceil,\ \tfrac{(1+\Lambda_{p, \varepsilon})\varsigma_* c_{\cA} \bar q}{L + r_s}\right\}, \quad \Lambda_{p, \varepsilon}:= \log \left(\tfrac{2\bar S \bar N}{p}\right).
\end{align}
Then, at epoch $S$, \autoref{alg:ARVI} computes an approximate solution $x_S$ such that $\mathbb P(\res_F(x_S) \leq \varepsilon) \geq 1 - p$ after 
\begin{align}\label{eq:hp-state-dependent-ar-complexity}
\widetilde{\mathcal{O}}\left(\tfrac{LD_0}{\varepsilon}\log(\tfrac{LD_0}{p\varepsilon})^{1/2} + \tfrac{\varsigma_* D_0}{\varepsilon} (\log \tfrac{LD_0}{p\varepsilon})^5 + \tfrac{\sigma_*^2}{\varepsilon^2}(\log \tfrac{LD_0}{p\varepsilon})^5\right)
\end{align}
calls to $\mathcal{SO}$.
\end{theorem}

\begin{proof}
Observe that by \autoref{lem:bar-q-bound}, $\bar q \le \tilde c_{\cA}(1 + \log \tfrac{LD_0}{p \varepsilon})$, hence we have $S \le \bar S$, and thus $\bar r\leq r_1\leq r_s$ since $r_s$ is increasing in $s$. Hence, $N_{s,0} \le N_{\text{base}}$ and $N_s \le \bar N$. Thus, by the definition of $q_s$,
\begin{equation}\label{eqn:upper-q}
\begin{aligned}
q_s &\le 1 + \Lambda_{p, \varepsilon} + \log \left(N_s + \tfrac{L + r_s}{r_s}\right) \overset{\text{(i)}}{\le} 1 + \Lambda_{p, \varepsilon} + \log(2\bar N)\overset{\text{(ii)}}{\leq} 1 + \log\left(\tfrac{4 \bar S \bar N^2}{p}\right) \overset{\text{(iii)}}{\leq}  1 + \log\left(\tfrac{16 \bar S N_{\text{base}}^2}{p}\right) + \log \bar q\overset{\text{(iv)}}{\leq} \bar q,
\end{aligned}
\end{equation}
where in (i) we have used $\tfrac{L + r_s}{r_s}\le N_s \le \bar N$; in (ii) we have substituted the definition of $\Lambda_{p, \varepsilon}$ from \eqref{eqn:hp-state-dep-batch-epoch};  in (iii) we used that $\bar N  \le 2N_{\text{base}} \sqrt{\bar q}$; and in (iv) we used the definition of $\bar q$ in \eqref{eq:Qbar-Nbar-Lambda} with the fact that $\log \bar q \le \bar q /2$. Consequently, by the definition of $N_s$ in \eqref{hp-state-dep-epoch-length},
\begin{equation}\label{eq:ratio-qNs-to-Ns}
\tfrac{q_s}{N_s^2} \leq \tfrac{1}{N_{s,0}^2}, \quad \forall s\in [S]. 
\end{equation}
Define $\mathcal E_s$ as in \eqref{eq:hp-induction-event}. 
We show by induction that 
\begin{equation}\label{eq:hp-AR-induction}
\mathbb P(\mathcal E_s) \geq 1 - 2M_s \exp(-\Lambda_{p, \varepsilon}), \quad \forall s \in [S]
\end{equation}
where $ M_s :=\tsum_{i = 1}^s N_i,$ and $\mathcal E_s$ is the event defined in \eqref{eq:hp-induction-event}. Indeed for $s = 1$, we have
\begin{equation*}
\begin{aligned}
\|x_1 - x_1^*\|^2 \overset{\text{(v)}}{\leq} \tfrac{c_{\cA}(L +  r_1 )^2}{( \mu + r_1)^2N_1^2}\|x_{1,0} - x_1^*\|^2 + \tfrac{c_{\cA}(1+\Lambda_{p, \varepsilon})  \sigma_{1,*}^2}{(\mu +  r_1)^2m_1N_1} \overset{\text{(vi)}}{\leq} \tfrac{2c_{\mathcal A}(L+ r_1 )^2D_0^2}{ r_1 ^2N_{1,0}^2}  \overset{\text{(vii)}}{\leq} \tfrac{D_0^2}{16}.
\end{aligned}
\end{equation*}
where in (v) we have used \autoref{as:hp-assumption_inner_sto-state-dependent}, \eqref{eqn:hp-distance-s-state-dependent} and \eqref{eq:ratio-qNs-to-Ns}; in (vi) we have used the definition of $m_s$ in \eqref{eqn:hp-state-dep-batch-epoch} along with \autoref{lem:hp-state-dependent-iterate-boundedness}, and in (vii) we have used the definition of $N_{1,0}$ in \eqref{hp-state-dep-epoch-length}. 
\smallskip

Suppose now that $\mathbb P(\mathcal E_{s-1}) \ge 1 - 2M_{s-1}\exp(-\Lambda_{p, \varepsilon})$. Then on $\mathcal E_{s - 1}$, by \eqref{eq:ratio-qNs-to-Ns} and \autoref{as:hp-assumption_inner_sto-state-dependent}, we have
\begin{equation*}
\|x_s - x_s^*\|^2 \overset{\eqref{eqn:hp-distance-s-state-dependent}}{\leq} \tfrac{c_{\mathcal A}(L+  r_s )^2}{(  \mu+r_s )^2N_{s,0}^2}\|x_{s - 1} - x_{s-1}^*\|^2 + \tfrac{c_{\mathcal A}(1+\Lambda_{p, \varepsilon}) \sigma_{s, *}^2}{(\mu + r_s)^2 m_sN_{s,0}} \overset{\text{(viii)}}{\leq}\tfrac{2c_{\mathcal A}(L+ r_1 )^2D_0^2}{ r_s ^2N_{s,0}^2} \overset{\text{(ix)}}{\leq} \tfrac{D_0^2}{16^s}.
\end{equation*}
where in (viii) we use the inductive hypothesis and the fact that on $\mathcal E_{s-1}$, by \autoref{lem:hp-state-dependent-iterate-boundedness}, we can bound  $\sigma_{s,*}^2\leq24\varsigma_*D_0^2\tsum_{i=1}^s\tfrac{r_i+r_{i-1}}{4^{i-1}}+9\sigma_*^2$; in (ix) we apply the choice of $N_{s,0}$. Then, using a union bound, we have shown through induction that for all $s \in [S]$, $\mathbb P(\mathcal E_s) \geq 1 - 2M_s\exp(-\Lambda_{p, \varepsilon})$. 
\smallskip

To bound the residual on $\mathcal{E}_s$, we combine the residual decomposition \eqref{eqn:res}, the residual bound \eqref{eqn:hp-residual-s-state-dependent} in \autoref{as:hp-assumption_inner_sto-state-dependent}, and the bound $q_s \le \bar q$ in \eqref{eqn:upper-q} to obtain
\begin{equation}\label{eq:hp-sd-residual-decomp-full}
\begin{aligned}
&\res_{F}(x_S)\\
&\,\le c_{\cA}\sqrt{\bar q}\left[(L+r_S)\|x_{S-1}-x_S^*\|+\tfrac{\sqrt{N_S}\sqrt{1+\Lambda_{p, \varepsilon}}\sigma_{S,*}}{\sqrt{m_S}}\right] +r_S\|x_S -x_S^*\|+\tsum_{s = 1}^S ( r_s  +  r_{s - 1}) \|x_{s - 1} - x_{s - 1}^*\| \\
&\overset{\text{(x)}}{\le} \tfrac{c_{\mathcal{A}}\sqrt{\bar q}(L+ r_S ) D_0}{4^{S-1}}+\tfrac{c_{\mathcal{A}} \sqrt{\bar q}\sigma_{S,*}\sqrt{1 + \Lambda_{p, \varepsilon}}\sqrt{N_S}}{\sqrt{m_S}}  + \tsum_{s = 1}^S \tfrac{( r_s  +  r_{s - 1})D_0}{4^{s - 1}}+ \tfrac{ r_S  D_0}{4^S} \\
&\overset{\text{(xi)}}{\le} \tfrac{4c_{\cA} \sqrt{\bar q}r_SD_0}{4^{S-1}} + \tfrac{5D_0}{4}\tsum_{s=1}^S \tfrac{r_s}{4^{s - 1}},
\end{aligned}
\end{equation}
where in (x) we used \autoref{lemma_ARVI_proximity}, \eqref{eq:epoch_proximity1}, \eqref{eq:epoch_proximity2} and the induction; in (xi) we used the choice of $m_S$, and the fact that $r_0 = 0$ and $r_S \ge L$, 
the latter of which which will be established after we bound the residual.

We proceed with with bounding the first term in the residual.

\begin{equation}\label{eq:hp-sd-residual-part-i-bound}
\tfrac{4c_{\cA}\sqrt{\bar q}r_SD_0}{4^{S-1}} \overset{\eqref{eq:hp-state-dependent-ar-regularization}}{=}\tfrac{c_{\cA}\varepsilon}{4(16c_{\cA} + 4)} \tfrac{\sqrt{\bar q}}{S \log S} \overset{\text{(xii)}}{\le} \tfrac{c_{\cA}\varepsilon}{16c_{\cA} + 4} \le \tfrac{\varepsilon}{16}, 
\end{equation}
where in (xii), we bound $\sqrt{\bar q}/(S\log S)$ as follows. By Young's inequality and \autoref{lem:bar-q-bound}, we have
$$
\sqrt{\bar q} \le \sqrt{\tilde c_{\cA}} \sqrt{1 + \log_4 \tfrac{L D_0}{p \varepsilon}} \leq\tfrac{\tilde c_{\cA}}{8} + 2(1 + \log_4 \tfrac{LD_0}{p \varepsilon}).
$$
Next, to bound $S \log S$ from below, we use the definition of $S$ in \eqref{eq:hp-state-dependent-ar-regularization} to get
$$
S \log S \ge S \ge 2 + \log_4(\tfrac{16(16c_{\cA} + 4) LD_0}{p \varepsilon}) {\ge} \tfrac{\tilde c_{\cA}}{16} + (1 + \log_4 \tfrac{LD_0}{p \varepsilon}),
$$
where the last inequality follows from the definition of the $\tilde{c}_{\mathcal{A}}$ in \eqref{eq:hp-state-dependent-def-tildecA}.
Hence, $\sqrt{\bar q}/(S\log S) \le 4$ due to $\tfrac{a+b}{c+d} \le \tfrac{a}{c} + \tfrac{b}{d}$ for $a,b,c,d>0$.
\smallskip

We next bound the second term in the residual.
\begin{align}\label{eq:hp-sd-residual-part-ii-bound}
\tsum_{s=1}^{S}\tfrac{ r_s }{4^{s-1}}
&\overset{\eqref{eq:hp-state-dependent-ar-regularization}}{=}
\tsum_{s=1}^{S}\tfrac{4^{s-1}\varepsilon}{16\left(16c_{\mathcal{A}}+4\right)D_0s\log S}\left(\tfrac{1}{4}\right)^{s-1}
\le
\tfrac{\varepsilon(1+\log S)}{16\left(16c_{\mathcal{A}}+4\right)D_0\log S}
\le
\tfrac{\varepsilon}{64D_0},
\end{align}
where in the final inequality, we used that $\tfrac{1+\log S}{\log S} \le 2$ and $16(16c_{\cA} + 4)\ge 64$. 

Substituting \eqref{eq:hp-sd-residual-part-i-bound} and \eqref{eq:hp-sd-residual-part-ii-bound} into \eqref{eq:hp-sd-residual-decomp-full} gives that on $\mathcal{E}_S$, $\res_F(x_S) \le \tfrac{\varepsilon}{16} + \tfrac{5\varepsilon}{256} < \varepsilon. $
\smallskip

It remains to show $r_S \ge L$ in \eqref{eq:hp-sd-residual-decomp-full}. By the definition of $S$ in \eqref{eq:hp-state-dependent-ar-regularization}, we have
$$S \ge 2+S_0 + \log_4({S_0\log_4(S_0)\sqrt{\bar q}}),\,\,\text{and}\,\, 
4^S \ge \tfrac{256(16 c_{\cA} + 4)LD_0}{p \varepsilon}S_0 \log_4(4S_0)\sqrt{\bar q}.
$$
Using again the definition of $S$ in \eqref{eq:hp-state-dependent-ar-regularization}, we have the following upper bound
$$
S \le 3 +  S_0+ \log_4 \left(S_0\log_4(4S_0) \sqrt{\bar q} \right) = 3 + S_0 + \log_4(S_0) + \log_4(\log_4(4S_0)) +\log_4(\sqrt{\bar q})\le 4S_0 + \log_4 \sqrt{\bar q},
$$
where in the final inequality, we have used that $3 \le S_0$, $\log_4(S_0) \le S_0$, and $\log_4 \log_4 (4S_0) \le S_0$. Substituting the bounds for $S$ and $4^S$ into the definition of $r_S$ in \eqref{eq:hp-state-dependent-ar-regularization}, we have
$$
r_S = \tfrac{\varepsilon}{64(16 c_{\cA} + 4)D_0} \tfrac{4^{S}}{S \log S} \ge\tfrac{4LS_0 \log(4S_0)\sqrt{\bar q}}{(4S_0 + \log_4\sqrt{\bar q}) \log(4S_0 +\log_4\sqrt{\bar q})} \geq L,
$$
where the last inequality holds by the fact that $\tfrac{4S_0 \log(4S_0)\sqrt{\bar q}}{(4S_0 + \log_4\sqrt{\bar q}) \log(4S_0 +\log_4\sqrt{\bar q})} \ge 1$ due to $S_0 \ge 1$ and $\bar q \ge 4$. 

To characterize the probability, by the fact that $S \le \bar S$ and $N_s \leq \bar N$ along with the definition of $\Lambda_{p, \varepsilon}$ in \eqref{eqn:hp-state-dep-batch-epoch}, it follows that
$M_S = \tsum_{s = 1}^SN_s \leq \bar S\bar N $ and $2M_S\exp(-\Lambda_{p, \varepsilon}) \leq p$.
\smallskip

It remains to specify the complexity. Notice that 
\begin{equation*}
    \begin{aligned}
      \tfrac{24\varsigma_*D_0^2\,\tsum_{i = 1}^s \tfrac{ r_{i} +  r_{i-1}}{4^{i-1}} + 9\sigma_*^2}{\varepsilon^2}=\mathcal{O}\left( \tfrac{\varsigma_* D_0}{\varepsilon}+ \tfrac{\sigma_*^2}{\varepsilon^2} \right).
    \end{aligned}
\end{equation*}
Since $\bar q = \mathcal{\widetilde O} \left( \log \tfrac{LD_0}{p \varepsilon}\right)$ by \autoref{lem:bar-q-bound} and $\Lambda_{p, \varepsilon} = \widetilde {\mathcal O}\left(\log \tfrac{LD_0}{\varepsilon p}\right)$,
the sample complexity is
\begin{align*}
\tsum_{s = 1}^S N_sm_s & \le \tsum_{s=1}^S \left[ \tfrac{4\sqrt{2c_{\cA}}(L + r_s)\sqrt{\bar q}}{r_s}+1\right]\left[2 + \tfrac{16^{s - 1}N_s(1 + \Lambda_{p, \varepsilon})}{(L + r_s)^2 D_0^2}[24 \varsigma_* D_0^2 \tsum_{i = 1}^s \tfrac{r_i + r_{i - 1}}{4^{i - 1}} + 9\sigma_*^2] + \tfrac{(1 +\Lambda_{p, \varepsilon})\varsigma_* c_{\cA} \bar q)}{L + r_s}\right]\\
&= \widetilde{\mathcal{O}}\left(\tfrac{LD_0}{\varepsilon}\log(\tfrac{LD_0}{p\varepsilon})^{1/2} + \tfrac{\varsigma_* D_0}{\varepsilon} (\log \tfrac{LD_0}{p\varepsilon})^5 + \tfrac{\sigma_*^2}{\varepsilon^2}(\log \tfrac{LD_0}{p\varepsilon})^5\right). 
\end{align*}
\end{proof}

% We remark on the differences in the analysis from the uniform noise case. First, since the subroutine assumption in \eqref{eqn:hp-distance-s-state-dependent} includes the prefactor $q_s$ that scales both the deterministic and stochastic error, the iteration counts $N_s$ must be enlarged accordingly to solve the subproblems to accuracy $D_0/16^s$ to bound the state-dependent variance proxy $\sigma_{s, *}$ with its deterministic upper bound with high probability. Secondly, since the prefactor $\sqrt{q_s}$ appears in \eqref{eqn:hp-residual-s-state-dependent}, we require now running more total epochs $S$ to control the residual. This latter modification does not affect the overall complexity significantly, since the total number of calls to $\mathcal SO$ remains on the order of $N_1$. 

% It is worth noting that, if the parameter choices in \autoref{thm:hp-state-dep-main-ar-new} are applied directly in the noiseless regime $\varsigma_*=0$ and $\sigma_*^2=0$, the nominally deterministic term in \eqref{eq:hp-state-dependent-ar-complexity} retains an additional factor $\left(\log(LD_0/(p\varepsilon))\right)^{1/2}$. This factor is an artifact of the state-dependent high-probability parameterization rather than an intrinsic loss in the noiseless setting. Indeed, in this case, all stochastic-error terms vanish, $q_s$ can be chosen as a fixed constant and absorbed into $c_{\mathcal A}$, and $N_s$ can be set to $N_{s,0}$. The analysis then reduces to the deterministic case and recovers the deterministic complexity bound in \autoref{the:sublinear_AR}.
% \smallskip

In the strongly monotone case, we have the following result on the squared distance to optimal solution. 
\begin{corollary}\label{cor:state-dependent-strongly-monootne}
Assume that $F$ is $L$-Lipschitz continuous and $\mu$-strongly
monotone, and
$0<\varepsilon \le \|x_0-x^*\|^2 \leq D_0^2$. In \autoref{alg:ARVI}, suppose the parameters $S$ and $r_s$ satisfy \eqref{eqn:regularization-exp-unbiased-mu}. Define
$$
N_{s,0} ={\left\lceil\tfrac{4\sqrt{2c_{\mathcal{A}}}L}{{\mu}}\right\rceil}, \quad \Lambda_{p, \varepsilon} := 2\log\left(\tfrac{2S}{p}\right) + 3\log(N_{s,0}) + 2\log\left(\sqrt{\tfrac32} + \tfrac14\right). 
$$
For each epoch, suppose $N_s\coloneqq  \left\lceil \sqrt{\tfrac{3}{2}}N_{s,0}\sqrt{1+\Lambda_{p, \varepsilon} + \log(N_{s,0})}\right\rceil,$
and the batch size $m_s$ satisfies 
\begin{align}
m_s=\max\left\{\left\lceil\tfrac{32\cdot16^{s-1}c_{\mathcal{A}}(1+\Lambda_{p, \varepsilon}) \sigma_*^2q_s}{N_s\mu^2 D_0^2}\right\rceil,\, \tfrac{(1 + \Lambda_{p, \varepsilon})\varsigma_*  c_{\cA}q_s}{L},1\right\}.
\end{align}
Then, at epoch $S$, \autoref{alg:ARVI} computes a solution $x_S$ such that  $\mathbb P\{\|x_S - x^*\|^2
\le\varepsilon\} \geq 1 - p$ after
\begin{align}
\mathcal {\widetilde O}\left\{\tfrac{L}{\mu}\log\left(\tfrac{D_0^2}{\varepsilon}\right)\left(\log \tfrac{L}{p\mu}\right)^{1/2} + \tfrac{\varsigma_*}{\mu} \log\left( \tfrac{D_0^2}{\varepsilon}\right)\left(\log \tfrac{L}{p\mu}\right)^{5/2} + \tfrac{\sigma_*^2}{\mu^2 \varepsilon}\left(\log\tfrac {L}{p\mu}\right)^2\right\}
\end{align}
calls to the $\mathcal{SO}$.
\end{corollary}
\begin{proof}
We highlight the main ideas of the proof. It can be first verified that
\begin{equation}\label{eq:q_s-bound-strongly-monotone}
\begin{aligned}
q_s &= 1 + \Lambda_{p, \varepsilon} + \log(N_s + \tfrac{L}{\mu}) \\
&\le 1 + \Lambda_{p, \varepsilon} + \log\left(\sqrt{\tfrac32}N_{s,0}\sqrt{1 + \Lambda_{p, \varepsilon} + \log(N_{s,0})} + 1 + \tfrac{L}{\mu}\right) \\
&\overset{\text{(i)}}{\le} 1 + \Lambda_{p, \varepsilon} + \log(\sqrt{\tfrac32}+ \tfrac{1}{2}) + \log(N_{s,0}) + \tfrac{1}{2}\log(1 + \Lambda_{p, \varepsilon} + \log(N_{s,0})) \\
&\overset{\text{(ii)}}{\le} (1 + \log(\sqrt{\tfrac32}+1)) + \tfrac{3}{2}\Lambda_{p, \varepsilon} + \tfrac32\log(N_{s,0}),
\end{aligned}
\end{equation}
where in (i) we have used that $1 \le \tfrac{L}{\mu}$ and that $\tfrac{2L}{\mu} \le \tfrac{N_{s,0}}{2}$ due to $c_{\cA} \ge 1$; in (ii) we have used the fact that $\log(1+x)\le x$. Hence, 
$$
q_s \leq \tfrac{3}{2}(1+\Lambda_{p, \varepsilon} + \log(N_{s,0})), \qquad \tfrac{q_s}{N_s^2} \leq \tfrac{1}{N_{s,0}^2}. 
$$
Hence, by the same inductive strategy as in \autoref{thm:hp-state-dep-main-ar-new}, it is straightforward to show that with probability at least $1 - 2M_s \exp(-\Lambda_{p, \varepsilon})$, for each $s \in [S]$,
$$
\|x_s - x^*\|^2 \leq c_{\cA}q_s\left[\tfrac{L^2}{\mu^2N_s^2}\|x_{s-1}-x_s^*\|^2+\tfrac{(1+\Lambda_{p, \varepsilon})\sigma_{*}^2}{N_sm_s\mu^2}\right] \leq \tfrac{D_0^2}{16^s}. 
$$
At epoch $S$, it holds that with probability at least $1 - 2M_S\exp(-\Lambda_{p, \varepsilon})$, $\|x_S - x^*\|^2 \leq \varepsilon$. To characterize the probability, notice first that $M_S = S N_s$ and that $2M_S\exp(-\Lambda_{p, \varepsilon}) \le p$ if and only if $\Lambda_{p, \varepsilon} \ge \log\left(\tfrac{2S}{p}\right) + \log(N_S)$. Indeed, we have 
\begin{equation*}
\begin{aligned}
\log\left(\tfrac{2S}{p}\right) + \log\left(N_s\right)
&\overset{\text{(iii)}}{\le} \log(\tfrac{2S}{p}) + \log(\sqrt{\tfrac{3}{2}}+ \tfrac{1}{4}) + \log(N_{s,0}) + \tfrac{1}{2} \log(1 + \Lambda_{p, \varepsilon} + \log(N_{s,0}))\overset{\text{(iv)}}{\le} \Lambda_{p, \varepsilon},
\end{aligned}
\end{equation*}
where in (iii) we use the fact that $N_s \le \sqrt{\tfrac{3}{2}}N_{s,0}\sqrt{1+\Lambda_{p, \varepsilon} + \log(N_{s,0})} + 1$ along with $1 \le \tfrac{L}{\mu} \le \tfrac{N_{s,0}}{4}$ and in (iv) we use the definition of $\Lambda_{p, \varepsilon}$. 

It remains to specify the complexity. We have $\Lambda_{p, \varepsilon} = \mathcal{\widetilde O}\left(\log\tfrac{L}{\mu p}\right)$ by definition, $q_s = \mathcal O\left(\log\tfrac{L}{\mu p}\right)$ by \eqref{eq:q_s-bound-strongly-monotone}, and $N_s = \mathcal O\left(\tfrac{L}{\mu}\left( \log \tfrac{L}{\mu p}\right)^{1/2}\right)$. Hence,
$$
\tsum_{s = 1}^S N_s m_s = \mathcal {\widetilde O}\left\{\tfrac{L}{\mu}\log\left(\tfrac{D_0^2}{\varepsilon}\right)\left(\log \tfrac{L}{\mu p}\right)^{1/2} + \tfrac{\varsigma_*}{\mu} \log\left( \tfrac{D_0^2}{\varepsilon}\right)\left(\log \tfrac{L}{\mu p}\right)^{5/2} + \tfrac{\sigma_*^2}{\mu^2 \varepsilon}\left(\log\tfrac {L}{\mu p}\right)^2\right\}.
$$
\end{proof}

\subsection{Accumulative regularization for uniform light-tailed noise}\label{sec:hp-uniform}
In this subsection, we establish the convergence of AR under light-tailed uniform noise. Since this setting is simpler than the state-dependent noise setting, we omit the proofs and state only the assumptions and results.

We assume access to i.i.d data $\{\zeta_k\}$ and unbiased stochastic operator evaluations \eqref{eq:fast-decreasing-bias-o}. 
In addition, we assume that there exists $\sigma > 0$ such that the following holds. For any $\mathcal F_k$-measurable point $x\in X$, \begin{align}\label{assumption:subgaussian}
        \mathbb{E}\!\left[\exp\left\{\tfrac{\left\|[\tilde{F}(x,\zeta_{k})- F(x)]\right\|^{2}}{\sigma^{2}}\right\}\ \bigg| \ \mathcal F_k\right]\leq \exp\{1\} \quad \textnormal{a.s.}, 
\end{align}

In order for AR to attain $\res_F(x_S) \leq \varepsilon$ with high probability, we require that the inner subroutine $\mathcal A$ obey a sublinear convergence rate with high probability within each epoch. Fix $\Lambda \geq 0$, and let $E_{s,k}^1$ and $E_{s,k}^2$ denote the events defined in \eqref{eqn:hp-distance-s-state-dependent} and \eqref{eqn:hp-residual-s-state-dependent}, respectively, with $q_s \equiv 1$.

\begin{assumption}\label{as:hp-subroutine-guarantee}
For any mini-batch size $m_s \geq 1$, the iterates generated by the subroutine $\mathcal A(\mathcal{SO},r_s,\bar x_s,x_{s-1},m_s)$ for solving the regularized problem \eqref{eq:sub-problem} satisfy
\begin{equation}
\mathbb P\left(E_{s,k}^1\cap E_{s,k}^2\,\middle|\,\mathcal F_{s,0}\right)\geq 1-2k\exp\{-\Lambda\},\quad\textnormal{a.s.}
\end{equation}
\end{assumption}
High-probability performance guarantees of the form in \autoref{as:hp-subroutine-guarantee} have not been established for SOE in the literature. However, a simplified version of the state-dependent analysis in \autoref{sec:hp-soe-state-dependent-new}, specialized to $q_s=1$, shows that SOE satisfies this assumption.

 We provide the convergence result for AR below. 

\begin{theorem}\label{thm:hp-uniform}
Let $p \in (0,1)$ and assume that $0 <\varepsilon \le LD_0$, where $D_0\geq \textnormal{dist}(x_0, X^*)$. In \autoref{alg:ARVI}, suppose $S$ and $ r_s $ satisfy \eqref{eqn:regularization-exp-unbiased} and $N_s$ satisfies \eqref{eqn:batch-epoch}. Moreover, the batch size $m_s$ is defined as
\begin{equation}\label{eqn:hp-batch-and-confidence}
m_s=\left\lceil\tfrac{16^{s-1}N_s(1+\Lambda_{p, \varepsilon}) \sigma^2}{(L+ r_s )^2 D_0^2}+1\right\rceil, \qquad \Lambda_{p, \varepsilon} :=  \log\left(\tfrac{16\sqrt{c_{\mathcal A}}}{p}\left[S + \tfrac{16(16 c_{\mathcal A}+4) L D_0\log S}{9\varepsilon}\right]\right).
\end{equation}
Then, at epoch $S$, \autoref{alg:ARVI} computes an approximate solution $x_S$ such that $\mathbb P(\res_F(x_S) \leq \varepsilon) \geq 1 - p$ after
\begin{align}
   \widetilde{\mathcal{O}}\left(\tfrac{LD_0}{\varepsilon}+\tfrac{ \sigma^2}{\varepsilon^2}\left(\log\tfrac{LD_0}{\varepsilon}\right)^3\left(\log \tfrac{LD_0}{\varepsilon p}\right)\right)
\end{align}
calls to the $\mathcal{SO}$.
\end{theorem}
Notice that the choice of parameters for \autoref{thm:hp-uniform} does not require knowledge for $\mu$ in its implementation. However, when $\mu > 0$ is known, a simpler choice of parameters reduces AR to a direct restarting scheme with optimal convergence rates in terms of the distance to the optimal solution.
\begin{corollary}\label{thm:hp-main-unbiased-mu}
Let $p \in (0,1)$. Assume that $F$ is $L$-Lipschitz continuous and $\mu$-strongly
monotone, and
$0<\varepsilon \le \|x_0-x^*\|^2 \leq D_0^2$. In \autoref{alg:ARVI}, suppose $S$ and $ r_s $ satisfy \eqref{eqn:regularization-exp-unbiased-mu}, $N_s$ satisfies \eqref{eqn:batch-epoch-mu}, and $m_s$ is defined as
\begin{align}\label{eq:hp-strongly-monotone-ms-lambda}
m_s=\max\left\{\left\lceil\tfrac{32\cdot16^{s-1}c_{\mathcal{A}}(1+\Lambda_{p, \varepsilon}) \sigma^2}{N_s\mu^2 D_0^2}\right\rceil,\, 1\right\}, \quad \Lambda_{p, \varepsilon} := \log\left( \tfrac{2}{p}\left\lceil\tfrac{4\sqrt{2c_{\mathcal{A}}}L}{{\mu}} + 1\right\rceil\left\lceil
\log_{16}\!\left(\tfrac{D_0^2}{{\varepsilon}}\right)\right\rceil\right)
\end{align}
Then, at epoch $S$, \autoref{alg:ARVI} computes an approximate solution $x_S$ s.t. $\mathbb P\{\|x_S - x^*\|^2
\le\varepsilon\} \geq 1 - p$ after
\begin{align}
\widetilde{\mathcal O}\left(\tfrac{L}{\mu}\log\left(\tfrac{D_0^2}{\varepsilon}\right) + \tfrac{ \sigma^2}{\mu^2\varepsilon}\log\left(\tfrac{L}{\mu p}\right)\right)
\end{align}
calls to the $\mathcal{SO}$.
\end{corollary}

\section{Application: Online Policy Evaluation via Variational Inequalities}\label{sec:Applications: Online Policy Evaluation via Variational Inequalities}
In this section, we consider policy evaluation, a key step in reinforcement learning, which computes the value function $V^\pi$ associated with a given policy $\pi$ in a Markov decision process (MDP). We show how the methods developed in \autoref{sec:Stochastic Variational Inequalities under State-Dependent Noise} can be applied to this problem. In \autoref{sec:Projected Bellman Equation}, we recall the reformulation of the projected Bellman equation from \cite{kotsalis2022simple2,li2023accelerated}; since some properties follow there, we omit the related proofs. In \autoref{sec:PolicyEvaluationEstimators}, we introduce stochastic estimators satisfying the assumptions \eqref{eq:fast-decreasing-bias-o} and \eqref{eqn:state-dependent-assumption}. Finally, in \autoref{sec:FTD}, we present accumulatively regularized fast temporal difference learning (AR-FTD), which applies AR with SOE as a subroutine to both strongly monotone ($\gamma$ small) and  monotone ($\gamma$ large) cases.

\subsection{Projected Bellman equation}\label{sec:Projected Bellman Equation}
We consider a discounted Markov decision process (MDP) denoted by
$\mathcal{M}=(\mathcal{S},\mathcal{A},\mathbb{P},c,\gamma)$, where $\mathcal{S}$ is the state space, $\mathcal{A}$ is the action space, $\mathbb{P}(\cdot\mid s,a)$ is the transition kernel, $c(s,a)$ is the one-stage cost, and $\gamma\in(0,1)$ is the discount factor. A policy $\pi$ specifies how actions are selected; in particular, $\pi(a\mid s)$ denotes the probability of choosing action $a\in\mathcal{A}$ at state $s\in\mathcal{S}$. Under a fixed policy $\pi$, the trajectory evolves according to $a_t\sim \pi(\cdot\mid s_t)$ and $s_{t+1}\sim \mathbb{P}(\cdot\mid s_t,a_t)$. The value function $V^\pi$ associated with policy $\pi$ is defined as
\begin{align}
V^\pi(s) &\coloneqq
\bbe_\pi
\left[
\tsum_{t=0}^\infty \gamma^t c(s_t, a_t)  \mid
s_0 = s
\right],
\label{eq:def_V_function0}
\end{align}
where $\bbe_\pi$ denotes expectation with respect to the trajectory generated by following the policy $\pi$.

Recall that once a policy is fixed, the sequence of states $ \{s_0, s_1, \hdots, \} $ becomes a time-homogeneous Markov
chain with transition probability matrix $\mathbb{P}_\pi$,
where  the $(s,s')$-th entry of $\mathbb{P}_\pi$ is
\beq \label{eq:def_cP_under}
\mathbb{P}_\pi(s, s') \coloneqq \tsum_{a \in \cA} \pi(a|s) \mathbb{P}(s'|s,a), \quad \forall\, (s, s') \in \cS \times \cS.
\eeq
We denote by
$c_\pi(s) \coloneqq \tsum_{a \in \cA} \pi(a|s)c(s,a)$, for all $s \in \cS$,
the expected instantaneous cost associated with state $s$ under policy $\pi$.
By the Bellman equation e.g.,\cite{puterman1994markov,bertsekas2012dynamic}, $V^\pi$ must satisfy the linear system
\begin{align}
V^\pi &= c_\pi + \gamma \mathbb{P}_\pi V^\pi, \qquad V^\pi \in \bbr^{|\cS|}. \label{eq:Bellman_V_eval}
\end{align}

For large state and action spaces, \eqref{eq:Bellman_V_eval} is often approximated using linear function approximation. In particular, we choose $d \ll |\cS|$ linearly independent basis vectors and define
$\Phi \coloneqq [\phi_1,\ldots,\phi_d]$ and $\phi(s)\coloneqq[\phi_1(s),\ldots,\phi_d(s)]^\top$. Throughout this section, we assume
\begin{align}
|c(s, a)| &\le \bar c, \quad\forall (s, a) \in \cS \times \cA, \quad
\|\phi(s)\|_2 \le \bar \phi, \quad\forall s \in \cS. \label{eq:bound_psi}
\end{align}

Let $\brho$ be a positive definite diagonal matrix, with inner product $\langle x,y\rangle_{\brho}\coloneqq x^\top\brho y$ and norm $\|x\|_{\brho}\coloneqq\sqrt{x^\top\brho x}$. Let $S\coloneqq\operatorname{span}\{\phi_1,\ldots,\phi_d\}$, and let $\Pi_S(x)$ denote the projection of $x$ onto $S$ under the $\|\cdot\|_{\brho}$-norm, given by
\beq \label{eq:def_projection_over_subspace}
\Pi_S(x) \coloneqq \argmin_{y \in S} \|y - x\|_\brho^2 =
\Phi [\Phi^\top \brho \Phi]^{-1} \Phi^\top \brho x.
\eeq

The choice of \( \brho \) is tied both to desirable properties of the operator \( \mathbb{P}_\pi \)
under the \( \|\cdot\|_\brho \) norm and to the underlying sampling method from the state space $\cS$.
In this paper, we mainly focus on the classic choice  $\label{eq:def_diag_steady}
\brho \coloneqq \text{diag}(\{\nu_\pi(1), \ldots, \nu_\pi(|\cS|)\}),$
where $\nu_\pi$ denotes the stationary distribution induced by policy $\pi$.
We further assume that the Markov chain associated with $\mathbb{P}_\pi$ is ergodic, and hence $\brho_s \equiv \nu_\pi(s) > 0, $ for all $ s \in \cS.$

To compute the value function \( V^\pi \) in \eqref{eq:Bellman_V_eval},  we define \( \vbar \) as the solution to the \emph{projected Bellman equation}:
\begin{align}\label{proj_fixed_point}
\vbar = \Pi_S(\gamma \mathbb{P}_\pi \vbar + c_\pi) = \Phi [\Phi^\top \brho \Phi]^{-1} \Phi^\top \brho (\gamma \mathbb{P}_\pi \vbar + c_\pi),
\end{align}
where the second equality follows from \eqref{eq:def_projection_over_subspace}.
To show that the solution $\vbar$ to \eqref{proj_fixed_point} exists,
we first demonstrate the nonexpansiveness of the operator \( \mathbb{P}_\pi \)
under the \( \|\cdot \|_\brho \) norm.

\begin{lemma}\cite{kotsalis2022simple2}
\label{lem:transition-pi}
		For each vector $x \in \bbr^{|\cS|}$, we have $	\|\mathbb{P}_\pi x\|_\brho \leq \|x\|_\brho.$
        Furthermore, the projected Bellman equation \eqref{proj_fixed_point} has a unique solution $\vbar$ and $\|\vbar\|_\brho \leq \bar c (1-\gamma)^{-1}$.
	\end{lemma}

Observe that the projected Bellman equation in \eqref{proj_fixed_point} involves computing $[\Phi^\top \brho \Phi]^{-1}$. Moreover, the fixed-point equation \eqref{proj_fixed_point} is defined over the state space and hence involves $|\cS|$ equations, which can be large. Therefore, solving this system directly can be computationally expensive. To obtain a lower-dimensional equivalent formulation, we multiply both sides of \eqref{proj_fixed_point} by $\Phi^\top \brho$, yielding 
\beq \label{proj_fixed_point1}
\Phi^\top \brho \vbar = \Phi^\top \brho (\gamma \mathbb{P}_\pi \vbar + c_\pi).
\eeq
In the sequel, for any \( v^\Diamond \in S \), let us use \( \theta^\Diamond \) to denote its corresponding parameterization in \( \bbr^d \),
such that, for example, \( \Phi \theta' = v' \) or \( \Phi \bar{\theta} = \bar{v} \).
With this shorthand, equation~\eqref{proj_fixed_point1} can be equivalently rewritten as
\begin{align}\label{proj_fix_point_2}
\Phi^\top \brho \Phi \bar{\theta} = \gamma \Phi^\top \brho \mathbb{P}_\pi \Phi \bar{\theta} + \Phi^\top \brho c_\pi.
\end{align}
While any solution \( \vbar \) of \eqref{proj_fixed_point} must also satisfy \eqref{proj_fixed_point1}, and
hence admits a corresponding parameter vector \( \bar{\theta} \) solving \eqref{proj_fix_point_2}, the converse
is not immediately obvious.
It can be shown  that \eqref{proj_fixed_point1}-\eqref{proj_fix_point_2} are, in fact, equivalent to \eqref{proj_fixed_point}.
Clearly, by \eqref{proj_fix_point_2}, \( \bar{\theta} \) is a solution to the following linear system:
\beq \label{eq:def_Bellman_fixed_point}
F(\theta) \coloneqq \Phi^\top \brho (\Phi \theta - c_\pi - \gamma \mathbb{P}_\pi \Phi \theta) = 0.
\eeq
The approach is to demonstrate that the linear operator \( F \), defined in \eqref{eq:def_Bellman_fixed_point},
is strongly monotone. As a result, \eqref{eq:def_Bellman_fixed_point} will have a unique solution \( \bar{\theta} \),
thereby establishing the equivalence between \eqref{proj_fixed_point1}–\eqref{proj_fix_point_2} and \eqref{proj_fixed_point}.

To characterize the strong monotonicity and other properties of \( F \),
it is often desirable to work with an orthonormal basis that spans the subspace \( S = \text{span}\{\phi_1, \ldots, \phi_d\} \).
To construct such a basis, define the matrix \( B \in \mathbb{R}^{d \times d} \) as  $B = \Phi^\top \brho \Phi,$
which can be viewed as the covariance matrix of features.
Since \( B \) is positive definite, we can define  $\bPhi \coloneqq [\bphi_1, \bphi_2, \ldots, \bphi_d] = \Phi B^{-\tfrac{1}{2}}.$
By construction, the matrix \( \bPhi \) satisfies $\bPhi^\top \brho \bPhi = B^{-\tfrac{1}{2}} \Phi^\top \brho \Phi B^{-\tfrac{1}{2}} = I, $
 which implies that
the vectors \( \bphi_1, \ldots, \bphi_d \) form an orthonormal set with respect to the \( \langle \cdot, \cdot \rangle_\brho \)
inner product, i.e., \( \langle \bphi_i, \bphi_j \rangle_\brho = \mathbb{I}(i = j) \).
With this orthonormal basis, it can be readily verified
that the projection operator in \eqref{eq:def_projection_over_subspace} simplifies to $\Pi_S(x) = \Phi B^{-1} \Phi^\top \brho x
= [ \bPhi B^{\tfrac{1}{2}} ] B^{-1} [B^{\tfrac{1}{2}}  \bPhi^\top] \brho x
= \bPhi \bPhi^\top \brho x.  $
Let us define the scalars  $\bar \beta \coloneqq \lambda_{\max}(B) $ $ \text{and}$ $\underline{\beta} \coloneqq \lambda_{\min}(B),  $
where \( \bar \beta / \underline{\beta} \) represents the condition number of \( B \).
We can readily verify that
\beq \label{eq:Gen_Bellman_eig}
\underline{\beta} \|\theta\|_2^2 \leq \|\theta\|_B^2 = \|\theta\|_{\Phi^\top \brho \Phi}^2 = \|\Phi \theta\|_\brho^2
= \|v\|_\brho^2 \leq \bar \beta \|\theta\|_2^2
\eeq
 for $v = \Phi \theta$.
 For any $\theta\in \bbr^d$, we have $\|\theta\|_2 = \|\bar \Phi \theta\|_\brho, \|\Phi \theta\|_\brho = \|B^{\tfrac{1}{2}}\theta\|_2.$
These identities are useful for characterizing the smoothness of $F(\theta)$.
Notice that $\bar \beta \leq \bar \phi^2$ as follows,
for all $\theta\in\mathbb{R}^d,$ we have
\begin{equation}\label{eqn:function-approx-eigenvalue-relation}
\begin{aligned}
\theta^\top \Phi^\top \brho \Phi \theta
&=\tsum_{s\in\mathcal S}\nu_\pi(s)
    \langle \phi(s),\theta\rangle^2\leq \tsum_{s\in\mathcal S}\nu_\pi(s)
    \|\phi(s)\|_2^2\|\theta\|_2^2\leq \bar\phi^2\|\theta\|_2^2,
\end{aligned}
\end{equation}
hence $\bar \beta = \lambda_{\max}(B)=\lambda_{\max}( \Phi^\top \brho \Phi) \leq \bar\phi^2.$

With those setup, we
recall $F$ is strongly monotone and smooth as follows. Since strong monotonicity implies uniqueness, $\bar\theta$ is the unique solution of \eqref{eq:def_Bellman_fixed_point}.
	\begin{lemma}\label{lem:linear-operator}
	For any $\theta, \theta' \in \bbr^d$, we have
    \begin{equation}\label{eq:strong_monotone_Bellman_operator}
        \begin{aligned}
        \langle F(\theta) - F(\theta'), \theta - \theta' \rangle&
	\ge (1-\gamma) \|v - v'\|_\brho^2 \ge (1-\gamma) \underline \beta \|\theta- \theta'\|_2^2,\\
    \|F(\theta) - F(\theta')\|_2& \le (1 +\gamma) {\bar \beta^{\tfrac{1}{2}}} \|v - v'\|_\brho
	\le (1+ \gamma) \bar \beta \|\theta - \theta'\|_2\\
\|v-v'\|_\brho^2
&\leq
\tfrac{1}{(1-\gamma)^2\underline\beta}
\|F(\theta)-F(\theta')\|_2^2.
        \end{aligned}
    \end{equation}
	\end{lemma}
\begin{proof}
The first two properties follow from \cite{kotsalis2022simple2}.
  We just prove the third one as follows.
\begin{align*}
(1-\gamma)\|v-v'\|_\brho^2
&\overset{\text{(i)}}{\leq}
\langle F(\theta)-F(\theta'),\theta-\theta'\rangle\leq
\|F(\theta)-F(\theta')\|_2\|\theta-\theta'\|_2\overset{\text{(ii)}}{\leq}
\tfrac{1}{\sqrt{\underline\beta}}
\|F(\theta)-F(\theta')\|_2\|v-v'\|_\brho,
\end{align*}
where (i) (resp. (ii)) follows from the first (resp. second) equation from \eqref{eq:strong_monotone_Bellman_operator}.
Dividing by $\|v-v'\|_\brho$ and squaring proves the claim.
\end{proof}
It is natural to solve \eqref{eq:def_Bellman_fixed_point}, or equivalently $F(\theta)=0$, through a root-finding or VI approach \cite{kotsalis2022simple2}. In particular, the problem can be written as solving the following VI, 
\begin{equation}\label{eqn:PE-VI}
    \langle F(\theta), \theta' - \theta \rangle \geq 0,\quad\text{for all}\quad\theta' \in \bbr^d.
\end{equation}

\subsection{Stochastic estimator} \label{sec:PolicyEvaluationEstimators}

Our goal in this subsection is to present a stochastic estimator of $F$ and to analyze its statistical properties.
Specifically, for a given stochastic estimator $\tilde{F}(\theta,\zeta_t)$ of $F(\theta)$, we are interested in verifying that the unbiased estimates assumption \eqref{eq:fast-decreasing-bias-o} and state-dependent noise assumption \eqref{eqn:state-dependent-assumption}  hold.

	To simplify the design of stochastic estimation of policy evaluation
operator $F$, we assume the access to a simulator (generative model)
that generates samples from the transition kernel.  In particular, we observe i.i.d. triples $\zeta_t=(s_t, a_t, s_t')$ such that  $s_t \sim \nu_\pi,~a_t \sim \pi(\cdot|s_t),~ s_t'\sim  \mathbb{P}(\cdot|s_t, a_t),$
	where $\nu_\pi$ is the stationary distribution of the Markov chain
	associated with $ \mathbb{P}_\pi$.
 However, the stochastic estimator is not restricted to i.i.d. sampling and can incorporate other conditional sampling methods.

Observe that operator $F$ in \eqref{eq:def_Bellman_fixed_point}
admits unbiased stochastic estimators under a generative model \cite{kotsalis2022simple2}.
\begin{align*}
F(\theta)
&= \tsum_{s \in \cS} \nu_\pi(s) \phi(s) \left[
\phi(s)^\top \theta - c_\pi(s) - \gamma \mathbb{P}_\pi(s, \cdot)^\top \Phi \theta
\right] \\
&= \tsum_{s \in \cS} \nu_\pi(s) \phi(s) \left\{
\phi(s)^\top \theta - \tsum_{a \in \cA} \pi(a \mid s) \left[c(s, a) +
\gamma \mathbb{P}(\cdot \mid s, a)^\top \Phi \theta\right]
\right\}\\
&= \tsum_{s \in \cS} \nu_\pi(s) \phi(s) \left\{
\phi(s)^\top \theta - \tsum_{a \in \cA} \pi(a \mid s) \left[c(s, a) +
\gamma \tsum_{s' \in \cS} \mathbb{P}(s' \mid s, a) \phi(s')^\top\theta \right]
\right\},
\end{align*}
where \( \phi(s) \) denotes the feature vector \( [\phi_1(s), \phi_2(s), \ldots, \phi_d(s)]^\top \).
Then, an unbiased estimator of \( F(\theta) \) reads
\beq \label{eq:stochastic_estimator_basic_F}
\tilde{F}(\theta,\zeta) = \phi(s) \left[ \phi(s)^\top \theta - c(s, a)
- \gamma \phi(s')^\top \theta \right].
\eeq
It remains to bound the variance.
The first inequality in the following lemma  bounds the noise at the solution $\bar\theta$ by $\bar\sigma^2$, see \cite{kotsalis2022simple2} for its proof.
The second shows that,
when $s_t \sim \nu_\pi$,
the variance of the difference between two stochastic operators can be controlled
by the inner product $\langle F(\theta) - F(\theta'), \theta - \theta'\rangle$.

\begin{lemma}  \label{lemma_operator_variance_2}
Let $\bar \theta$ be the solution to $F(\theta) = 0$.
Then
\begin{align}
&\bbe[\|\tilde F( \bar \theta,\zeta_t)
- F(\bar \theta) \|_2^2] \le \bar \sigma^2\coloneqq
\tfrac{8  \bar c^2 \bar \phi^2}{(1-\gamma)^2} + 2\bar c^2 \bar \phi^2,
\label{eq:def_variance_at_root}
\end{align}
where $\bar c$ and $\bar \phi$ are defined in \eqref{eq:bound_psi},
respectively. Furthermore, for any $\theta$ and $\theta'$, we have
\begin{align} \label{eqn:stochastic-estimator-state-dependent}
&\bbe
[\|\tilde F( \theta,\zeta_t) - \tilde F( \theta',\zeta_t)
- (F(\theta) - F(\theta'))\|_2^2] \le 2 \bar \phi^2 \langle F(\theta) - F(\theta'), \theta - \theta'\rangle.
\end{align}
\end{lemma}

\begin{proof}
Since $\bbe[\tilde F(\theta, \zeta_t)] =F(\theta)$,
it follows immediately that
\begin{align*}
&\bbe [
\|\tilde F( \theta,\zeta_t) - \tilde F( \theta',\zeta_t)
- (F(\theta) - F(\theta'))\|_2^2] \le \bbe[\|\tilde F( \theta,\zeta_t) - \tilde F( \theta',\zeta_t)\|_2^2].
\end{align*}
Next, observe that
\begin{align*}
\|\tilde F( \theta,\zeta_t) - \tilde F( \theta',\zeta_t)\|_2^2
&= \| \phi(s_t) (\langle \phi(s_t) - \gamma \phi(s_{t}'),
\theta - \theta'\rangle )\|_2^2 \\
&= (\theta - \theta')^\top [\phi(s_t) - \gamma \phi(s_{t}')] \phi(s_t)^\top
 \phi(s_t) [\phi(s_t)^\top - \gamma \phi(s_{t}')^\top] (\theta - \theta')\\
 &\overset{\text{(i)}}{\le} \bar \phi^2  (\theta - \theta')^\top [\phi(s_t) - \gamma \phi(s_{t}')]
  [\phi(s_t)^\top - \gamma \phi(s_{t}')^\top] (\theta - \theta')\\
  &= \bar \phi^2 (\theta - \theta')^\top \phi(s_t)  [\phi(s_t)^\top - \gamma \phi(s_{t}')^\top] (\theta - \theta') - \gamma  \bar \phi^2 (\theta - \theta')^\top \phi(s_{t}')  \phi(s_t)^\top (\theta - \theta')\\
  &\quad + \gamma^2 \bar \phi^2  (\theta - \theta')^\top  \phi(s_{t}') \phi(s_{t}')^\top (\theta - \theta')\\
  &\overset{\text{(ii)}}{\le} \bar \phi^2 (\theta - \theta')^\top \phi(s_t)  [\phi(s_t)^\top - \gamma \phi(s_{t}')^\top] (\theta - \theta')- \gamma  \bar \phi^2 (\theta - \theta')^\top \phi(s_{t}')  \phi(s_t)^\top (\theta - \theta')\\
  &\quad + \bar \phi^2  (\theta - \theta')^\top  \phi(s_{t}') \phi(s_{t}')^\top (\theta - \theta')\\
&= \bar \phi^2 (\theta - \theta')^\top  [\phi(s_t)  \phi(s_t)^\top - \gamma \phi(s_t)  \phi(s_{t}')^\top+\phi(s_{t}') \phi(s_{t}')^\top
- \gamma \phi(s_t') \phi(s_{t})^\top] (\theta - \theta'),
\end{align*}
where (i) follows from the bound \eqref{eq:bound_psi}, and (ii) uses the fact that $\gamma \in [0,1]$.
Taking expectation on both sides and using the stationarity of $\nu_\pi$, we obtain
\[
\bbe [\phi(s_t)  \phi(s_t)^\top] = \bbe [\phi(s_{t}')  \phi(s_{t}')^\top] = \Phi^\top \brho \Phi,\quad \bbe [\phi(s_t) \phi(s_{t}')^\top]
= \Phi^\top \brho \mathbb{P}_\pi \Phi, \quad \bbe [\phi(s_t') \phi(s_{t})^\top]
= {(\Phi^\top \brho \mathbb{P}_\pi \Phi)^{\top}}.\]
Substituting these identities yields
\begin{align*}
\bbe
[\|\tilde F( \theta,\zeta_t) - \tilde F( \theta',\zeta_t)
- (F(\theta) - F(\theta'))\|_2^2]
&\leq \bar \phi^2 (\theta - \theta')^\top  [\Phi^\top \brho \Phi - \gamma \Phi^\top \brho \mathbb{P}_\pi\Phi+\Phi^\top \brho \Phi
- \gamma {(\Phi^\top \brho \mathbb{P}_\pi \Phi)^{\top}}] (\theta - \theta')\\
&= \bar \phi^2 (\theta - \theta')^\top \Phi^\top \brho [ \Phi - \gamma \mathbb{P}_\pi\Phi+\Phi
- \gamma {\mathbb{P}_\pi \Phi}] (\theta - \theta')\\
&=2 \bar \phi^2 \langle F(\theta) - F(\theta'), \theta - \theta'\rangle,
\end{align*}
which completes the proof.
\end{proof}
\subsection{Policy evaluation}\label{sec:FTD}
In policy evaluation, TD-type methods can be interpreted as stochastic approximation methods for solving the operator equation $F(\theta)=0.$ In particular, classical TD learning can be viewed as a stochastic forward operator method that uses a single-sample stochastic estimate of the policy evaluation operator $F$ at each iteration.

In this section, we develop accelerated algorithms for solving \eqref{eqn:PE-VI} by applying the AR method (cf. \autoref{alg:ARVI}) within this stochastic approximation framework. By \autoref{lem:linear-operator}, the linear operator $F$ is strongly monotone with parameter $(1-\gamma)\underline{\beta}$. Therefore, when $\gamma$ is bounded away from $1$, we can directly apply the AR method for strongly monotone VIs (cf. \autoref{state-dependent-thm-strongly-monotone}). In the long-horizon regime, however, the strong monotonicity parameter degenerates as $\gamma$ approaches $1$, leading to an unfavorable dependence on $1/(1-\gamma)$ in the strongly monotone analysis. We instead apply the AR method for monotone VIs (cf. \autoref{thm:ARVI-state-dependent}). The resulting AR-FTD method combines regularization, restarting, and the FTD subroutine to achieve fast decay of the Bellman residual.

We adopt the notation introduced in \autoref{sec:Stochastic Variational Inequalities under State-Dependent Noise}. At iteration $k$ of epoch $s$, we construct a stochastic estimator of $F$ from $m_s$ i.i.d.\ samples as in \eqref{eqn:minibatch}: $\tilde{F}_{s,k}(\theta_{s,k})\coloneqq\tfrac{1}{m_s}\tsum_{i=1}^{m_s}\tilde{F}(\theta_{s,k},\zeta_{s,i}).$ By \autoref{lemma_operator_variance_2}, $\tilde{F}_{s,k}$ satisfies the unbiasedness and state-dependent variance assumptions (cf. \eqref{eq:fast-decreasing-bias-o} and \eqref{eqn:state-dependent-assumption}). More precisely, for any $\theta\in\bbr^d$, setting $\varsigma_*\coloneqq 4\bar \phi^2$ and $\sigma_*^2=2\bar \sigma^2$, we have
\begin{align*}
\bbe[\|\tilde{F}_{s,k}(\theta)-F(\theta)\|_*^2\mid\mathcal{F}_{s,k}] &\le \tfrac{\varsigma_*}{m_s}\langle F(\theta)-F(\bar{\theta}),\theta-\bar{\theta}\rangle+\tfrac{\sigma_*^2}{m_s}.
\end{align*}

We next present convergence guarantees for AR-FTD in \autoref{the:mult_epoch_FTD_burn} and \autoref{prop:ARVI-state-dependent}. 
For each $s\in\{1,\ldots,S\}$, suppose that $\lambda_{s,k}\equiv\lambda_k$, $\eta_{s,k}\equiv\eta_k$, and $k_{s,0}\equiv k_0$, where $\lambda_k$, $\eta_k$, and $k_0$ are chosen as in \autoref{cor:subroutine-guarantee}, with $L$, $\mu$, $\varsigma_*$, and $\sigma_*^2$ are defined as
\begin{align}\label{eqn:final-bound}
L &=(1+\gamma)\bar \beta,\quad \mu=(1-\gamma)\underline\beta,\quad \varsigma_*=4\bar \phi^2,\quad\text{and}\quad \sigma_*^2=2\bar \sigma^2.
\end{align}

When $\gamma$ is sufficiently small, we apply AR-FTD with $\mu>0$; see \autoref{state-dependent-thm-strongly-monotone}. We measure its performance by the mean-square error $\bbe[\|\hat v_S-\bar v\|_\brho^2]$, where $\bar v=\Phi\bar\theta$ denotes the exact solution of the linear system in \eqref{eq:def_Bellman_fixed_point}, and $\hat v_S=\Phi\hat\theta_S\coloneqq\Phi\theta_{S,R_S}$ is the randomized output of the $S$-th epoch. Specifically, $R_S$ is sampled as \eqref{eqn:iterative-averaging-per-epoch}, i.e.,
\begin{equation*}
\mathbb{P}(R_S=k)
=
\tfrac{k+k_{S,0}-1}
{\sum_{j=1}^{N_S}(j+k_{S,0}-1)}
=
\tfrac{2(k+k_{S,0}-1)}
{N_S(N_S+2k_{S,0}-1)},
\end{equation*}
where $N_S$ is the epoch length and $k_{S,0}=\tfrac{16L}{\mu}+1=\tfrac{16(1+\gamma)\bar\beta}{(1-\gamma)\underline \beta}+1.$ \begin{proposition}\label{the:mult_epoch_FTD_burn}
 Suppose $D_0>0$ and
$0<\varepsilon\leq\underline\beta D_0^2/16$. Furthermore,
suppose the number of epochs is given by
\begin{equation}\label{second-output-epoch}
S=\left\lceil\log_{16}\left(\tfrac{\underline\beta(\sqrt{c_{\mathcal A}}+4)D_0^2}{16\sqrt{c_{\mathcal A}}\varepsilon}\right)\right\rceil,
\end{equation}
where $D_0\geq\|\theta^0-\bar\theta\|$. Furthermore, suppose the parameters $r_s$, $N_s$, and $m_s$ of AR (\autoref{alg:ARVI}) are chosen as in \autoref{state-dependent-thm-strongly-monotone}. Then, AR-FTD returns $\hat v_S$ satisfying $\bbe[\|\hat v_S-\bar v\|_\brho^2]\leq\varepsilon$ with a total sample complexity of
\beq\label{eqn:sample-FTD}
\mathcal{O}\left(\tfrac{\bar\beta+\bar\phi^2}{(1-\gamma)\underline\beta}\log_{16}\!\left(\tfrac{\underline\beta D_0^2}{\varepsilon}\right)+\tfrac{\bar\sigma^2}{\varepsilon(1-\gamma)^2\underline\beta}\right).
\eeq
\end{proposition}
\ifdefined\myfinal
\begin{proof}
By \autoref{prop:subroutine}, for each $1\leq s\leq S,$
AR-FTD with $r_s=0$
outputs solutions $(\theta_s, \hat{\theta}_s) $ such that
\begin{equation*}
    \begin{aligned}
    \tfrac{16N_s}{\mu(N_s+k_{s,0})}
\mathbb{E}[
\langle F(\hat{\theta}_s)-F(\bar{\theta}), \hat{\theta}_s-\bar{\theta}\rangle
]+    \mathbb E\left[\|\theta_s-\bar{\theta}\|^2\right]
&\leq
\tfrac{c_{\mathcal A}L^2}
{\mu^2N_s^2}
\mathbb E\left[\|\theta_{s-1}-\bar{\theta}\|^2\right]
+
\tfrac{c_{\mathcal A}\mathbb E[ \sigma_{s,*}^2]}
{N_sm_s\mu^2}\leq \tfrac{D_0^2}{16^s},
     \end{aligned}
\end{equation*}
where $c_\mathcal{A}=1301$ and $\hat{\theta}_s\coloneqq {\theta}_{s, R_s},$
where $R_s$ be sampled according to \eqref{eqn:iterative-averaging-per-epoch}.

Since $r_s=0$, we have
$F_s=F$, and
$\sigma_{s,*}^2=\sigma_*^2=2\bar\sigma^2$.
Hence, by \autoref{state-dependent-thm-strongly-monotone}, we have 
\begin{equation*}
\begin{aligned}
  &\tfrac{16N_s}{\mu(N_s+k_{s,0})}
\mathbb{E}[
\langle F(\hat{\theta}_s)-F(\bar{\theta}), \hat{\theta}_s-\bar{\theta}\rangle
]\overset{\text{(i)}}{\leq} \tfrac{D_0^2}{16^s},
    \end{aligned}
    \end{equation*}
    where in (i), we substituted the choices for the batch size $m_s$ and the epoch length $N_s, k_{s,0}.$
    Furthermore, by the choices of $N_s,$ and $k_{s,0}=\tfrac{16L}{\mu}+1, $ we have
    \begin{equation*}
        \begin{aligned}
              \tfrac{16N_s}{(1-\gamma)\underline\beta(N_s+k_{s,0})}
         \geq \tfrac{16\sqrt{c_{\mathcal{A}}}}{(1-\gamma)\underline\beta\left(\sqrt{c_{\mathcal{A}}}+4\right)}.
        \end{aligned}
    \end{equation*}
Furthermore, we have
\begin{equation*}
     \begin{aligned}
  \mathbb{E}\left[\|\hat{v}_S - \bar{v}\|_\brho^2\right]\overset{\text{(ii)}}{\leq} \tfrac{1}{1-\gamma} \mathbb{E}[
\langle F(\hat{\theta}_S)-F(\bar{\theta}), \hat{\theta}_S-\bar{\theta}\rangle
]&\leq \tfrac{\underline\beta\left(\sqrt{c_{\mathcal{A}}}+4\right)}{16\sqrt{c_{\mathcal{A}}}}\cdot\tfrac{D_0^2}{16^S},
    \end{aligned}
\end{equation*}
where in (ii), we used \autoref{lem:linear-operator}.
By choosing the number of epochs as in \eqref{second-output-epoch},
we conclude $\mathbb{E}\left[\|\hat{v}_S- \bar{v}\|_\brho^2\right]\leq \varepsilon.$ The total number of samples used by AR-FTD with $\mu>0$ is therefore
\begin{equation*}
    \begin{aligned}
       \tsum_{s=1}^{S}\tsum_{i=1}^{N_s} m_s
&\le \tsum_{s=1}^{S}\left[\tfrac{8\cdot16^{s-1}\sqrt{c_{\mathcal{A}}} \sigma^2_*}{\mu L \|\theta_0-\bar{\theta}\|_2^2}{\tfrac{1}{\sqrt{2}}}+\tfrac{c_{\mathcal A}\varsigma_*}{L}+1\right]\left[\tfrac{4L\sqrt{2c_{\mathcal{A}}}}{{\mu}}+1\right]
=\mathcal{O}\left(\tfrac{(1+\gamma)\bar{\beta}+\bar \phi^2}{(1-\gamma)\underline \beta}
\log_{16}\!\left(\tfrac{\underline \beta D_0^2}{{\varepsilon}}\right)+\tfrac{\bar{\sigma}^2}{\varepsilon(1-\gamma)^2\underline{\beta}}\right).
    \end{aligned}
\end{equation*}
\end{proof}
\medskip

The discount factor $\gamma$ determines the relative importance of future rewards. Values of $\gamma$ close to $1$ are widely used when long-term rewards are important, including in benchmark deep RL methods \cite{amit2020discount}. This setting is commonly referred to as long-horizon discounted policy evaluation and is central to RL.

However, the sample complexity bound in \eqref{eqn:sample-FTD} deteriorates as $\gamma$ approaches $1$ due to its dependence on inverse powers of $1-\gamma$. Thus, although AR-FTD is suitable when strong monotonicity is sufficiently large, its guarantee becomes less favorable for long-horizon policy evaluation. This motivates methods whose convergence guarantees remain stable as $\gamma$ approaches $1$.

When $\gamma$ is close to $1$, we apply AR-FTD with $\mu=0$; see \autoref{thm:ARVI-state-dependent}. The performance measure is the expected residual $\mathbb E[\res_F(x_S)]$. We have the following convergence guarantee for long-horizon policy evaluation.

\begin{proposition}\label{prop:ARVI-state-dependent}
Suppose the parameters $S$, $ r_s $, $N_s$, and $m_s$ of AR (\autoref{alg:ARVI}) are chosen as in \autoref{thm:ARVI-state-dependent}.
    Then AR-FTD computes an approximate solution $\theta_S$ such that $ \sqrt{\mathbb E[\|F(\theta_S)\|^2]}
\leq \varepsilon$
 after
\begin{align*}
  \mathcal{\widetilde{\mathcal{O}}}\left(\tfrac{\bar \beta D_0}{\varepsilon}+\tfrac{\bar{\phi}^2D_0}{\varepsilon}(\log\tfrac{\bar \beta D_0}{\varepsilon})^3+\tfrac{\bar\sigma^2}{\varepsilon^2}(\log\tfrac{\bar \beta D_0}{\varepsilon})^3\right)
\end{align*}
samples, where $D_0 \geq \textnormal{dist}(\theta^0, \bar{\Theta})$, and $\bar{\Theta}$ is the space of the optimal solutions.
\end{proposition}
\begin{proof}
By \autoref{thm:ARVI-state-dependent},
 we obtain $\theta_S$ such that $ \sqrt{\mathbb E[\res_F(\theta_S)^2]}
\leq \varepsilon$
 after
\begin{align*}
  \widetilde{\mathcal{O}}\left(\tfrac{LD_0}{\varepsilon}+
\tfrac{{\varsigma_*} D_0}{\varepsilon}
(\log\tfrac{LD_0}{\varepsilon})^3
+
\tfrac{\sigma_*^2}{\varepsilon^2}
(\log\tfrac{LD_0}{\varepsilon})^3
\right)
\end{align*}
samples.
In particular, for policy evaluation problem \eqref{eq:def_Bellman_fixed_point}, it holds that in \eqref{def_N_X} $X = \bbr^n$, then $N_X( \theta) = \{0\}$ and $\res_F( \theta) = \|F(\theta)\|$, which is exactly
 the residual of solving the linear equation $F( x) = 0$.
Substituting the bounds of $L,\sigma_*,\varsigma_*$ from \eqref{eqn:final-bound} concludes the proof.
\end{proof}

\smallskip

The above result shows that AR-FTD is well suited for long-horizon policy evaluation when the goal is to obtain a solution with small expected residual. 
The AR-FTD with small $\gamma$ guarantee applies in the strongly monotone regime and controls the mean-squared distance to the projected solution, while the AR-FTD with large $\gamma$ close to $1$ bound applies in the long-horizon, nearly monotone regime and controls the projected Bellman residual. Up to logarithmic factors, both guarantees match the lower bounds respectively: distance to the solution in strongly monotone problems and residual accuracy for monotone problems.

The two bounds exhibit different dependencies on the discount factor $\gamma$, although they are not directly comparable because they control different accuracy criteria. To achieve an estimate $\widehat v_S$ such that $\bbe[\|\widehat v_S-v^*\|_\brho^2]\leq\varepsilon$, the stochastic term in the FTD sample complexity scales as $\mathcal O\bigl(\tfrac{\bar\sigma^2}{\underline\beta(1-\gamma)^2\varepsilon}\bigr)$. By contrast, to achieve $\mathbb E[\|F(\theta_S)\|]\leq\varepsilon$, the stochastic term in the residual complexity of AR-FTD scales as $\widetilde{\mathcal O}\bigl(\tfrac{\bar\sigma^2}{\varepsilon^2}(\log\tfrac{\bar\beta D_0}{\varepsilon})^3\bigr)$. Thus, AR-FTD avoids the explicit $(1-\gamma)^{-2}$ factor while exhibiting a higher-order dependence on its residual tolerance. Although AR-FTD still depends on $\gamma$ through the optimal variance $\bar\sigma$, determining which bound is preferable requires translating the residual guarantee into a value-error guarantee.

Moreover, the residual guarantee in \autoref{prop:ARVI-state-dependent} can be converted into a value-function guarantee comparable to that in \autoref{the:mult_epoch_FTD_burn}. Indeed, by the third relation in \autoref{lem:linear-operator}, \eqref{eq:strong_monotone_Bellman_operator}, we have
\begin{equation*}
\bbe\left[\| v_S-v^*\|_\brho^2\right]
\leq
\tfrac{1}{(1-\gamma)^2\underline\beta}
\bbe\left[\res_F(\theta_S)^2\right].
\end{equation*}
Therefore, to achieve $\bbe[\|v_S-v^*\|_\brho^2]\leq\varepsilon$, it suffices to set the residual tolerance in \autoref{prop:ARVI-state-dependent} to $(1-\gamma)\sqrt{\underline\beta \varepsilon}$. Substituting this choice into the residual sample complexity shows that AR-FTD requires
\begin{equation*}
\widetilde{\mathcal O}\left(
\tfrac{(1+\gamma)\bar\beta D_0}{(1-\gamma)\sqrt{\underline\beta\varepsilon}}
+\tfrac{\bar\phi^2 D_0}{(1-\gamma)\sqrt{\underline\beta\varepsilon}}\left(\log\tfrac{(1+\gamma)\bar\beta D_0}{(1-\gamma)\sqrt{\underline\beta\varepsilon}}\right)^3
+\tfrac{\bar\sigma^2}{(1-\gamma)^2\underline\beta\varepsilon}\left(\log\tfrac{(1+\gamma)\bar\beta D_0}{(1-\gamma)\sqrt{\underline\beta\varepsilon}}\right)^3
\right).
\end{equation*}
samples. The dominant stochastic term matches that in the sample-complexity bound of \autoref{the:mult_epoch_FTD_burn} up to logarithmic factors, whereas the deterministic terms do not. Although the residual-to-value conversion and the resulting complexity bound depend on the strong-monotonicity parameter $(1-\gamma)\underline\beta$, the AR-FTD updates do not require its value.

\section{Concluding Remarks} \label{sec_conclusion}

In this paper, we present a unified AR framework for finding strong solutions to monotone VIs in both deterministic and stochastic settings. We show that, with appropriately chosen regularization parameters and moving centers, AR achieves optimal convergence rates. For the stochastic setting, we introduce a state-dependent noise model tailored to monotone VIs and develop algorithms with optimal convergence rates in both the monotone and strongly monotone regimes. Thus, AR provides a unified treatment of these two classes of VI problems. We further apply AR to policy evaluation and improve several existing convergence-rate results in both the long-horizon and short-horizon regimes.

\bibliographystyle{spmpsci}
\bibliography{references}

\begin{appendices}

\section{Stochastic operator extrapolation for strongly monotone VIs under state-dependent noise}\label{sec:hp-soe-state-dependent-new}
In this subsection, we show that for strongly monotone problems, SOE satisfies the performance guarantees in \autoref{as:hp-assumption_inner_sto-state-dependent} under the unbiased estimator assumption \eqref{eq:fast-decreasing-bias-o} and the high-probability state-dependent noise condition in \eqref{eq:hp-bounded-var-at-opt-o}. Throughout this section, we suppress the epoch index $s$ since our results hold for any fixed epoch. 
\smallskip

We begin with a concentration inequality regarding sums of sub-Gaussian vectors.

\begin{lemma}\label{lem:juditsky-bound}\cite[Theorem 4.1]{juditsky_large_2023}
Let $X_{1,k}, \cdots, X_{m,k} \in \mathbb R^d$ be a sequence of random vectors conditionally independent given $\mathcal F_k$ such that $\bbe[X_{i,k} \mid \mathcal F_k] = 0$ a.s. and $\mathbb E[\exp(\|X_{i,k}\|^2/ \sigma_k^2) \mid \mathcal F_k] \leq \exp(1)$ a.s. for all $i \in [m]$. It holds for any $t \geq 0$ that
$$
\mathbb P\left\{\|\tsum_{i = 1}^m X_{i,k}\| \geq (1+t)\sqrt{m}\sigma_k \mid \mathcal F_k\right\} \leq \exp(-t^2/3) \quad \text{a.s.}
$$
\end{lemma}
We begin with the following identity for the averaged operator error.

\begin{lemma}\label{prop:moment-tail-bound}
Under Assumptions \eqref{eq:fast-decreasing-bias-o} and \eqref{eq:hp-bounded-var-at-opt-o}, it holds that
\begin{equation}\label{eq:MGF-bound}
\mathbb E\left[\exp\left\{\tfrac{m\|\tilde F_{k}(x) - F(x)\|^2}{9[\sigma_*^2 + \varsigma_*\langle F(x) - F(x^*), x - x^*\rangle]}\right\} \,\bigg|\, \mathcal F_{k}\right] \leq \exp(1)
\end{equation}
\end{lemma}
\begin{proof}
Notice that the random vectors $\tilde F(x, \zeta_{k}^i) - F(x)$ satisfy the conditions of \autoref{lem:juditsky-bound} with $ \sigma_k^2 := \sigma_*^2 + \varsigma_* \langle F(x) - F(x^*), x - x^*\rangle$. By defining $W_{k} \coloneqq \tfrac{m_s\|\tilde F_k(x) - F(x)\|^2}{9\sigma_k^2 }$ we have
$$\mathbb P(W_{k} \geq w \mid \mathcal F_{k}) \leq \exp\left\{-\tfrac{\left(3\sqrt{w} - 1\right)^2}{3}\right\}, \quad \forall w \geq \tfrac{1}{9}.$$
Plugging this into the tail-sum identity for expectations,
\begin{align*}
\mathbb E[\exp(W_{k}) \mid \mathcal F_{k}] &= 1 + \tint_0^{1/9} e^w\mathbb P(W_{k}\geq w \mid \mathcal F_{k})dw + \tint_{1/9}^\infty e^w\mathbb P(W_{k}\geq w \mid \mathcal F_{k})dw \\
&\le e^{1/9} + \tint_{1/9}^\infty \exp\left(-2w + 2\sqrt{w} - \tfrac{1}{3}\right)dw\leq\exp(1).
\end{align*}
\end{proof}
\smallskip

At each  $k\geq 0,$ define the state-dependent variance proxy and the corresponding weighted variance proxy by
\begin{equation}\label{eqn:innerproduct-var}
\sigma_k^2\coloneqq\varsigma_*\langle F(x_k)-F(x^*),x_k-x^*\rangle+\sigma_*^2,\quad
\mathcal{V}_k\coloneqq \tsum_{i=0}^k [\theta_i(\eta_i\lambda_i)^2+\theta_{i+1}(\eta_{i+1}\lambda_{i+1})^2]\sigma_i^2.
\end{equation}
We further define
\begin{equation}\label{eqn:beta-def}
\beta_k\coloneqq[\theta_k(\eta_k\lambda_k)^2+\theta_{k+1}(\eta_{k+1}\lambda_{k+1})^2]\tfrac{\sigma_k^2}{\mathcal{V}_k}
\quad\text{and}\qquad W_k\coloneqq
\begin{cases}
\tfrac{m\|\Delta_k\|^2}{9\sigma_k^2},&\sigma_k^2>0,\\[1ex]
0,&\sigma_k^2=0,
\end{cases}
\end{equation}
where we recall from \eqref{eqn:mini-batch-estimator} and \eqref{eqn:delta-error} that $\tilde F_k(x)
\coloneqq
\tfrac{1}{m}\tsum_{i=1}^{m}\tilde F(x;\zeta_{k,i})$ and $\Delta_k\coloneqq\tilde F_k(x_k)-F(x_k)$. Under the state-dependent noise assumption, $\sigma_k^2=0$
implies $\Delta_k=0$ a.s.
By \autoref{prop:moment-tail-bound}, for any $\mathcal F_k$-measurable point $x$, it holds that
\begin{equation}\label{eqn:subgaussian-to-use}
\mathbb E\left[
\exp\left\{
\tfrac{m\|\tilde F_k(x)-F(x)\|^2}
{9[\varsigma_*\langle F(x)-F(x^*),x-x^*\rangle+\sigma_*^2]}
\right\}
\mid\mathcal F_k
\right]
\leq \exp(1).
\end{equation}

Unlike the uniform-noise case, the conditional
variance parameters are random. Therefore, we use self-normalized
concentration bounds based on nonnegative supermartingales. We start with the following guarantee for the cumulative squared gradient error.
\begin{proposition}\label{prop-gradient-norm-square}
Suppose unbiased estimator assumption \eqref{eq:fast-decreasing-bias-o} and state-dependent noise assumption \eqref{eq:hp-bounded-var-at-opt-o} hold. Suppose $\mathcal V_{0}>0,$
then, for any $\Lambda\geq0$, with probability at least $1-\exp\{-\Lambda\},$
we have
\begin{equation}
    \begin{aligned}
&\tsum_{k=0}^{N-1}[\theta_{k} (\eta_{k}\lambda_{k})^2+\theta_{k+1} (\eta_{k+1}\lambda_{k+1})^2]\|\Delta_k\|^2\\
&\leq\left(1+\Lambda+\log\tfrac{\mathcal V_{N-1}}{\mathcal V_{0}}\right)
\tsum_{k=0}^{N-1}\tfrac{9\sigma_k^2}{m}[\theta_{k} (\eta_{k}\lambda_{k})^2+\theta_{k+1} (\eta_{k+1}\lambda_{k+1})^2].
\end{aligned}
\end{equation}
\end{proposition}

\begin{proof}
Notice that the algorithmic parameters $\theta_k,\eta_k,\lambda_k$ are deterministic. Furthermore, for all $0\leq k\leq N$, $W_k$ is $\mathcal F_{k+1}$-measurable, while $\sigma_k^2$, $\mathcal V_k$, and $\beta_k$ are $\mathcal F_k$-measurable, and $0\leq\beta_k\leq1$. Therefore, we obtain
\begin{equation}\label{eqn:super-1}
\mathbb E[\exp\{\beta_k(W_k-1)\}\mid\mathcal F_k]
\overset{\text{(i)}}{\leq} \exp\{-\beta_k\}\bigl(\mathbb E[\exp\{W_k\}\mid\mathcal F_k]\bigr)^{\beta_k}
\overset{\eqref{eqn:subgaussian-to-use}}{\leq} 1,
\end{equation}
where in (i),
we used  the fact $x\mapsto x^{\beta_k}$ on $\mathbb R_+$ is a concave function, and the
conditional Jensen's inequality.
Let $M_0\coloneqq1$ and $M_n\coloneqq\exp\{\tsum_{k=0}^{n-1}\beta_k(W_k-1)\}$. Then, $\{M_n\}$ is a nonnegative supermartingale shown as follows.
\begin{equation*}
\mathbb E[M_{n+1}\mid\mathcal F_n]
=M_n\mathbb E[\exp\{\beta_n(W_n-1)\}\mid\mathcal F_n]\overset{\eqref{eqn:super-1}}{\leq} M_n.
\end{equation*}
Moreover, by the definition of $M_{N}$, we have
\begin{equation}\label{eqn:equivalent-event}
A_{N}\coloneqq\left\{M_{N}\leq\exp\{\Lambda\}\right\}
=
\left\{\tsum_{k=0}^{N-1}\beta_kW_k
\leq\Lambda+\tsum_{k=0}^{N-1}\beta_k\right\}.
\end{equation}
Furthermore, by Markov's inequality, we obtain
\begin{equation*}
\mathbb{P}(A_{N}^c)\leq\mathbb P\left(M_{N}\geq \exp\{\Lambda\}\right)
\leq \exp\{-\Lambda\}\mathbb E[M_{N}]
\overset{\text{(ii)}}{\leq}\exp\{-\Lambda\}.
\end{equation*}
Here (ii) follows from
the property of the  supermartingale
$\mathbb E[M_{N}]\leq\mathbb E[M_0]=1$.
Therefore,  on $A_N$, we have
\begin{equation*}
\begin{aligned}
\tfrac{m}{9\mathcal V_{N-1}}
\tsum_{k=0}^{N-1}
[\theta_k(\eta_k\lambda_k)^2+\theta_{k+1}(\eta_{k+1}\lambda_{k+1})^2]\|\Delta_k\|^2
\overset{\text{(iii)}}{\leq}
\tsum_{k=0}^{N-1}\beta_kW_k
\overset{\text{(iv)}}{\leq}
\Lambda+\tsum_{k=0}^{N-1}\beta_k
\overset{\text{(v)}}{\leq}
1+\Lambda+\log\tfrac{\mathcal V_{N-1}}{\mathcal V_0}.
\end{aligned}
\end{equation*}
where in (iii), we used $\mathcal{V}_{N-1}\geq\mathcal{V}_k;$ and in (iv), we used \eqref{eqn:equivalent-event}; in (v), we  use $1-x\leq -\log x$ and hence
\begin{equation*}
    \begin{aligned}
        \tsum_{k=0}^{N-1}\beta_k
\overset{\eqref{eqn:beta-def}}{=}1+\tsum_{k=1}^{N-1}\left(1-\tfrac{\mathcal{V}_{k-1}}{\mathcal{V}_{k}}\right)
\leq 1+\log\tfrac{\mathcal{V}_{N-1}}{\mathcal{V}_{0}}.
    \end{aligned}
\end{equation*}
This concludes the proof.
\end{proof}

\medskip

 We next bound the cumulative inner-product term. Define
\begin{align}
d_k
&\coloneqq
\theta_{k-1}\eta_{k-1}
\langle\Delta_k,x_k-x^*\rangle,
\quad\text{and}\quad
s_k^2
\coloneqq
\tfrac{9(\theta_{k-1}\eta_{k-1})^2
\|x_k-x^*\|^2\sigma_k^2}{m},
\end{align}
where $\Delta_k=\tilde F_k(x_k)-F(x_k)$ and $\sigma_k^2$ is
defined in \eqref{eqn:innerproduct-var}. Furthermore, let
$\tilde v_0>0$ be an arbitrary deterministic constant and for all $n\geq 1,$ define
\begin{equation}\label{eqn:U-def}
\tilde{\mathcal V}_n
\coloneqq
\tilde v_0+\tsum_{k=0}^{n}s_{k+1}^2.
\end{equation}
The following result bounds the cumulative inner-product term.
\begin{proposition}\label{prop:inner-product-concentration}
Suppose unbiased estimator assumption \eqref{eq:fast-decreasing-bias-o} and state-dependent noise assumption \eqref{eq:hp-bounded-var-at-opt-o} hold.
For any $\Lambda\geq0$, with probability at least $1-\exp\{-\Lambda\},$
we have
\begin{equation*}
\tsum_{k=0}^{N-2}
\theta_k\eta_k
\langle\Delta_{k+1}, x^*-x_{k+1}\rangle
\leq
\sqrt{
2(e+1)(1+\Lambda)\tilde{\mathcal V}_{N-2}
\left(
1+\tfrac{1}{2}
\log\tfrac{\tilde{\mathcal V}_{N-2}}{\tilde v_0}
\right)}
\end{equation*}
\end{proposition}

\begin{proof}
Notice that $d_{k+1}$ is $\mathcal F_{k+2}$-measurable and
$s_{k+1}^2$ is $\mathcal F_{k+1}$-measurable. Since $x_{k+1}$ is
$\mathcal F_{k+1}$-measurable and
$\mathbb E[\Delta_{k+1}\mid\mathcal F_{k+1}]=0$, we have
$\mathbb E[d_{k+1}\mid\mathcal F_{k+1}]=0$. Moreover, by the Cauchy--Schwarz inequality, we have
\begin{equation}\label{eqn:inner-square-mgf}
\mathbb E\left[
\exp\left\{\tfrac{d_{k+1}^2}{s_{k+1}^2}\right\}
\middle|\mathcal F_{k+1}
\right]
\leq
\mathbb E\left[
\exp\left\{
\tfrac{m\|\Delta_{k+1}\|^2}{9\sigma_{k+1}^2}
\right\}
\middle|\mathcal F_{k+1}
\right]
\overset{\eqref{eqn:subgaussian-to-use}}{\leq}
\exp\{1\},
\end{equation}
If $s_{k+1}=0$, then $d_{k+1}=0$ almost surely, and we adopt
the convention $d_{k+1}^2/s_{k+1}^2=0$.
Following the proof of
\cite[Proposition 2.6.1, step $(\mathrm{iii})\Rightarrow(\mathrm{iv})$]{vershynin2026}, for all $t\in\mathbb{R},$ we have
\begin{equation}\label{eqn:inner-linear-mgf}
\begin{aligned}
\mathbb E\left[\exp\{td_{k+1}\}\mid\mathcal F_{k+1}\right]
&\overset{\text{(i)}}{\leq}1+\tfrac{t^2}{2}\mathbb E\left[d_{k+1}^2\exp\{|td_{k+1}|\}\mid\mathcal F_{k+1}\right]\\
&\overset{\text{(ii)}}{\leq}1+\tfrac{t^2}{2}\exp\left\{\tfrac{t^2s_{k+1}^2}{2}\right\}\mathbb E\left[d_{k+1}^2\exp\left\{\tfrac{d_{k+1}^2}{2s_{k+1}^2}\right\}\middle|\mathcal F_{k+1}\right]\\
&\overset{\text{(iii)}}{\leq}1+\tfrac{et^2s_{k+1}^2}{2}\exp\left\{\tfrac{t^2s_{k+1}^2}{2}\right\}\overset{\text{(iv)}}{\leq}\left(1+\tfrac{et^2s_{k+1}^2}{2}\right)\exp\left\{\tfrac{t^2s_{k+1}^2}{2}\right\}\overset{\text{(v)}}{\leq}\exp\left\{\tfrac{(e+1)t^2s_{k+1}^2}{2}\right\},
\end{aligned}
\end{equation}
where in (i), we used $e^u\leq1+u+\tfrac{u^2}{2}e^{|u|}$ and $\mathbb E[d_{k+1}\mid\mathcal F_{k+1}]=0$; in (ii), we used $|td_{k+1}|\leq\tfrac{t^2s_{k+1}^2}{2}+\tfrac{d_{k+1}^2}{2s_{k+1}^2}$; in (iii), we used $ze^{z/2}\leq e^z$ with $z=d_{k+1}^2/s_{k+1}^2$ and \eqref{eqn:inner-square-mgf}; in (iv), we used $1\leq\exp\{t^2s_{k+1}^2/2\}$; and in (v), we used $1+z\leq e^z$.
\smallskip

Let $\tilde  M_0(t )\coloneqq1$ and for
$1\leq n\leq N-1$, define
\begin{equation*}
\tilde  M_n(t )
\coloneqq
\exp\left\{
t \tsum_{k=0}^{n-1}d_{k+1}
-\tfrac{e+1}{2}t ^2
\left(
\tilde{\mathcal V}_{n-1}-\tilde v_0
\right)
\right\}.
\end{equation*}
Then, $\{\tilde  M_n(t)\}_{n=0}^{N-1}$ is a nonnegative supermartingale. Indeed, for all $0\leq n\leq N-2$, we obtain
\begin{equation*}
\begin{aligned}
\mathbb E\left[
\tilde  M_{n+1}(t)
\mid\mathcal F_{n+1}
\right]
&=
\tilde  M_n(t)
\mathbb E\left[
\exp\left\{
t d_{n+1}
-\tfrac{e+1}{2}t^2s_{n+1}^2
\right\}
\middle|\mathcal F_{n+1}
\right]\overset{\eqref{eqn:inner-linear-mgf}}{\leq}
\tilde  M_n(t).
\end{aligned}
\end{equation*}
By the supermartingale property, we have
$\mathbb E[\tilde  M_{N-1}(t)]
\leq\mathbb E[\tilde  M_0(t)]=1$. 
\smallskip

Therefore, we have
$\mathbb E[
\exp\{t X-t^2Y^2/2\}]
\leq1$
for every $t\in\mathbb R$, where 
\begin{equation*}
X\coloneqq
\tsum_{k=0}^{N-2}d_{k+1}
\quad\text{and}\quad
Y^2\coloneqq
(e+1)
\left(
\tilde{\mathcal V}_{N-2}-\tilde v_0
\right).
\end{equation*}
Hence, the conditions in
\cite[Corollary 2.2]{deLaPenaKlassLai2004} are satisfied, and we have for all $x \geq \sqrt{2}, y>0,$ it holds that
\begin{equation}\label{eqn:inner-self-normalized}
\mathbb P\left\{
|X|>
x\sqrt{
(Y^2+y)
\left(
1+\tfrac{1}{2}
\log\left(1+\tfrac{Y^2}{y}\right)
\right)}
\right\}
\leq
\exp\left\{-\tfrac{x^2}{2}\right\}. 
\end{equation}
Therefore, by setting
$x\coloneqq\sqrt{2(1+\Lambda)}$ and
$y\coloneqq(e+1)\tilde v_0$, we obtain
$Y^2+y=(e+1)\tilde{\mathcal V}_{N-2}$ and $1+\tfrac{Y^2}{y}
=
\tfrac{\tilde{\mathcal V}_{N-2}}{\tilde v_0}.$
Substituting these identities into
\eqref{eqn:inner-self-normalized} and using
$\{X>t\}\subseteq\{|X|>t\}$ conclude the proof.
\end{proof}
\medskip
Unlike the uniform-noise case, both sides of the bounds in \autoref{prop-gradient-norm-square} (resp. \autoref{prop:inner-product-concentration}) are random. Indeed, conditional on $\mathcal F_k$, the error $\|\Delta_k\|^2$ (resp. $d_{k}^2$) satisfies an exponential-moment bound with the $\mathcal F_k$-measurable random scale $\sigma_k^2$ (resp. $s_{k}^2$). The proof therefore uses the cumulative variance proxy $\mathcal V_k$ (resp. $\tilde{\mathcal V}_k$) to self-normalize the errors and constructs a nonnegative supermartingale to obtain the high-probability guarantee. This self-normalization introduces the additional term $\log(\mathcal V_{N-1}/\mathcal V_0)$ (resp. $\log(\tilde{\mathcal V}_{N-2}/\widetilde{v}_0)$) compared with the corresponding uniform-noise bound.

Let $\{q_s\}_{s\geq0}$ be a deterministic nondecreasing sequence with $q_s\geq1$. For fixed $c_{\mathcal A}>0$, $\Lambda\geq0$, and $k\geq0$, define
\begin{equation}\label{eqn:hp-induction-event}
\begin{aligned}
B_k := q_s\left[ \tfrac{c_{\mathcal A}L^2\|x_0-x^*\|^2}{\mu^2} +\tfrac{c_{\mathcal A}(1+\Lambda)k\sigma^2_*}{\mu^2m} \right],\quad \widetilde{\mathcal E}_k &:= \left\{ \theta_k\|x_k-x^*\|^2\leq B_k \right\}, \quad \widehat{\mathcal E}_k := \bigcap_{i=0}^k\widetilde{\mathcal E}_i,
\end{aligned}
\end{equation}
For every fixed $N\geq2$, set $\widetilde v_{0}:=\tfrac{144\sigma_0^2B_{N-1}}{\mu^2m}$ and define $\ell_N(\Lambda)\coloneqq 72(e+2) \left( 1+\Lambda +\log\tfrac{\mathcal V_{N-1}}{\mathcal V_0} +\tfrac12 \log\tfrac{\tilde{\mathcal V}_{N-2}}{\widetilde v_{0}} \right). $
\medskip

\begin{theorem}\label{thm:hp-state-dependent-inner-loop}
Let $\{x_{ k}\}_{k\geq -1}$ denote the iterates generated by SOE \autoref{alg:SOE}, and let $\{\theta_{ k}\}_{k\geq 0}$ be a sequence of positive numbers. Let $\{\eta_k\}_{k\geq 0}$ and $\{\lambda_k\}_{k\geq 0}$ denote the stepsizes and extrapolation weights, with $\lambda_0=0$. Suppose that, for all $k\geq 0$, they satisfy \eqref{eq:inner-loop-B}, \eqref{eq:inner-loop-A} and \eqref{eq:inner-loop-C}. Furthermore, fix $N\geq2$, assume that $\mathcal V_0>0$ and $\widetilde v_{0}>0$, and suppose that, on $\widehat{\mathcal E}_{N-1}$, for all $0\leq k\leq N-2,$
\begin{align}
\tfrac{3(1+\Lambda)\ell_N(\Lambda)\varsigma_*}{m} \left[ \theta_{k+1}(\eta_{k+1}\lambda_{k+1})^2 +\theta_{k+2}(\eta_{k+2}\lambda_{k+2})^2 \right] \leq \theta_k\eta_k. \label{eq:hp-inner-loop-D}
\end{align}
Then,  with probability at least $1-2\exp\{-\Lambda\}$, either $\widehat{\mathcal E}_{N-1}$ does not hold or the following holds
\begin{equation}\label{eq:hp-main-thm}
\begin{aligned}
&\tfrac{\theta_N}{8}\|x_N-x^*\|^2 +\tsum_{k=0}^{N-1}\tfrac{\theta_k}{8}\|x_{k+1}-x_k\|^2 +\tsum_{k=1}^{N}\tfrac{\theta_{k-1}\eta_{k-1}}{6} \langle F(x_k)-F(x^*),x_k-x^*\rangle\\
&\leq \left\{ \tfrac{\theta_0}{2} +\tfrac{c_{\mathcal A}q_{N-1}L^2}{16\mu^2} +\tfrac{(1+\Lambda)\ell_N(\Lambda) L\varsigma_*}{m} \left[ \theta_1(\eta_1\lambda_1)^2+\tfrac{16}{\mu^2} \right] \right\}\|x_0-x^*\|^2\\
&\quad+\left\{\ell_N(\Lambda) \left[\tfrac{16}{\mu^2}+\tsum_{k=0}^{N-1} \left( \theta_k(\eta_k\lambda_k)^2 +\theta_{k+1}(\eta_{k+1}\lambda_{k+1})^2\right)\right] +\tfrac{c_{\mathcal A}q_{N-1}(N-1)}{16\mu^2} \right\} \tfrac{(1+\Lambda)\sigma_*^2}{m}.
\end{aligned}
\end{equation}
\end{theorem}

\begin{proof}
By \autoref{lem:per-itrate}, for every trajectory and any $N\geq2$,
\begin{equation*}
\begin{aligned}
&\tsum_{k = 0}^{N-1}\theta_{ k}\left[\tfrac{1}{2}\|x_{ k} - x^*\|^2-\tfrac{1}{2}\|x_{  k + 1} - x^*\|^2-\eta_{  k}\langle F(x_{  k + 1}), x_{  k + 1}-x^* \rangle\right]+\tfrac{3\theta_{N}}{8}\| x_{ N} - x^* \|^2+\theta_{  N} (\eta_{  N} \lambda_{ N})^2\|\Delta_{  N - 1}\|^2 \\
&\geq\tsum_{k = 0}^{N-1}\left[\tfrac{\theta_k}{8}\|x_{  k + 1} - x_{ k}\|^2-4\theta_k(\eta_{  k} \lambda_{ k})^2(\|\Delta_{  k}\|^2 + \|\Delta_{  k - 1}\|^2)\right]+\tsum_{k = 0}^{N-2}\theta_{ k}\eta_{  k} \langle \Delta_{ {k + 1}}, x_{  k + 1} - x^*\rangle.
\end{aligned}
\end{equation*}
Since $\Delta_{-1}=\Delta_0$ and $\lambda_0=0$,
\begin{align*}
&4\tsum_{k = 0}^{N-1}\theta_k(\eta_{  k} \lambda_{ k})^2(\|\Delta_{  k}\|^2 + \|\Delta_{  k - 1}\|^2)+\theta_{  N} (\eta_{  N} \lambda_{ N})^2\|\Delta_{  N - 1}\|^2\\
&\leq 4\tsum_{k = 0}^{N-1}[\theta_k(\eta_{  k} \lambda_{ k})^2+\theta_{k+1}(\eta_{  k+1} \lambda_{ k+1})^2]\|\Delta_k\|^2.
\end{align*}
Hence, by \autoref{prop-gradient-norm-square}, with probability at least $1- \exp\{-\Lambda\},$ we have
\begin{equation} \label{eqn:term-1}
\begin{aligned}
&\tsum_{k = 0}^{N - 1}[\theta_{k} (\eta_{k}\lambda_{k})^2+\theta_{k+1} (\eta_{k+1}\lambda_{k+1})^2]\|\Delta_k\|^2\leq\tfrac{9}{m}\left(1+\Lambda+\log\tfrac{\mathcal V_{N-1}}{\mathcal V_0}\right)\tsum_{k = 0}^{N-1} [\theta_{k} (\eta_{k}\lambda_{k})^2+\theta_{k+1} (\eta_{k+1}\lambda_{k+1})^2]\sigma_k^2,
\end{aligned}
\end{equation}
where $\sigma_k^2$ is defined in \eqref{eqn:innerproduct-var}. Furthermore, on $\widehat{\mathcal E}_{N-1}$, we have $\theta_k\|x_k-x^*\|^2\leq B_k\leq B_{N-1}$ because $r_k$ is nondecreasing. By the definition of $s_k^2$, we have
\begin{equation*}
\begin{aligned}
s_k^2 &= \tfrac{9(\theta_{k-1}\eta_{k-1})^2 \|x_k-x^*\|^2\sigma_k^2}{m}\leq \tfrac{9B_{N-1}}{m} \tfrac{(\theta_{k-1}\eta_{k-1})^2}{\theta_k}\sigma_k^2 \overset{\eqref{eq:inner-loop-A}}{=} \tfrac{9B_{N-1}}{m} \theta_k(\eta_k\lambda_k)^2\sigma_k^2, \qquad 1\leq k\leq N-1.
\end{aligned}
\end{equation*}
Furthermore, by the choice of $\widetilde v_{0}=\tfrac{144\sigma_0^2B_{N-1}}{\mu^2m}$ and the definition of $\tilde{\mathcal V}_{N-2}$, we obtain
\begin{equation}\label{eqn:hp-tilde-V-theorem}
\begin{aligned}
\tilde{\mathcal V}_{N-2} &\overset{\eqref{eqn:U-def}}{\leq} \tfrac{9B_{N-1}}{m} \left[ \tfrac{16\sigma_0^2}{\mu^2} +\tsum_{k=1}^{N-1} \theta_k(\eta_k\lambda_k)^2\sigma_k^2 \right].
\end{aligned}
\end{equation}
Hence, by \autoref{prop:inner-product-concentration}, with probability at least $1 - \exp(-\Lambda)$,
\begin{equation}\label{eqn:term-2}
\begin{aligned}
\tsum_{k = 0}^{N - 2}\theta_k \eta_k \langle \Delta_{k + 1}, x^*-x_{k + 1} \rangle &\leq \sqrt{2(e+1)(1+\Lambda)\tilde{\mathcal V}_{N-2}\left(1+\tfrac{1}{2}\log\tfrac{\tilde{\mathcal V}_{N-2}}{\widetilde v_{0}}\right)}\\
&\overset{\text{(i)}}{\leq} \tfrac{B_{N-1}}{16} +\tfrac{72(e+1)(1+\Lambda)}{m} \left(1+\tfrac{1}{2}\log\tfrac{\tilde{\mathcal V}_{N-2}}{\widetilde v_{0}}\right)\cdot \left[ \tfrac{16\sigma_0^2}{\mu^2} +\tsum_{k=1}^{N-1} \theta_k(\eta_k\lambda_k)^2\sigma_k^2 \right],
\end{aligned}
\end{equation}
where in (i), we used Young's inequality and \eqref{eqn:hp-tilde-V-theorem}.
\medskip

Substituting \eqref{eqn:term-1} and \eqref{eqn:term-2} into \autoref{lem:per-itrate}, we have
\begin{equation}
\begin{aligned}
&\tsum_{k = 0}^{N-1}\theta_{ k}\left[\tfrac{1}{2}\|x_{ k} - x^*\|^2-\tfrac{1}{2}\|x_{  k + 1} - x^*\|^2-\eta_{  k}\langle F(x_{  k + 1}), x_{  k + 1}-x^* \rangle\right]+\tfrac{3\theta_{N}}{8}\| x_{ N} - x^* \|^2\\
&\geq\tsum_{k = 0}^{N-1}\left[\tfrac{\theta_k}{8}\|x_{  k + 1} - x_{ k}\|^2\right]-\tfrac{36}{m} \left[ 1+\Lambda +\log\tfrac{\mathcal V_{N-1}}{\mathcal V_0} \right]\cdot \tsum_{k=0}^{N-1} \left[ \theta_k(\eta_k\lambda_k)^2 +\theta_{k+1}(\eta_{k+1}\lambda_{k+1})^2 \right] \sigma_k^2\\
&\quad- \tfrac{B_{N-1}}{16} -\tfrac{72(e+1)(1+\Lambda)}{m} \left(1+\tfrac{1}{2}\log\tfrac{\tilde{\mathcal V}_{N-2}}{\widetilde v_{0}}\right)\cdot \left[ \tfrac{16\sigma_0^2}{\mu^2} +\tsum_{k=1}^{N-1} \theta_k(\eta_k\lambda_k)^2\sigma_k^2 \right].
\end{aligned}
\end{equation}
Furthermore, using $\mathcal V_{N-1}\geq\mathcal V_0$, $\tilde{\mathcal V}_{N-2}\geq\widetilde v_{0}$, and $\Lambda\geq0$, we have
\begin{equation*}
\begin{aligned}
&\left(1+\Lambda+\log\tfrac{\mathcal V_{N-1}}{\mathcal V_0}\right)+2(e+1)(1+\Lambda) \left(1+\tfrac{1}{2}\log\tfrac{\tilde{\mathcal V}_{N-2}}{\widetilde v_{0}}\right)\leq2(e+2)(1+\Lambda) \left[ 1+\Lambda+\log\tfrac{\mathcal V_{N-1}}{\mathcal V_0} +\tfrac{1}{2}\log\tfrac{\tilde{\mathcal V}_{N-2}}{\widetilde v_{0}} \right].
\end{aligned}
\end{equation*}
Hence, by $\lambda_0=0$, the $L$-Lipschitz continuity of $F$, and $\theta_k(\eta_k\lambda_k)^2 \leq \theta_k(\eta_k\lambda_k)^2 +\theta_{k+1}(\eta_{k+1}\lambda_{k+1})^2,$ we obtain
\begin{equation}\label{eqn:hp-before-I-N}
\begin{aligned}
&\tsum_{k=0}^{N-1}\theta_k \left[ \tfrac{1}{2}\|x_k-x^*\|^2 -\tfrac{1}{2}\|x_{k+1}-x^*\|^2 \right] +\tfrac{3\theta_N}{8}\|x_N-x^*\|^2\\
&\quad+\tfrac{B_{N-1}}{16} +\tfrac{72(e+2)(1+\Lambda)L\varsigma_*}{m} \left[ 1+\Lambda+\log\tfrac{\mathcal V_{N-1}}{\mathcal V_0} +\tfrac{1}{2}\log\tfrac{\tilde{\mathcal V}_{N-2}}{\widetilde v_{0}} \right] \left[ \theta_1(\eta_1\lambda_1)^2+\tfrac{16}{\mu^2} \right] \|x_0-x^*\|^2\\
&\quad+\tfrac{72(e+2)(1+\Lambda)\sigma_*^2}{m} \left[ 1+\Lambda+\log\tfrac{\mathcal V_{N-1}}{\mathcal V_0} +\tfrac{1}{2}\log\tfrac{\tilde{\mathcal V}_{N-2}}{\widetilde v_{0}} \right] \left\{ \tfrac{16}{\mu^2} +\tsum_{k=0}^{N-1} \left[ \theta_k(\eta_k\lambda_k)^2 +\theta_{k+1}(\eta_{k+1}\lambda_{k+1})^2 \right] \right\}\\
&\geq \tsum_{k=0}^{N-1}\tfrac{\theta_k}{8}\|x_{k+1}-x_k\|^2 +\mathcal I_N.
\end{aligned}
\end{equation}
where, combining with by \eqref{eq:x-star-solves-the-VI}, we  define $\mathcal I_N$ as
\begin{equation*}
\begin{aligned}
\mathcal I_N &\coloneqq \tsum_{k=1}^{N-1} \left\{ \theta_{k-1}\eta_{k-1} -\tfrac{72(e+2)(1+\Lambda)\varsigma_*}{m} \left[ 1+\Lambda+\log\tfrac{\mathcal V_{N-1}}{\mathcal V_0} +\tfrac{1}{2}\log\tfrac{\tilde{\mathcal V}_{N-2}}{\widetilde v_{0}} \right] \cdot \left[ \theta_k(\eta_k\lambda_k)^2 +\theta_{k+1}(\eta_{k+1}\lambda_{k+1})^2 \right] \right\}\\
&\quad\cdot \langle F(x_k)-F(x^*),x_k-x^*\rangle+\theta_{N-1}\eta_{N-1} \langle F(x_N)-F(x^*),x_N-x^*\rangle\\
&\overset{\eqref{eq:hp-inner-loop-D}}{\geq} \tsum_{k=1}^{N} \left[ \tfrac{\theta_{k-1}\eta_{k-1}}{2} +\tfrac{\theta_{k-1}\eta_{k-1}}{6} \right] \langle F(x_k)-F(x^*),x_k-x^*\rangle\\
&\,\,\,\geq \tsum_{k=0}^{N-1} \tfrac{\mu\theta_k\eta_k}{2}\|x_{k+1}-x^*\|^2 +\tsum_{k=1}^{N} \tfrac{\theta_{k-1}\eta_{k-1}}{6} \langle F(x_k)-F(x^*),x_k-x^*\rangle,
\end{aligned}
\end{equation*}
where the last inequality follows from the strong monotonicity of $F$. Finally, by \eqref{eq:inner-loop-C}, we have
\begin{equation*}
\begin{aligned}
&\tsum_{k=0}^{N-1}\tfrac{\theta_k}{2} \left(\|x_k-x^*\|^2-\|x_{k+1}-x^*\|^2\right) +\tfrac{3\theta_N}{8}\|x_N-x^*\|^2 -\tsum_{k=0}^{N-1} \tfrac{\mu\theta_k\eta_k}{2}\|x_{k+1}-x^*\|^2\\
&\leq \tfrac{\theta_0}{2}\|x_0-x^*\|^2 -\tfrac{\theta_N}{8}\|x_N-x^*\|^2.
\end{aligned}
\end{equation*}
By a union bound, the concentration bounds in \eqref{eqn:term-1} and \eqref{eqn:term-2} hold simultaneously with probability at least $1-2\exp\{-\Lambda\}$. On this event, by \eqref{eqn:hp-before-I-N} and the definition of $B_{N-1}$ in \eqref{eqn:hp-induction-event}, we conclude either $\widehat{\mathcal E}_{N-1}$ does not hold or \eqref{eq:hp-main-thm} holds.
\end{proof}
\medskip
We next bound the logarithmic terms arising from the random state-dependent variance.
\begin{lemma}\label{lem:hp-variance-ratio}
Suppose $k_0, \eta_k,\theta_k$ are chosen in \eqref{eqn:stepsize-inner}. Let $N\geq2$ and $\Lambda >0$, assume that $\mathcal V_0>0$ and $\widetilde v_{0}>0$, and further suppose that $m \geq \lceil \tfrac{c_{\mathcal A}\varsigma_*(1+\Lambda)}{L} \rceil,$ and  $\widetilde v_{0} = \tfrac{144\sigma_0^2B_{N-1}}{\mu^2m}.$ Then, on $\widehat{\mathcal E}_{N-1}$, we have
\begin{equation}\label{eqn:hp-log-ratio-lemma}
\begin{aligned}
\log\tfrac{\mathcal V_{N-1}}{\mathcal V_0} &\leq 2\log(N+k_0)+\log(1+c_{\mathcal A})+\log q_{N-1}\\
\log\tfrac{\tilde{\mathcal V}_{N-2}}{\widetilde v_{0}} &\leq 2\log(N+k_0)+\log(1+c_{\mathcal A})+\log q_{N-1}.
\end{aligned}
\end{equation}
\end{lemma}

\begin{proof}
For all $k\geq1$, substituting the parameters from \eqref{eqn:stepsize-inner}, we have $\theta_k(\eta_k\lambda_k)^2 = \tfrac{16}{\mu^2}\tfrac{k+k_0-1}{k+k_0} \leq \tfrac{16}{\mu^2},$ and hence
\begin{equation}\label{eqn:hp-weight-sums-lemma}
\begin{aligned}
\tsum_{k=0}^{N-1} \left[ \theta_k(\eta_k\lambda_k)^2 +\theta_{k+1}(\eta_{k+1}\lambda_{k+1})^2 \right] &\leq \tfrac{32N}{\mu^2},\\
\tsum_{k=0}^{N-1} \tfrac{ \theta_k(\eta_k\lambda_k)^2 +\theta_{k+1}(\eta_{k+1}\lambda_{k+1})^2 }{\theta_k} &\leq \tfrac{32}{\mu^2} \tsum_{k=0}^{N-1} \tfrac{1}{(k+k_0)(k+k_0-1)} \leq \tfrac{32}{\mu^2(k_0-1)} \leq \tfrac{32}{15\mu L}.
\end{aligned}
\end{equation}
On $\widehat{\mathcal E}_{N-1}$, the $L$-Lipschitz continuity of $F$ implies $\sigma_k^2\leq \sigma^2_*+\varsigma_*L B_k/\theta_k$ for all $0\leq k\leq N-1$. Since $B_k\leq B_{N-1}$ for all $0\leq k\leq N-1$, substituting this and \eqref{eqn:hp-weight-sums-lemma} into $\mathcal V_{N-1}$ defined in \eqref{eqn:innerproduct-var}, we obtain
\begin{equation}\label{eqn:hp-V-bound-lemma}
\begin{aligned}
\mathcal V_{N-1} &\leq \tfrac{32N\sigma^2_*}{\mu^2} +\tfrac{32\varsigma_*B_{N-1}}{15\mu}\leq \tfrac{32N\sigma^2_*}{\mu^2} +\tfrac{32c_{\mathcal A}q_{N-1}\varsigma_*L^2\|x_0-x^*\|^2}{15\mu^3} +\tfrac{32c_{\mathcal A}q_{N-1}\varsigma_*(1+\Lambda)(N-1)\sigma^2_*}{15\mu^3m}.
\end{aligned}
\end{equation}
Moreover, by the strong monotonicity of $F,$ we have $\sigma_0^2\geq \sigma^2_*+\varsigma_*\mu\|x_0-x^*\|^2$. Hence, we have
\begin{equation}\label{eqn:hp-V0-bound-lemma}
\begin{aligned}
\mathcal V_0 &= \theta_1(\eta_1\lambda_1)^2\sigma_0^2 \geq \tfrac{8}{\mu^2} \left( \sigma^2_*+\varsigma_*\mu\|x_0-x^*\|^2 \right).
\end{aligned}
\end{equation}
Combining the batch size condition $m \geq \left\lceil \tfrac{c_{\mathcal A}\varsigma_*(1+\Lambda)}{L} \right\rceil,$ \eqref{eqn:hp-V-bound-lemma}, and \eqref{eqn:hp-V0-bound-lemma}, we obtain
\begin{equation*}
\begin{aligned}
\tfrac{\mathcal V_{N-1}}{\mathcal V_0} &\leq 4N+\tfrac{4q_{N-1}NL}{15\mu}+\tfrac{4c_{\mathcal A}q_{N-1}L^2}{15\mu^2} \leq (1+c_{\mathcal A})q_{N-1}(N+k_0)^2.
\end{aligned}
\end{equation*}
We next bound $\tilde{\mathcal V}_{N-2}$. By the definition of $B_k$ in \eqref{eqn:hp-induction-event} and $B_k\leq B_{N-1}$, for all $1\leq k\leq N-1$, we have
\begin{equation}\label{eqn:hp-sk-bound-lemma}
\begin{aligned}
s_k^2 = \tfrac{9(\theta_{k-1}\eta_{k-1})^2\|x_k-x^*\|^2\sigma_k^2}{m}\leq \tfrac{9B_{N-1}}{m} \tfrac{(\theta_{k-1}\eta_{k-1})^2}{\theta_k}\sigma_k^2\overset{\eqref{eqn:stepsize-inner}}{=} \tfrac{144B_{N-1}}{\mu^2m}\cdot \tfrac{k+k_0-1}{k+k_0}\sigma_k^2 \leq \tfrac{144B_{N-1}}{\mu^2m}\sigma_k^2.
\end{aligned}
\end{equation}
By the $L$-Lipschitz continuity of $F$ and the choice of $\theta_k$, we have
\begin{equation}\label{eqn:hp-sigma-sum-lemma}
\begin{aligned}
\tsum_{k=1}^{N-1}\sigma_k^2 &\leq (N-1)\sigma^2_* +\varsigma_*L B_{N-1} \tsum_{k=1}^{N-1} \tfrac{1}{(k+k_0)(k+k_0-1)}\leq (N-1)\sigma^2_* +\tfrac{\varsigma_*\mu B_{N-1}}{16}.
\end{aligned}
\end{equation}
Substituting the choice for $\widetilde v_{0} = \tfrac{144\sigma_0^2B_{N-1}}{\mu^2m},$ \eqref{eqn:hp-sk-bound-lemma} and \eqref{eqn:hp-sigma-sum-lemma} into \eqref{eqn:U-def}, we obtain
\begin{equation}\label{eqn:hp-tilde-V-bound-lemma}
\begin{aligned}
\tilde{\mathcal V}_{N-2} &\leq \tfrac{144B_{N-1}}{\mu^2m} \left[ \sigma_0^2+(N-1)\sigma^2_* +\tfrac{\varsigma_*\mu B_{N-1}}{16} \right].
\end{aligned}
\end{equation}
Dividing \eqref{eqn:hp-tilde-V-bound-lemma} by $\widetilde v_{0}$, using the lower bound of $\sigma_0^2$ in \eqref{eqn:hp-V0-bound-lemma}, and substituting $B_{N-1}$ from \eqref{eqn:hp-induction-event}, we obtain
\begin{equation*}
\begin{aligned}
\tfrac{\tilde{\mathcal V}_{N-2}}{\widetilde v_{0}} \overset{\eqref{eqn:U-def}}{\leq} 1+\tfrac{(N-1)\sigma_*^2}{\sigma_0^2} +\tfrac{\varsigma_*\mu B_{N-1}}{16\sigma_0^2}&\overset{\text{(i)}}{\leq} N+\tfrac{c_{\mathcal A}q_{N-1}L^2}{16\mu^2} +\tfrac{c_{\mathcal A}q_{N-1}\varsigma_*(1+\Lambda)(N-1)}{16\mu m}\\
&\overset{\text{(ii)}}{\leq} N+\tfrac{q_{N-1}NL}{16\mu} +\tfrac{c_{\mathcal A}q_{N-1}L^2}{16\mu^2}\leq (1+c_{\mathcal A})q_{N-1}(N+k_0)^2,
\end{aligned}
\end{equation*}
where in (i), we used the definition of $B_{N-1}$ in \eqref{eqn:hp-induction-event} and the lower bound of $\sigma_0^2$ in \eqref{eqn:hp-V0-bound-lemma}; in (ii), we used the condition on the batch size $m \geq \left\lceil {c_{\mathcal A}\varsigma_*(1+\Lambda)}/{L} \right\rceil$; and in the last step, we used the choice for $k_0$ in \eqref{eqn:stepsize-inner}.
\end{proof}

We now show that \eqref{eqn:hp-induction-event} holds at iteration $N-1$ with probability at least $1-2(N-1)\exp\{-\Lambda\}$ via an appropriate selection of algorithmic parameters. This will allow us to complete the induction and specify the convergence rate in the subroutine assumption \eqref{eqn:hp-distance-s-state-dependent} required for AR.
Throughout the following corollary, the definitions of $B_k$, $\widetilde{\mathcal E}_k$, and $\widehat{\mathcal E}_k$ in \eqref{eqn:hp-induction-event} are taken with the numerical constant $2304$.
Let $k_0:=\tfrac{16L}{\mu}$. For fixed $\Lambda>0$, set $q_0=1,$ and, for all $k\geq1,$ define
\begin{equation}\label{eqn:hp-r-choice}
q_s:=\max\left\{q_{k-1},72(e+2)\left[1+\Lambda+3\log(k+k_0)+\tfrac32\log(2305)+\tfrac32\log q_{k-1}\right]\right\}.
\end{equation}
Furthermore, let  $a_k:=1+\Lambda+\log(k+k_0)$. Choose a fixed constant $C_q\geq1$ sufficiently large so that
$C_q\geq 72(e+2)\left[\tfrac92+\tfrac32\log(2305C_q)\right].$ Notice that  $C_q$ is well defined since the right-hand side grows logarithmically in $C_q$. We have the following sublinear convergence of SOE.
\begin{corollary}\label{cor:hp-state-dependent-concrete-rate}
Suppose the assumptions in \autoref{lem:hp-variance-ratio}, and  further suppose that
\begin{equation}\label{eqn:hp-combined-batch-condition}
m\geq\left\lceil\tfrac{(1+\Lambda)\varsigma_*}{L}\max\left\{2304,\tfrac32q_N\right\}\right\rceil.
\end{equation}
Then, with probability at least $1-2N\exp\{-\Lambda\},$ it holds that
\begin{equation}\label{eqn:hp-sd-subroutine-guarantee}
\begin{aligned}
&(N+k_0)(N+k_0-1)\|x_N-x^*\|^2+\tsum_{k=0}^{N-1}(k+k_0)(k+k_0-1)\|x_{k+1}-x_k\|^2\\
&\quad+\tfrac{16}{3\mu}\tsum_{k=1}^{N}(k+k_0-1)\langle F(x_k)-F(x^*),x_k-x^*\rangle\\
&\leq 2304 q_N\left(\tfrac{L^2\|x_0-x^*\|^2}{\mu^2}+\tfrac{ (1+\Lambda)N\sigma_*^2}{\mu^2m}\right).
\end{aligned}
\end{equation}
Furthermore, $q_N\leq C_q\left[1+\Lambda+\log(N+k_0)\right]$.
\end{corollary}
\begin{proof}
We show by induction that, for all $N\geq2,$ $\widehat{\mathcal E}_N$ defined in \eqref{eqn:hp-induction-event} holds with probability at least $1-2N\exp\{-\Lambda\}$. Recall that $\widehat{\mathcal E}_1$ holds with probability at least $1-2\exp\{-\Lambda\}$. Fix $N\geq2,$ and assume that $\widehat{\mathcal E}_{N-1}$ holds with probability at least $1-2(N-1)\exp\{-\Lambda\}$. Since the choices of $k_0$, $\eta_k$, $\theta_k$, and $\lambda_k$ are the same as those in \autoref{cor:subroutine-guarantee}, conditions \eqref{eq:inner-loop-A}, \eqref{eq:inner-loop-B}, and \eqref{eq:inner-loop-C} all hold. Moreover, for all $k\geq1,$
\begin{equation}\label{eqn:hp-exact-parameter-identities}
\lambda_k=\tfrac{k+k_0-1}{k+k_0},\qquad \theta_k\eta_k=\tfrac{4(k+k_0)}{\mu},\qquad \theta_k(\eta_k\lambda_k)^2=\tfrac{16(k+k_0-1)}{\mu^2(k+k_0)}\leq\tfrac{16}{\mu^2}.
\end{equation}
Hence, we have
\begin{equation}\label{eqn:hp-weight-sum-induction}
\tfrac{16}{\mu^2}+\tsum_{k=0}^{N-1}\left[\theta_k(\eta_k\lambda_k)^2+\theta_{k+1}(\eta_{k+1}\lambda_{k+1})^2\right]\leq\tfrac{32N}{\mu^2}.
\end{equation}
By \autoref{lem:hp-variance-ratio}, on $\widehat{\mathcal E}_{N-1},$ we have
\begin{equation}\label{eqn:hp-ell-by-r}
\ell_N(\Lambda)\leq72(e+2)\left[1+\Lambda+3\log(N+k_0)+\tfrac32\log 2305+\tfrac32\log q_{N-1}\right]\leq q_N.
\end{equation}
Furthermore, by \eqref{eqn:hp-exact-parameter-identities}, for all $0\leq k\leq N-2,$
\begin{equation*}
\theta_{k+1}(\eta_{k+1}\lambda_{k+1})^2+\theta_{k+2}(\eta_{k+2}\lambda_{k+2})^2\leq\tfrac{\theta_k\eta_k}{2L}.
\end{equation*}
Therefore, \eqref{eqn:hp-combined-batch-condition} and \eqref{eqn:hp-ell-by-r} imply
\begin{equation*}
\tfrac{3(1+\Lambda)\ell_N(\Lambda)\varsigma_*}{m}\left[\theta_{k+1}(\eta_{k+1}\lambda_{k+1})^2+\theta_{k+2}(\eta_{k+2}\lambda_{k+2})^2\right]\leq\tfrac{3(1+\Lambda)q_N\varsigma_*}{2Lm}\theta_k\eta_k\overset{\eqref{eqn:hp-combined-batch-condition}}{\leq}\theta_k\eta_k.
\end{equation*}
Hence, \eqref{eq:hp-inner-loop-D} holds, and the conditions of \autoref{thm:hp-state-dependent-inner-loop} apply, and thus \eqref{eq:hp-main-thm} holds.

We proceed with bounding the initial optimality gap term in \autoref{thm:hp-state-dependent-inner-loop}, using \eqref{eqn:hp-combined-batch-condition}, \eqref{eqn:hp-ell-by-r} and $q_{N-1}\leq q_N,$ and  we have
\begin{equation*}
\begin{aligned}
&8\left\{\tfrac{\theta_0}{2}+\tfrac{2304q_{N-1}L^2}{16\mu^2}+\tfrac{(1+\Lambda)\ell_N(\Lambda)L\varsigma_*}{m}\left[\theta_1(\eta_1\lambda_1)^2+\tfrac{16}{\mu^2}\right]\right\}\leq\left[1024+\tfrac{2304q_N}{2}+\tfrac{256q_N}{2304}\right]\tfrac{L^2}{\mu^2}\leq\tfrac{2304q_NL^2}{\mu^2}.
\end{aligned}
\end{equation*}
Moreover, by \eqref{eqn:hp-weight-sum-induction}, we can bound the noise related part as follows
\begin{equation*}
\begin{aligned}
&8\left\{\ell_N(\Lambda)\left[\tfrac{16}{\mu^2}+\tsum_{k=0}^{N-1}\left(\theta_k(\eta_k\lambda_k)^2+\theta_{k+1}(\eta_{k+1}\lambda_{k+1})^2\right)\right]+\tfrac{ 2304q_{N-1}(N-1)}{16\mu^2}\right\}\leq\left(256+\tfrac{ 2304}{2}\right)\tfrac{q_NN}{\mu^2}\leq\tfrac{ 2304q_NN}{\mu^2}.
\end{aligned}
\end{equation*}
Noting that $\theta_N=(N+k_0)(N+k_0-1)$ and $\theta_{k-1}\eta_{k-1}=\tfrac{4(k+k_0-1)}{\mu},$ multiplying \eqref{eq:hp-main-thm} by $8$ and substituting the preceding bounds yields \eqref{eqn:hp-sd-subroutine-guarantee}. In particular, \eqref{eqn:hp-sd-subroutine-guarantee} implies $\theta_N\|x_N-x^*\|^2\leq B_N,$ and hence $\widetilde{\mathcal E}_N$ holds.
Therefore, by \autoref{thm:hp-state-dependent-inner-loop}, we have
\begin{equation*}
\mathbb P\left(\widehat{\mathcal E}_{N-1}\cap\widetilde{\mathcal E}_N^c\right)\leq2\exp\{-\Lambda\}.
\end{equation*}
Since $\widehat{\mathcal E}_N=\widehat{\mathcal E}_{N-1}\cap\widetilde{\mathcal E}_N,$ by the union bound,
\begin{equation*}
\begin{aligned}
\mathbb P\left(\widehat{\mathcal E}_N^c\right)&=\mathbb P\left(\widehat{\mathcal E}_{N-1}^c\cup\left(\widehat{\mathcal E}_{N-1}\cap\widetilde{\mathcal E}_N^c\right)\right)\leq\mathbb P\left(\widehat{\mathcal E}_{N-1}^c\right)+\mathbb P\left(\widehat{\mathcal E}_{N-1}\cap\widetilde{\mathcal E}_N^c\right)\leq2N\exp\{-\Lambda\}.
\end{aligned}
\end{equation*}
This completes the induction, and 
hence \eqref{eqn:hp-sd-subroutine-guarantee} holds with probability at least $1-2N\exp\{-\Lambda\}$.

Finally, we show by induction that $q_s\leq C_qa_k$ for all $k\geq0$. When $k=0,$ $q_0=1\leq C_qa_0$. Suppose $q_{k-1}\leq C_qa_{k-1}$ for some $k\geq1$. Since $a_{k-1}\leq a_k,$ we have $q_{k-1}\leq C_qa_k.$
Moreover, by the definition of $q_s$ in \eqref{eqn:hp-r-choice},
\begin{equation*}
\log q_{k-1}\leq\log C_q+\log a_{k-1}\leq\log C_q+\log a_k.
\end{equation*}
Using $1+\Lambda+3\log(k+k_0)=a_k+2\log(k+k_0)$, $\log(k+k_0)\leq a_k$, $\log a_k\leq a_k$, and $a_k\geq1,$ we obtain
\begin{equation*}
\begin{aligned}
72(e+2)\left[1+\Lambda+3\log(k+k_0)+\tfrac32\log(2305)+\tfrac32\log q_{k-1}\right]
&\leq 72(e+2)\left[\tfrac92+\tfrac32\log(2305C_q)\right]a_k\leq C_qa_k.
\end{aligned}
\end{equation*}
Therefore, by \eqref{eqn:hp-r-choice}, both terms in the maximum defining $q_s$ are bounded by $C_qa_k$, and hence $q_s\leq C_qa_k$. This completes the induction. Consequently, we have $q_N\leq C_q\left[1+\Lambda+\log(N+k_0)\right].$
\end{proof}
\medskip

In addition, \autoref{cor:hp-state-dependent-concrete-rate} also allows the SOE algorithm to meet the residual guarantee of \eqref{eqn:hp-residual-s-state-dependent} as follows.
\begin{corollary}\label{cor:hp-residual-concrete-rate-state}
Suppose the conditions of \autoref{cor:hp-state-dependent-concrete-rate} hold. Let $\Lambda \geq 0$ and $N \geq 2$. It holds that
\begin{equation}\label{eqn:hp-state-dep-slow-residual}
\mathbb P\left\{\res_{F}(x_{N})\leq 34\sqrt{q_N}\left(L\|x_0-x^*\|+\tfrac{\sqrt{N}\sqrt{1+\Lambda}\sigma_*}{\sqrt{m}}\right)\right\} \geq 1 - 2N\exp\{-\Lambda\},
\end{equation}
where $q_N$ is defined in \eqref{eqn:hp-r-choice}.
\end{corollary}

\begin{proof}
Since the selection of $\theta_k$, $\eta_k$ and $\lambda_k$ are identical to those in \autoref{cor:subroutine-guarantee} with similar conditions on $m$, showing that \eqref{eqn:hp-state-dep-slow-residual} holds with probability at least $1 - 2N\exp\{-\Lambda\}$ is similar to the proof of \autoref{lem:slow-residual-last}.
\end{proof}

It is immediate to see that when applied to the regularized operator $F_s$ for the AR framework, the results of \autoref{cor:hp-state-dependent-concrete-rate} and \autoref{cor:hp-residual-concrete-rate-state} imply that SOE satisfies the required rates in \autoref{as:hp-assumption_inner_sto-state-dependent}. The proof is similar to that of \autoref{prop:subroutine}, so we omit it for brevity.
\begin{proposition}\label{prop:state-dependent-high-p}
Suppose that the conditions of \autoref{cor:hp-state-dependent-concrete-rate} are satisfied when the generic parameters $L$ and $\mu$ therein are specialized to $L+r_s$ and $\mu+r_s$, respectively. Then SOE exhibits the following performance guarantee after $k$ iterations: for all $2\leq k\leq N_s$, it holds that
\begin{align*}
\mathbb P\left\{\|x_{s,k}-x_s^*\|^2\leq2304q_s\left[\tfrac{(L+r_s)^2}{(\mu+r_s)^2k^2}\|x_{s-1}-x_s^*\|^2+\tfrac{(1+\Lambda)\sigma_{s,*}^2}{km_s(\mu+r_s)^2}\right]\,\middle|\,\mathcal F_{s,0}\right\}&\geq1-2k\exp\{-\Lambda\},\\
\mathbb P\left\{\res_{F_s}(x_{s,k})\leq34\sqrt{q_s}\left[(L+r_s)\|x_{s-1}-x_s^*\|+\tfrac{\sqrt{k}\sqrt{1 + \Lambda}\sigma_{s,*}}{\sqrt{m_s}}\right]\,\middle|\,\mathcal F_{s,0}\right\}&\geq1-2k\exp\{-\Lambda\},
\end{align*}
where $\sigma_{s,*}^2\coloneqq9\left[\varsigma_s r_s\|x_s^*-\bar x_s\|\,\|x_s^*-x^*\|+\sigma_{*}^2\right]$.
Moreover, $q_s$ satisfies  $q_s\leq C_q\left[1+\Lambda+\log(N_s+k_0)\right],$
where $C_q$ is a fixed constant satisfying $C_q\geq72(e+2)\left[\tfrac92+\tfrac32\log(2305C_q)\right]$.
\end{proposition}
Observe that in the deterministic case, we have $\sigma_k^2=0$ and $\Delta_k=0$ for all $k\geq0$, and hence $\mathcal V_k=\tilde{\mathcal V}_k=\tilde v_0=0$. In the deterministic case, $\Delta_k=0$ almost surely for every
$k\geq0$. Therefore, all stochastic-error terms vanish, and we
apply \autoref{lem:per-itrate} directly without invoking
\autoref{prop-gradient-norm-square} or
\autoref{prop:inner-product-concentration}. The resulting
deterministic telescoping argument yields the same rates with a
fixed constant $q_{k}$ independent of $k$ and $N$.

\section{Supporting Lemma}
\begin{lemma}\label{lem:bar-q-bound}
Let $p \in (0,1)$ and $0 < \varepsilon < LD_0$ where $D_0 \ge \operatorname{dist}(x_0, X^*).$ It holds that
\begin{align}
    \bar q &\le \tilde c_{\cA}(1 + \log_4 \tfrac{LD_0}{p \varepsilon}). 
\end{align}
\end{lemma}
\begin{proof}
Observe that
\begin{align}\label{eq:hp-bound-S0}
S_0 &= \log_4\left(16(16c_{\cA} + 4) \tfrac{LD_0}{p\varepsilon}\right) = \log_4\left(16(16c_{\cA} + 4) \right) + \log_4 \tfrac{LD_0}{p \varepsilon} \overset{\eqref{eq:hp-state-dependent-def-tildecA}}{\le}\tfrac{\tilde c_{\cA}}{16}(1 + \log_4 \tfrac{LD_0}{p \varepsilon}).
\end{align}
It follows by the definition of $\bar S$,
\begin{equation}\label{eq:bar-S-bound}
\begin{aligned}
\bar S&\le4+\log_4\left(\tfrac{16(16c_{\cA}+4)LD_0}{p\varepsilon}S_0\log_4(4S_0)\right)+\tfrac{1}{2}\log\left(\tilde c_{\cA}\left(1+\log_4\tfrac{LD_0}{p\varepsilon}\right)\right) \\
&\overset{\text{(i)}}{\le} 4S_0 + \tfrac{1}{2} \log_4\left(\tilde c_{\cA}(1 + \log_4 \tfrac{LD_0}{p\varepsilon})\right)\overset{\text{(ii)}}{\le}\tfrac{\tilde c_{\cA}}{2} \left(1 + \log_4 \tfrac{LD_0}{p\varepsilon}\right)
\end{aligned}
\end{equation}
where in (i) follows by $S_0 \ge 4$, and in (ii) we use \eqref{eq:hp-bound-S0} the fact that $\log(x) \le x/4$ for $x \ge 4$. By substituting the bound for $\bar S$ and the definition for $N_{\text{base}}$ into the definition of $\bar q$, 
\begin{equation}\label{eq:q-bar-decomp}
\begin{aligned}
\bar q\le 6 + 2\log_4(\tfrac{\tilde c_{\cA}}{2}(1 + \log_4 \tfrac{LD_0}{p \varepsilon})) + 2\log_4(\tfrac{1}{p}) + 4\log_4\left(4\sqrt{2c_{\cA}}\right) +4\log_4\left(1 +\tfrac{16(16c_{\cA}+4)LD_0 \log \bar S}{\varepsilon}\right). 
\end{aligned}
\end{equation}
To bound the first two terms on the RHS of \eqref{eq:q-bar-decomp}, since $\tilde c_{\cA} \ge 64$, $(1 + \log \tfrac{LD_0}{p \varepsilon}) \ge 1$, and $\log_4(x) \le x/8$ for $x \ge 16$, 
\begin{equation}\label{eq:bar-q-decomp-i}
6 +2\log_4(\tfrac{\tilde c_{\cA}}{2}(1 + \log_4 \tfrac{LD_0}{p \varepsilon})) \le \tfrac{\tilde c_{\cA}}{10}(1 + \log_4\tfrac{LD_0}{p \varepsilon})+\tfrac{\tilde c_{\cA}}{8}(1 + \log_4 \tfrac{LD_0}{p \varepsilon}) \le \tfrac{\tilde c_{\cA}}{4}(1 + \log_4 \tfrac{LD_0}{p \varepsilon}). 
\end{equation}
For the next two terms, since $\varepsilon \le LD_0$ and $4\log_4(4\sqrt{2c_{\cA}}) = 2\log_4(16(2c_{\cA})) \le \tfrac{\tilde c_{\mathcal{A}}}{8}$, 
\begin{equation}\label{eq:bar-q-decomp-ii}
2\log_4\tfrac{1}{p} +4 \log_4(4\sqrt{2c_{\cA}})\le 2\log_4\tfrac{LD_0}{p\varepsilon} +\tfrac{\tilde c_{\cA}}{8} \le \tfrac{\tilde c_{\cA}}{8}(1 + \log_4 \tfrac{LD_0} {p \varepsilon}).
\end{equation}
For the final term, by the fact that $\log_4(1+x) \le1 +\log_4(x)$ for all $x\ge 1$
\begin{equation}\label{eq:bar-q-decomp-iv}
\begin{aligned}
4\log_4\left(1 +\tfrac{16(16c_{\cA}+4)LD_0 \log_4 \bar S}{\varepsilon}\right) &\le 4 + 4\log_4(16(16c_{\cA} +4)) + 4\log_4\left(\tfrac{LD_0}{p \varepsilon}\right) + \log_4(\log_4 \bar S) \\
&\overset{\text{(iii)}}{\le} \tfrac{\tilde c_{\cA}}{16}(1 + \log_4 \tfrac{LD_0}{p \varepsilon}) +\tfrac{\tilde c_{\cA}}{4} + \tfrac{\tilde c_{\cA}}{16}(1 + \log \tfrac{LD_0}{p \varepsilon}) \\
&\le \tfrac{\tilde c_{\cA}}{2} (1 + \log_4 \tfrac{LD_0}{p \varepsilon}). 
\end{aligned}
\end{equation}
where in (iii), we have used that $\log_4 \log_4 \bar S \le \tfrac{\bar S}{16}$ along with \eqref{eq:bar-S-bound}. 
Substituting \eqref{eq:bar-q-decomp-i}, \eqref{eq:bar-q-decomp-ii}, and \eqref{eq:bar-q-decomp-iv} into \eqref{eq:q-bar-decomp} admits the result. 
\end{proof}

\end{appendices}
\end{document}